\documentclass[reqno]{amsart}

\usepackage{silence}
\usepackage{imakeidx}
\makeindex[intoc]

\newcommand{\maass}{\mathrm{Maass}}

\author{Leonard Tomczak}

\usepackage[utf8]{inputenc}

\usepackage[T1]{fontenc}
\usepackage[backend=biber, 
style=alphabetic,
doi=false,
isbn=false,
eprint=false,
giveninits=true,
useprefix=true,
dateabbrev=false,
maxbibnames=6,
maxcitenames=6,
maxalphanames=6,
urldate=long,
]
{biblatex}
\AtEveryBibitem{\clearfield{pages}} 
\AtEveryBibitem{\clearfield{language}}
\AtEveryBibitem{\clearlist{language}}

\DeclareFieldFormat{postnote}{#1} 
\DeclareFieldFormat{multipostnote}{#1}

\DeclareLabelalphaNameTemplate{%
  \namepart[use=true, pre=true, strwidth=1, compound=true]{prefix}%
  \namepart{family}%
}
\DeclareLabelalphaNameTemplate[hyphenated]{%
  \namepart[use=true, pre=true, strwidth=1, compound=true]{prefix}%
  \namepart[strwidth=1, compound=true]{family}%
}
\DeclareSourcemap{%
  \maps[datatype=bibtex]{%
    \map{%
      \step[fieldsource=author, match=\regexp{Lowry-Duda}, final]%
      \step[fieldset=options, fieldvalue={labelalphanametemplatename=hyphenated}]%
    }%
  }%
}

\usepackage{amsmath}
\usepackage{mathtools}
\usepackage{amssymb}

\usepackage{stmaryrd}

\usepackage{bbm}

\usepackage{amsthm}
\usepackage{adjustbox}

\usepackage{float} 

\usepackage{graphicx}
\usepackage{subcaption}

\usepackage{caption}

\usepackage{diagbox}

\usepackage{bm} 

\usepackage[left=3cm,right=3cm]{geometry}

\usepackage{xspace}

\usepackage[table,dvipsnames]{xcolor}

\usepackage{tikz-cd}

\usepackage{spectralsequences}

\usepackage[shortlabels]{enumitem}
\usepackage{crossreftools}

\usepackage{stackrel}

\usepackage{fncylab}
\labelformat{equation}{(#1)}

\usepackage[
colorlinks = false,
pdfpagelabels,
pdfstartview = FitH,
bookmarksopen = true,
bookmarksnumbered = true,
linkcolor = black,
citebordercolor=violet,
urlbordercolor = blue,
plainpages = false,
hypertexnames = false,
citecolor = black,
] {hyperref}

\usepackage{hyperxmp}

\makeatletter
\AtBeginDocument{

  \hypersetup{
    keeppdfinfo,
    pdfauthor={Leonard Tomczak},
    pdftitle={\@title}
  }}
\makeatother

\usepackage[capitalize]{cleveref}

\Crefname{thm}{Theorem}{Theorems}
\Crefname{theorem}{Theorem}{Theorems}
\Crefname{lem}{Lemma}{Lemmas}
\Crefname{lemma}{Lemma}{Lemmas}
\Crefname{prop}{Proposition}{Propositions}
\Crefname{proposition}{Proposition}{Propositions}
\Crefname{cor}{Corollary}{Corollaries}
\Crefname{corollary}{Corollary}{Corollaries}
\Crefname{section}{Section}{Sections}

\usepackage{etoolbox}
\makeatletter
\patchcmd{\chapter}{\if@openright\cleardoublepage\else\clearpage\fi}{}{}{}
\makeatother

\DeclareFontFamily{U}{wncy}{}
\DeclareFontShape{U}{wncy}{m}{n}{<->wncyr10}{}
\DeclareSymbolFont{mcy}{U}{wncy}{m}{n}
\DeclareMathSymbol{\Sha}{\mathord}{mcy}{"58} 

\newcommand{\abs}[1]{\ensuremath{\left\vert#1\right\vert}}
\newcommand{\fabs}[1]{\ensuremath{\vert#1\vert}}
\newcommand{\norm}[1]{\left\lVert#1\right\rVert}

\newcommand{\vertiii}[1]{{\left\vert\kern-0.25ex\left\vert\kern-0.25ex\left\vert #1 
    \right\vert\kern-0.25ex\right\vert\kern-0.25ex\right\vert}}

  \newcommand\legendre[2]{\left(\frac{#1}{#2}\right)}
    \newcommand\legendrebig[2]{\big(\frac{#1}{#2}\big)}

\newcommand{\lquot}{\!\setminus\!}

\usepackage{bbm}

\newcommand{\what}{\widehat}
\newcommand{\wtilde}{\widetilde}

\newcommand{\cusp}{\mathrm{cusp}}

\newcommand{\Z}{\mathbb Z}
\newcommand{\Q}{\mathbb Q}

\newcommand{\R}{\mathbb R}

\newcommand{\C}{\mathbb C}

\renewcommand{\H}{\mathbb H}

\newcommand{\A}{\mathbb A}

\newcommand{\one}{\mathbbm1}

\newcommand{\fp}{\mathfrak p}
\newcommand{\fP}{\mathfrak P}

\newcommand{\FF}{\mathcal F}

\newcommand{\MM}{\mathcal M}

\newcommand{\CHom}{\mathcal{H}\kern -.5pt om}
\newcommand{\CExt}{\mathcal{E}\kern -.5pt xt}

\newcommand{\epsi}{\varepsilon}

\newcommand{\new}{\mathrm{new}}

\renewcommand{\O}{\mathcal O}

\DeclareMathOperator{\tr}{tr}
\DeclareMathOperator{\trd}{trd}
\DeclareMathOperator{\Nrd}{nrd}
\DeclareMathOperator{\nrd}{nrd}

\DeclareMathOperator{\Cls}{Cls}

\DeclareMathOperator{\GL}{GL}
\DeclareMathOperator{\PGL}{PGL}
\DeclareMathOperator{\SU}{SU}
\DeclareMathOperator{\PSU}{PSU}

\DeclareMathOperator{\PSL}{PSL}

\DeclareMathOperator{\PSO}{PSO}

\DeclareMathOperator{\SO}{SO}

\DeclareMathOperator{\Sym}{Sym}

\DeclareMathOperator{\vol}{vol}

\newcommand{\id}{\mathrm{id}}

\DeclareMathOperator{\supp}{supp}

\DeclareMathOperator{\Ind}{Ind}

\renewcommand{\mod}{~\mathrm{mod}\ }

\let\Im\relax
\DeclareMathOperator{\Im}{Im}

\DeclareMathOperator{\sign}{sign}

\renewcommand{\subset}{\subseteq}
\renewcommand{\supset}{\supseteq}

\makeatletter
\@ifundefined{ifGlobalTheoremNumbering}{%
  \newif\ifGlobalTheoremNumbering
  \GlobalTheoremNumberingfalse
}{}
\makeatother

 \makeatletter
 \@ifclassloaded{amsart}{
\renewcommand{\subsection}{%
  \@startsection{subsection}{2}%
    {\z@}%
    {1.25ex \@plus 1ex \@minus .2ex}
    {2ex \@plus .2ex}
    {\normalfont\bfseries}}

\renewcommand{\section}{%
  \@startsection{section}{1}%
    {\z@}%
    {3ex \@plus 1ex \@minus .2ex}
    {3ex \@plus .2ex}
    {\normalfont\scshape\centering}} 
 }
\makeatother

\ifdefstring{\ColoredTheorems}{true}{
    \usepackage{thmtools}
    \usepackage[framemethod=TikZ]{mdframed}
    \usepackage{footnote}

    \declaretheoremstyle[bodyfont = \itshape,
        mdframed={
            skipabove=4pt
            linewidth=2pt,
            rightline=false, topline=false, bottomline=false,
            linecolor=ForestGreen, backgroundcolor=ForestGreen!5,
        }
      ]{theoremgreen}

      \declaretheoremstyle[bodyfont = \itshape,
        mdframed={
            skipabove=4pt
            linewidth=2pt,
            rightline=false, topline=false, bottomline=false,
            linecolor=YellowOrange, backgroundcolor=YellowOrange!5,
        }
      ]{theoremorange}

      \declaretheoremstyle[bodyfont = \itshape,
        mdframed={
            skipabove=6pt
            linewidth=2pt,
            rightline=false, topline=false, bottomline=false,
            linecolor=CornflowerBlue, backgroundcolor=CornflowerBlue!5,
        }
      ]{theoremblue}

      \declaretheoremstyle[
        mdframed={
            skipabove=6pt
            linewidth=2pt,
            rightline=false, topline=false, bottomline=false,
            linecolor=TealBlue, backgroundcolor=TealBlue!5,
        }
      ]{theoremtealblue}
      \declaretheoremstyle[
        mdframed={
            skipabove=6pt
            linewidth=2pt,
            rightline=false, topline=false, bottomline=false,
            linecolor=Fuchsia, backgroundcolor=Fuchsia!5,
        }
      ]{theoremfuchsia}
      \declaretheoremstyle[
        mdframed={
            skipabove=6pt
            linewidth=2pt,
            rightline=false, topline=false, bottomline=false,
            linecolor=GreenYellow, backgroundcolor=GreenYellow!5,
        }
      ]{theoremGreenYellow}

      \declaretheoremstyle[]{empty}

    \declaretheorem[style=theoremgreen, name=Theorem, numberwithin=section]{theorem}
    \declaretheorem[style=theoremgreen, name=Theorem, numbered = no]{theorem*}
    \declaretheorem[style=theoremgreen, name=Corollary, sharenumber=theorem]{corollary}
    \declaretheorem[style=theoremgreen, name=Corollary, numbered = no]{corollary*}
    \declaretheorem[style=theoremgreen, name=Proposition, sharenumber=theorem]{proposition}
    \declaretheorem[style=theoremgreen, name=Proposition, numbered = no]{proposition*}
    \declaretheorem[style=theoremblue, name=Lemma, sharenumber=theorem]{lemma}
    \declaretheorem[style=theoremblue, name=Lemma, sharenumber=theorem, numbered = no]{lemma*}

    \declaretheorem[style=theoremorange, name=Definition, numbered = no]{definition*}
    \declaretheorem[style = theoremfuchsia, name=Remark, numbered = no]{remark*}
    \declaretheorem[style = theoremfuchsia, name=Remark, numbered = no]{remark}
    \declaretheorem[style = theoremfuchsia, name=Remarks, numbered = no]{remarks*}
    
    \declaretheorem[style = theoremtealblue, name=Example, numbered = no]{example*}
    \declaretheorem[style = theoremtealblue, name=Example, numbered = no]{example}
    \declaretheorem[style = theoremtealblue, name=Examples, numbered = no]{examples*}
    \declaretheorem[style = theoremtealblue, name=Examples]{examples}
    \declaretheorem[style = empty, name=Problem]{problem}
    \declaretheorem[style = empty, name=Problem, numbered = no]{problem*}
    \declaretheorem[style = empty, name=Solution, numbered = no]{solution*}
    \declaretheorem[style = theoremGreenYellow, name=Conjecture]{conjecture}
    \declaretheorem[style = theoremGreenYellow, name=Conjecture, numbered = no]{conjecture*}
    
    \declaretheorem[style = theoremfuchsia, name=Fact, numbered = no]{fact*}
}{
  \usepackage{thmtools}

\declaretheoremstyle[
    spaceabove=5pt plus 2pt minus 2pt,
    spacebelow=5pt plus 2pt minus 2pt,
    bodyfont=\itshape,
    headfont=\bfseries,
    headpunct={.},
    postheadspace=.5em,
]{theoremstyle}

\declaretheoremstyle[
    spaceabove=5pt plus 2pt minus 2pt,
    spacebelow=5pt plus 2pt minus 2pt,
    bodyfont=\normalfont,
    headfont=\bfseries,
    headpunct={.},
    postheadspace=.5em,
]{defstyle}

\ifGlobalTheoremNumbering
  \declaretheorem[
    style=theoremstyle,
    refname={Theorem,Theorems},
    Refname={Theorem,Theorems},
    name=Theorem
  ]{theorem}
\else
  \declaretheorem[
    style=theoremstyle,
    name=Theorem,
    refname={Theorem,Theorems},
    Refname={Theorem,Theorems},
    numberwithin=section
  ]{theorem}
\fi

\declaretheorem[
    style=theoremstyle,
    name=Theorem,
    numbered=no,
]{theorem*}

\declaretheorem[
    style=theoremstyle,
    name=Corollary,
    sharenumber=theorem,
    refname={Corollary,Corollaries},
    Refname={Corollary,Corollaries},
]{corollary}

\declaretheorem[
    style=theoremstyle,
    name=Corollary,
    numbered=no,
]{corollary*}

\declaretheorem[
    style=theoremstyle,
    name=Proposition,
    sharenumber=theorem,
    refname={Proposition,Propositions},
    Refname={Proposition,Propositions},
]{proposition}

\declaretheorem[
    style=theoremstyle,
    name=Proposition,
    numbered=no,
]{proposition*}

\declaretheorem[
    style=theoremstyle,
    name=Definition,
    sharenumber=theorem,
    refname={Definition,Definitions},
    Refname={Definition,Definitions},
]{definition}

\declaretheorem[
    style=theoremstyle,
    name=Definition,
    numbered=no,
]{definition*}

\declaretheorem[
    style=theoremstyle,
    name=Lemma,
    sharenumber=theorem,
    refname={Lemma,Lemmas},
    Refname={Lemma,Lemmas},
]{lemma}

\declaretheorem[
    style=theoremstyle,
    name=Lemma,
    numbered=no,
]{lemma*}

\declaretheorem[
    style=theoremstyle,
    name={Satz + Def.},
    sharenumber=theorem,
    refname={Satz + Def.,Satz + Def.},
    Refname={Satz + Def.,Satz + Def.},
]{theoremdef}

\declaretheorem[
    style=defstyle,
    name=Remark,
    refname={Remark,Remarks},
    Refname={Remark,Remarks},
]{remark}

\declaretheorem[
    style=defstyle,
    name=Remark,
    numbered=no,
]{remark*}

\declaretheorem[
    style=defstyle,
    name=Remarks,
    numbered=no,
]{remarks*}

\declaretheorem[
    style=defstyle,
    name=Example,
    refname={Example,Examples},
    Refname={Example,Examples},
]{example}

\declaretheorem[
    style=defstyle,
    name=Example,
    numbered=no,
]{example*}

\declaretheorem[
    style=defstyle,
    name=Examples,
    refname={Examples,Examples},
    Refname={Examples,Examples},
]{examples}

\declaretheorem[
    style=defstyle,
    name=Examples,
    numbered=no,
]{examples*}

\declaretheorem[
    style=defstyle,
    name=Conjecture,
    sharenumber=theorem,
    refname={Conjecture,Conjectures},
    Refname={Conjecture,Conjectures},
]{conjecture}

\declaretheorem[
    style=defstyle,
    name=Conjecture,
    numbered=no,
]{conjecture*}

\declaretheorem[
    style=defstyle,
    name=Problem,
    sharenumber=theorem,
    refname={Problem,Problems},
    Refname={Problem,Problems},
]{problem}
\declaretheorem[
    style=defstyle,
    name=Problem,
    numbered = no
]{problem*}

\declaretheorem[
    style=defstyle,
    name=Solution,
    numbered = no
]{solution*}
}

 \newboolean{drawpics} 
 
\setboolean{drawpics}{true} 

\allowdisplaybreaks

\DeclareFontFamily{U}{mathx}{\hyphenchar\font45}
\DeclareFontShape{U}{mathx}{m}{n}{
      <5> <6> <7> <8> <9> <10>
      <10.95> <12> <14.4> <17.28> <20.74> <24.88>
      mathx10
      }{}
\DeclareSymbolFont{mathx}{U}{mathx}{m}{n}
\DeclareFontSubstitution{U}{mathx}{m}{n}
\DeclareMathAccent{\widecheck}{0}{mathx}{"71}
\DeclareMathAccent{\wideparen}{0}{mathx}{"75}

\makeatletter

\newcommand*\wthelper[2]{%
        \hbox{\dimen@\accentfontxheight#1%
                \accentfontxheight{#1}1.2\dimen@
                $\m@th#1\widetilde{#2}$%
                \accentfontxheight#1\dimen@
        }%
}

\newcommand*\whhelper[2]{%
        \hbox{\dimen@\accentfontxheight#1%
                \accentfontxheight{#1}1.1\dimen@
                $\m@th#1\widehat{#2}$%
                \accentfontxheight#1\dimen@
        }%
}

\newcommand*\accentfontxheight[1]{%
        \fontdimen5\ifx#1\displaystyle
                \textfont
        \else\ifx#1\textstyle
                \textfont
        \else\ifx#1\scriptstyle
                \scriptfont
        \else
                \scriptscriptfont
        \fi\fi\fi3
}
\makeatother

\usepackage{scalerel,stackengine}
\stackMath
\newcommand\reallywidehat[1]{%
\savestack{\tmpbox}{\stretchto{%
  \scaleto{%
    \scalerel*[\widthof{\ensuremath{#1}}]{\kern-.6pt\bigwedge\kern-.6pt}%
    {\rule[-\textheight/2]{1ex}{\textheight}}
  }{\textheight}%
}{0.5ex}}%
\stackon[1pt]{#1}{\tmpbox}%
}

\makeatletter
\newcommand\mathcircled[1]{%
  \mathpalette\@mathcircled{#1}%
}
\newcommand\@mathcircled[2]{%
  \tikz[baseline=(math.base)] \node[draw,circle,inner sep=1pt] (math) {$\m@th#1#2$};%
}
\newcommand\mathcircledsmall[1]{%
  \mathpalette\@mathcircled{#1}%
}
\newcommand\@mathcircledsmall[2]{%
  \tikz[overlay,baseline=(math.base)] \node[draw,circle,inner sep=.5pt] (math) {$\m@th#1#2$};%
}
\makeatother

\title[Murmurations in the Depth Aspect]{Murmurations in the Depth Aspect for Maass and Modular Forms}

\usepackage{longtable}
\usepackage{booktabs}

\makeatletter
\renewcommand{\l@section}{\@tocline{1}{.5\baselineskip plus 1pt}{0pt}{2.5em}{}}
\renewcommand{\l@subsection}{\@tocline{2}{.5\baselineskip plus 1pt}{2.5em}{3.5em}{}}
\makeatother

\numberwithin{theorem}{section}
\begin{document}
\begin{abstract}
    We study murmurations in the depth aspect for holomorphic cusp forms of conductor $\ell^{2a}$ and fixed weight, where $\ell$ is an odd prime. For both $\GL_2$ and the definite quaternion algebra ramified at $\{\infty,\ell\}$, we determine the murmuration density as $a\to\infty$ with $\ell$ fixed. The resulting density agrees with the one previously obtained for odd conductor exponents, and hence gives a uniform density for cusp forms of conductor $\ell^n$ as $n\to\infty$. We also consider the case of Maass forms of conductor $\ell^{n}$. Finally, we compute the murmuration density in conductor $\ell^n$ as $\ell\to\infty$ with $n\geq3$ fixed.
\end{abstract}
\address{Department of Mathematics, Evans Hall, University of California, Berkeley, CA 94720, USA}
\email{leonard.tomczak@berkeley.edu}
\maketitle
\tableofcontents
\section{Introduction and Results}
Murmuration is a recently discovered phenomenon in the theory of automorphic representations. It refers to oscillatory patterns in the averages of Dirichlet coefficients of certain families of $L$-functions when grouped by root number. Although it was originally observed in the case of families of elliptic curves \cite{HeLeeOliPozMEC}, similar phenomena have subsequently been observed in a variety of other families of automorphic representations. We refer to \cite{SarLetter} and \cite{LDMurmsTF} for more background and examples. 
One of the families studied in this paper is the family of Hecke cusp forms of level $\ell^n$ and fixed weight, where $\ell$ is an odd prime. The case of odd $n$ has already been treated in \cite{BKLMYMurms}. We show that in the case of $\GL_2$ modular forms of level $\ell^{2a}$ we obtain the same murmuration density in the limit $a\to\infty$ as in the case of odd exponents, so in particular the whole family of level $\ell^n$ cusp forms has a murmuration density as $n\to\infty$. Additionally, we also show that if $D$ is the quaternion algebra over $\Q$ ramified at $\{\infty,\ell\}$, then cusp forms for $D$ of conductor $\ell^{2a}$, which via the Jacquet--Langlands correspondence are identified with a subfamily of the $\GL_2$ forms, have the same murmuration density. We also consider the non-holomorphic counterpart of this family, weight $0$ Maass forms of prime-power conductor. Previously murmurations of Maass forms have only been established in the spectral aspect at conductor $1$ in \cite{BLLDSHZMurms}. Finally, in all three cases we study the family in which the conductor exponent is fixed and $\ell\to\infty$, which is similar to \cite{ZubM}, but with a different averaging. Our results will be conditional on the Generalized Riemann Hypothesis (GRH) for Dirichlet $L$-functions.

In order to state our results, let us introduce some notation. Throughout, $\ell$ denotes an odd prime, $k$ an even positive integer, and $E$ a compact interval of positive real numbers. For $t\in\Z$ define \begin{align*}
    A_{t}=\prod_{\substack{p\nmid t}}\left(\frac{p^2-p-1}{p(p-1)}\right),
\end{align*}
where the product runs over primes $p$. We let \begin{align*}
    \MM_{\ell, k}(v) &= (-1)^{\frac k2+1}\frac{2\pi}{(k-1)(1-\ell^{-1})}\sum_{\substack{r\in\Z\\ \abs r < 2\ell\sqrt v}}(\one_{\ell\mid r}-\ell^{-1})A_{\ell r}\sqrt{v-\frac {r^2}{4\ell^2}}U_{k-2}\left(\frac r{2\ell\sqrt v}\right),\\
    M_{E,\ell, k} &= \frac1{\abs E}\int_E\MM_{\ell, k}(v)\mathrm dv.
\end{align*}
Here $U_{k-2}$ denotes the Chebyshev polynomial of the second kind. Note that $M_{E,\ell, k}$ is the same number that appears in the limit of \cite[Theorem 1.1]{BKLMYMurms}.\footnote{In their theorem the Euler product excludes the prime $2$. We believe the $p=2$ factor should be included unless $2\mid r$ as follows from the computations in \cite[Lemma 4.3]{BBLLDMurms}. This is also consistent with the smoothed density limit analogue in \cite[Theorem 5]{ZubM} which is expected to be $\frac12$ by the Katz-Sarnak 1-level density predictions.} Let $H_k^\new(N)$ denote a normalized Hecke eigenbasis for $S_k^\new(N):=S_k^\new(\Gamma_0(N))$. For a cusp form (holomorphic or not) $\phi$ with trivial central character over $\GL_2$ or a quaternion algebra, the Hecke eigenvalue $\lambda_P(\phi)$ at a prime $P$ is normalized so that the Ramanujan conjecture reads $\abs{\lambda_P(\phi)}\leq2$. We denote by $\epsi(\phi)$ its global root number, i.e.\ $\epsi(\phi) = \epsi(\frac12, \phi)=\pm1$. For classical newforms we will use the terms level and conductor interchangeably.

Next let $D=D_\ell$ be the unique quaternion algebra over $\Q$ ramified at $\{\infty,\ell\}$. Let $\FF_{D, k}(N)$ denote the set of infinite-dimensional automorphic representations $\pi$ of $D^\times$ with trivial central character, conductor $N$, and $\pi_\infty\cong \Sym^{k-2}(\C^2)$. Note that we will define the conductor of an irreducible representation of $D_\ell^\times$ later in \cref{section:tfs}.

Let $H_{\maass}^\new(N)$ be the set of normalized weight $0$ Hecke Maass eigenforms of conductor $N$. If $\phi\in H_{\maass}^\new(N)$, let $t_\phi$ denote the spectral parameter (unique up to sign), so that the Laplace eigenvalue of $\phi$ is $\frac14+t_\phi^2$. Let $h$ be an even function that is holomorphic in $\{z\in \C\mid \abs{\Im z}<A\}$ and that satisfies $\abs {h(z)}=O((1+\abs z)^{-B})$ for some $A>1$, $B >2$. $h=h_f$ is the Selberg transform (see \cref{section:maass_props}) of an $\SO(2)$-biinvariant function $f$ on $\PGL_2(\R)^+$. Let $Q_f\in C([0,\infty))$ denote the Abel transform of $f$, defined in \cref{section:maass_props}. We assume that $f(1)\ne0$. Define \begin{align*}
    \MM_{\ell, f}(v) &= \frac{\pi}{2f(1)(1-\ell^{-1})}\sum_{\substack{r\in\Z}}(\one_{\ell\mid r}-\ell^{-1})A_{\ell r}\sqrt v Q_f\left(\frac{r^2}{\ell^2 v}\right),\\
    M_{E,\ell, f} &= \frac1{\abs E}\int_E\MM_{\ell, f}(v)\mathrm dv.
\end{align*}

Then our main result for fixed $\ell$ is:
\begin{theorem}\label{theorem:n_limit}
    Fix an odd prime $\ell$, a compact interval $E\subset\R_{>0}$, $\epsi\in (0, \frac15)$, and the remaining notation as above. We assume GRH for Dirichlet $L$-functions. The following asymptotics hold as $a\to\infty$ in (i) and (ii), and as $n\to\infty$ in (iii): \begin{enumerate}[(i)]
        \item For $a\geq2$ and holomorphic modular forms: \begin{align*}
            \frac{\displaystyle\sum_{\substack{P/\ell^{2a}\in E\\ \text{$P$ prime}}}\log P\sum_{\phi\in H^\new_k(\ell^{2a})}\epsi(\phi)P^{1/2}\lambda_P(\phi)}{\displaystyle\sum_{\substack{P/\ell^{2a}\in E\\ \text{$P$ prime}}}\log P \sum_{\phi\in H^\new_k(\ell^{2a})}1}
            = M_{E,\ell,k} + O_{E, k, \ell, \epsi}(\ell^{(-\frac15+\epsi)a}).
             \end{align*}
        \item For $a\geq2$ and automorphic forms on the quaternion algebra $D$:\begin{align*}
                &\frac{\displaystyle\sum_{\substack{P/\ell^{2a}\in E\\ \text{$P$ prime}}}\log P\sum_{\pi\in\FF_{D, k}(\ell^{2a})}\epsi(\pi)P^{1/2}\lambda_P(\pi)}{\displaystyle\sum_{\substack{P/\ell^{2a}\in E\\ \text{$P$ prime}}}\log P \sum_{\pi\in\FF_{D, k}(\ell^{2a})}1} = M_{E,\ell,k} + O_{E, k,\ell, \epsi}(\ell^{(-\frac15+\epsi)a}).
            \end{align*}
        \item For $n\geq3$ and Maass forms:
            \begin{align*}
                &\frac{\displaystyle\sum_{\substack{P/\ell^{n}\in E\\ \text{$P$ prime}}}\log P\sum_{\phi\in H_\maass^\new(\ell^{n})}h_f(t_\phi)\epsi(\phi)P^{1/2}\lambda_P(\phi)}{\displaystyle\sum_{\substack{P/\ell^{n}\in E\\ \text{$P$ prime}}}\log P \sum_{\phi\in H_\maass^\new(\ell^{n})}h_f(t_\phi)} = M_{E,\ell,f} + O_{E, f, \ell, \epsi}(\ell^{(-\frac1{5}+\epsi)\lfloor n/2\rfloor}).
            \end{align*}
    \end{enumerate}
\end{theorem}

\begin{remarks*}~
    \begin{enumerate}
        \item In \cite{BKLMYMurms} the GRH assumption is removed from their result at the cost of obtaining a weaker error bound $O_C(a^{-C})$ by invoking a specialization of the Bombieri--Vinogradov theorem to square moduli due to Baker \cite{BakPAPSM}. We do not attempt to remove GRH here. The reason is that, in the proof of \cite[Lemma 3.3]{BKLMYMurms}, we do not see how to justify the treatment of the tail of the $L(1,\psi_D)$-sum using only the unconditional bound in \cite[Lemma 4.3]{BKLMYMurms}. The argument appears to require the GLH-strength tail estimate from \cite[Lemma 4.4]{BBLLDMurms}. If this can be resolved, the same strategy can very likely be applied in our case. We need moduli of the form $\ell m^2$, but since $\ell$ is fixed here, one can apply \cite{BakPAPSM} to $(\ell m)^2$.
        \item Combining \cref{theorem:n_limit} (i) with the result from \cite{BKLMYMurms} for odd exponents, we obtain the murmuration for conductors including all powers of $\ell$.
        \item By the Jacquet--Langlands correspondence, automorphic representations $\pi\in\FF_{D, k}(\ell^n)$ are in bijection with the subset of Hecke cusp forms $f\in H^\new_k(\ell^n)$ that have a discrete series component at $\ell$, and moreover this bijection preserves the Hecke action. In particular if $n\geq3$ is odd, then any Hecke cusp form $f\in H^\new_k(\ell^n)$ has a discrete series component at $\ell$ since the ramified principal series representations of $\GL_2(\Q_\ell)$ with trivial central character have even conductor exponent. Thus, in this case\[ 
            \frac{\displaystyle\sum_{\substack{P/\ell^{2a+1}\in E\\ \text{$P$ prime}}}\log P\sum_{\pi\in\FF_{D, k}(\ell^{2a+1})}\epsi(\pi)P^{1/2}\lambda_P(\pi)}{\displaystyle\sum_{\substack{P/\ell^{2a+1}\in E\\ \text{$P$ prime}}}\log P \sum_{\pi\in\FF_{D, k}(\ell^{2a+1})}1}
        \]
        is the same as the corresponding density for $H^\new_k(\ell^{2a+1})$. This will also follow from our explicit formulas for the traces later on in \cref{section:tf_proofs}. Hence, in the quaternion algebra case as well, we obtain murmurations for conductors including all powers of $\ell$.
        \item It is interesting to note that the explicit trace formulas used to prove these results, including the one in \cite{BKLMYMurms}, are different. Most notably, we sum over elliptic conjugacy classes of determinant $P$ in the even conductor exponent case and determinant $\ell P$ in the odd case. Yet in all the cases - quaternionic forms, classical modular forms, even vs.\ odd conductor exponent - we obtain the same limit $M_{E, \ell,k}$. It would be desirable to have a more uniform proof of this.
    \end{enumerate}
\end{remarks*}
\begin{figure}[H]
    \includegraphics[scale = 0.8]{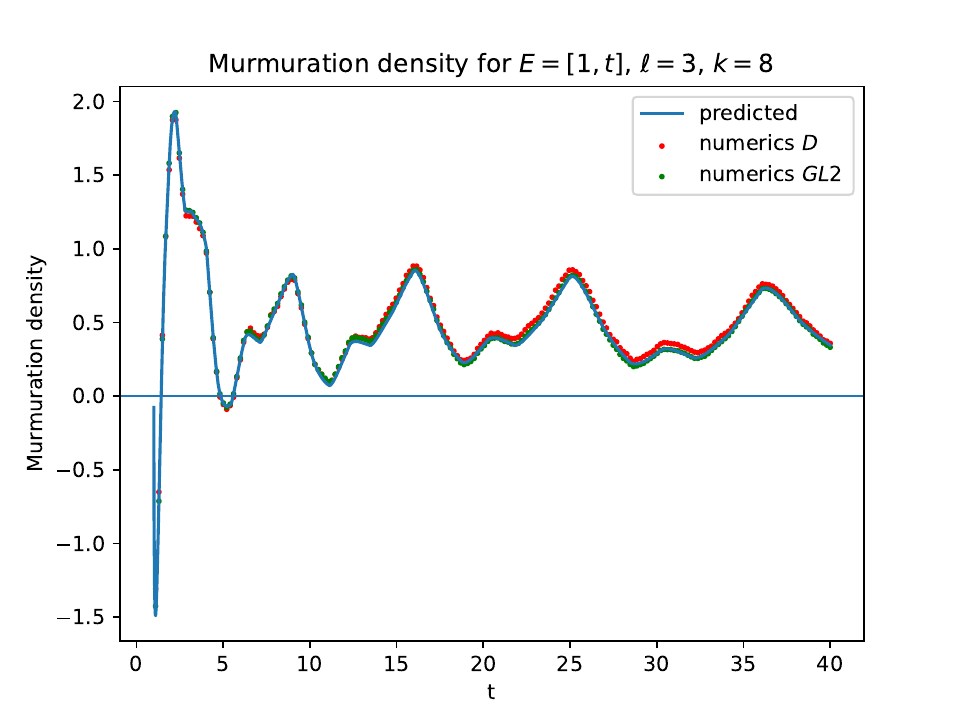}
    \caption[Murmuration plot for \cref{theorem:n_limit} \textup{(i)}, \textup{(ii)}]{The blue line plots the function $t\mapsto M_{[1,t], \ell, k}=\frac1{t-1}\int_1^t\MM_{\ell, k}(v)\mathrm dv$ for $\ell=3,k=8$. The green (resp.\ red) dots show direct numerical evaluations of the quotient on the left-hand side of \cref{theorem:n_limit} \textup{(i)} (resp.\ \textup{(ii)}), with $a = 4$.\footnotemark }
\end{figure}
\footnotetext{The plots in \cite{BKLMYMurms} appear to differ from ours, even after accounting for the contribution of the local factor at $p=2$. In particular our graph does not exhibit the ``fuzziness'' observed there. We were not able to reproduce the graphs from \cite{BKLMYMurms}.}

\begin{figure}[H]
    \centering
    \includegraphics[scale = 0.8]{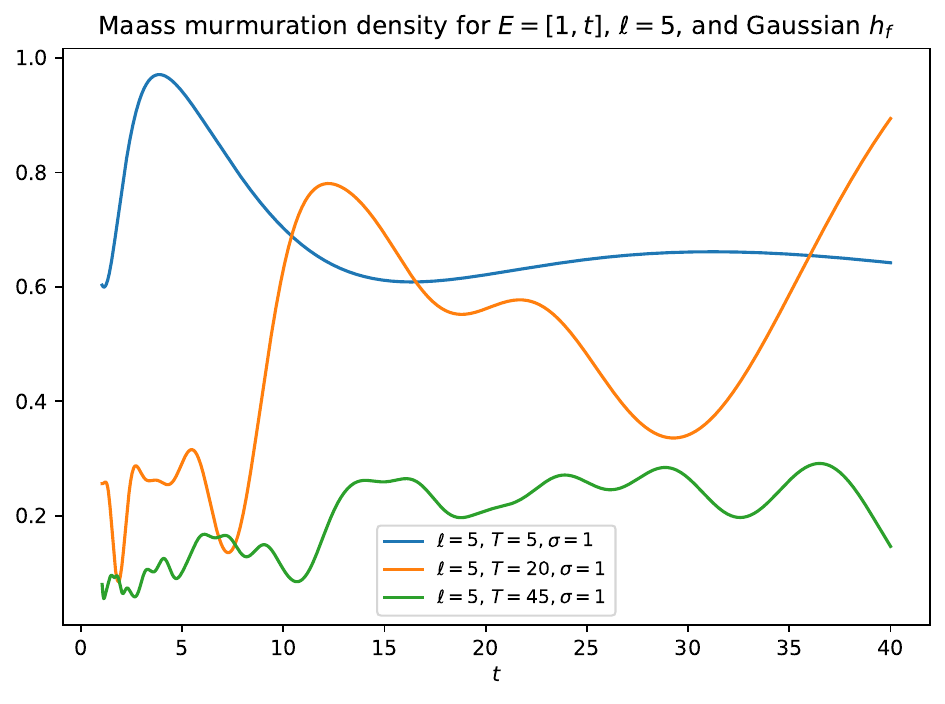}
    \caption[Murmuration plot for \cref{theorem:n_limit} (iii)]{Plot of $t\mapsto M_{[1,t],\ell,f}$ for $\ell = 5$, where $h_f(t) = G\left(\frac{t-T}{\sigma}\right)+G\left(\frac{t+T}{\sigma}\right)$ with $G(x) = e^{-x^2/2}$, for various values of $T$.}
\end{figure}
The proofs of these results will follow along the same lines as in \cite{BKLMYMurms}. We first derive explicit formulas for the relevant trace sums in terms of class numbers and then use \cite[Lemma 4.5]{BBLLDMurms} to estimate the prime sums. In the case of holomorphic forms on $\GL_2$ the trace formula follows quickly from the Yamauchi--Skoruppa--Zagier variant \cite{SkoZagJFMF} of the Eichler--Selberg trace formula as in \cite[Proposition 2.1]{BKLMYMurms}. In the case of quaternion algebras we derive the trace formula directly using the simple compact quotient trace formula by choosing appropriate test functions, and similarly for Maass forms in which case we use the adelic trace formula for $\GL_2$.

During the proofs we will keep track of the $\ell$-dependence in the error term, which also allows us to consider the limit in $\ell$ with $a\geq2$ fixed. This is closer to the situation in \cite{ZubM}, but with the substantial difference that here there is no averaging over the conductor, but instead over the prime $P$. Let \begin{align*}
    \MM_{\infty, k}(v)&= (-1)^{\frac k2+1}\frac{2\pi}{k-1}\sum_{\substack{s\in\Z\\ \abs s < 2\sqrt v}}A_{s}\sqrt{v-\frac{s^2}4}U_{k-2}\left(\frac s{2\sqrt v}\right) - 12\one_{k=2}v,\\
    M_{E, \infty, k} &= \frac1{\abs E}\int_E\MM_{\infty, k}(v)\mathrm dv.
\end{align*}
For the Maass form case, define $C_f = \frac{6}{\pi^2}\int_\R Q_f(x^2)\mathrm dx$ and \begin{align*}
    \MM_{\infty, f}(v) &= \frac{\pi}{2f(1)}\left(\sum_{s\in\Z}A_s\sqrt vQ_f\left(\frac{s^2}{v}\right)-C_fv\right),\\
    M_{E,\infty,f}&=\frac{1}{\abs E}\int_E \MM_{\infty, f}(v)\mathrm dv.
\end{align*}
Then we have:
\begin{theorem}\label{theorem:ell_limit}
    We assume GRH for Dirichlet $L$-functions. Then as $\ell\to\infty$ through odd primes we have the following limits.\begin{enumerate}[(i)]
        \item Fix $a\geq2$. For holomorphic modular forms:  \begin{align*}
            \frac{\displaystyle\sum_{\substack{P/\ell^{2a}\in E\\ \text{$P$ prime}}}\log P\sum_{\phi\in H^\new_k(\ell^{2a})}\epsi(\phi)P^{1/2}\lambda_P(\phi)}{\displaystyle\sum_{\substack{P/\ell^{2a}\in E\\ \text{$P$ prime}}}\log P \sum_{\phi\in H^\new_k(\ell^{2a})}1}
            = M_{E,\infty,k} + O_{E, k, a, \epsi}(\ell^{-a/5+\epsi}+\ell^{-1}).
        \end{align*}
        \item Fix $a\geq2$. For automorphic forms on the quaternion algebra $D_\ell$:\begin{align*}
            \frac{\displaystyle\sum_{\substack{P/\ell^{2a}\in E\\ \text{$P$ prime}}}\log P\sum_{\pi\in\FF_{D_\ell, k}(\ell^{2a})}\epsi(\pi)P^{1/2}\lambda_P(\pi)}{\displaystyle\sum_{\substack{P/\ell^{2a}\in E\\ \text{$P$ prime}}}\log P \sum_{\pi\in\FF_{D_\ell, k}(\ell^{2a})}1} 
            = M_{E,\infty,k} + O_{E, k, a, \epsi}(\ell^{-\frac{1}{5}(a-1)+\epsi}+\ell^{-1}).
        \end{align*}
        \item Fix $n\geq3$. For Maass forms:
            \begin{align*}
                &\frac{\displaystyle\sum_{\substack{P/\ell^{n}\in E\\ \text{$P$ prime}}}\log P\sum_{\phi\in H_\maass^\new(\ell^{n})}h_f(t_\phi)\epsi(\phi)P^{1/2}\lambda_P(\phi)}{\displaystyle\sum_{\substack{P/\ell^{n}\in E\\ \text{$P$ prime}}}\log P \sum_{\phi\in H_\maass^\new(\ell^{n})}h_f(t_\phi)} = M_{E,\infty,f} + O_{E, f, n, \epsi}(\ell^{-n/10+\epsi}+\ell^{-1}).
            \end{align*}
    \end{enumerate}
\end{theorem} 

\begin{remarks*}~
    \begin{enumerate}
        \item Although it may seem unnatural at first to consider representations of varying groups in \cref{theorem:ell_limit} (ii), the Jacquet--Langlands correspondence transfers the automorphic representations of $D_\ell^\times$ to automorphic representations of $\GL_2$. Thus, this theorem can be viewed as a statement about a harmonic family, in the sense of \cite{SarShiTemFamily}, on $\GL_2$.
        \item The same results in the case of holomorphic forms should hold for odd exponents: the argument essentially consists of taking the limit of $M_{E, \ell, k}$ as $\ell\to\infty$. One just has to keep track of the $\ell$-dependence of the error term in \cite{BKLMYMurms}.
        \item It is interesting to note that the murmuration densities in \cref{theorem:ell_limit} do not depend on $a$ (or $n$ in the Maass case).
        \item Our murmuration density $\MM_{\infty,k}$ has the very same shape as $\MM_k$ computed in \cite{ZubM}. It differs only in the coefficients. Interestingly, the $\one_{k=2}$ term in \cite{ZubM} comes from the residual spectrum contribution to the trace formula, whereas in our case this contribution cancels during the inclusion-exclusion process in the proof of \cref{proposition:GL2_TF}, and the corresponding term instead comes from the orthogonality relations for Chebyshev polynomials. In work in progress \cite{TomMMFSC}, the author proves an analogue of Zubrilina's result for Maass forms of varying squarefree conductor. In this setting as well, the density has the same shape as $\MM_{\infty,f}$, differing only in the coefficients. Moreover, the residual spectrum contribution behaves analogously.
    \end{enumerate}
\end{remarks*}

\begin{figure}[H]
    \centering
    \includegraphics[scale = 0.8]{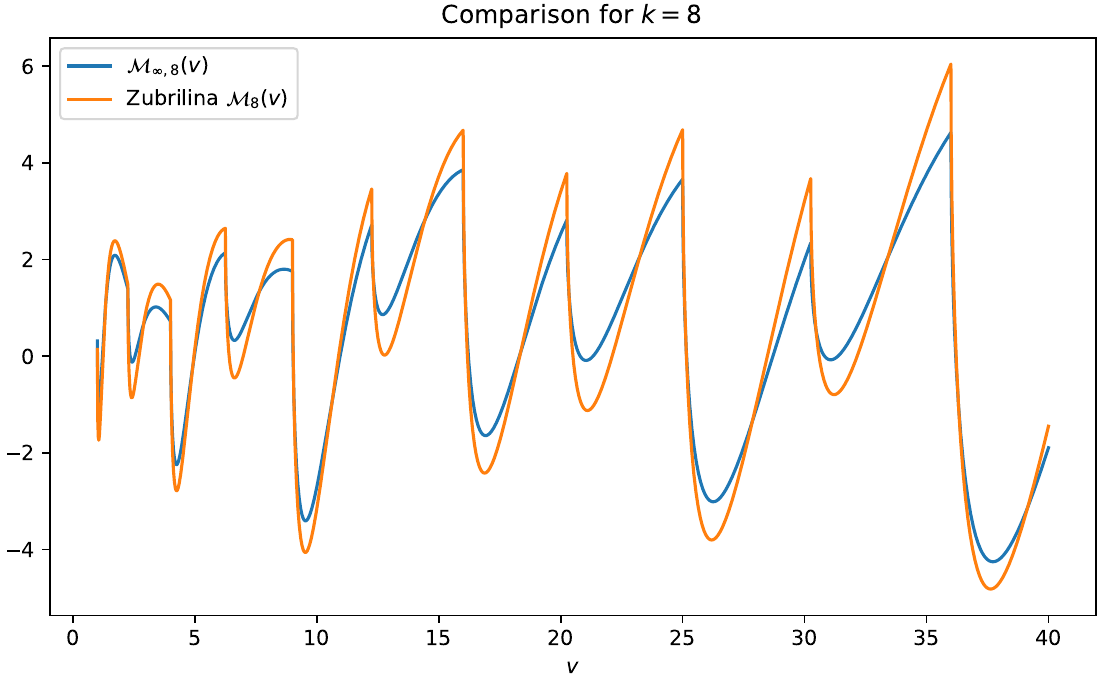}
    \caption{Comparison of our $\MM_{\infty, 8}$ with $\MM_8$ in \cite{ZubM}.}
\end{figure}

\section{Acknowledgements}
I thank my advisor Sug Woo Shin for suggesting this problem to me and for many helpful conversations.

\section{Trace Formula Results}\label{section:tfs}
In order to prove our results we need explicit formulas for both the numerator and denominator. We will record them here for use in the proof of the main theorems in the next section. Their proofs are deferred to \cref{section:tf_proofs}. Let $\ell$ be an odd prime and $k\geq2$ an even integer. $\chi_\ell:=\legendrebig{-}{\ell}$ denotes the quadratic Dirichlet character mod $\ell$. If $D$ is a quadratic discriminant and $D=df^2$ where $d$ is a fundamental discriminant, then let $\psi_D(n) = \legendre{d}{n/\gcd(n,f)}$ for $n\in\Z$. Although this is not a character if $f>1$, we can still define $L(1,\psi_D)=\sum_{n=1}^\infty\frac{\psi_D(n)}n$. We let $H(D)$ be the Hurwitz class number, which is $H(D) = \sum_{m\mid f}\frac{h(\O_{dm^2})}{\fabs{\O_{dm^2}^\times/\pm1}}$, the sum of all weighted class numbers of quadratic orders containing the order of discriminant $D$.
\begin{proposition}\label{proposition:GL2_TF}
For $a\geq2$ and any prime $P\ne\ell$ we have:\begin{align*}
    &\sum_{\substack{f\in H^\new_k(\ell^{2a})}}\epsi(f)\sqrt P\lambda_P(f) \\
    &= (-1)^{k/2+1}\left(\frac{\ell-\chi_\ell(-P)}{2\ell}\sum_{t\in\Z, t^2<4P}U_{k-2}\left(\frac{t}{2\sqrt{P}}\right)c_{\ell^a}(t)H(t^2-4P) + \frac{\ell-1}\ell c_{\ell^a}(P+1)P^{1-k/2}\right)\\
    &=(-1)^{k/2+1}\Bigg(\frac{\ell-\chi_\ell(-P)}{2\pi\ell}\sum_{t\in\Z, t^2<4P}\sqrt{4P-t^2}U_{k-2}\left(\frac{t}{2\sqrt{P}}\right)c_{\ell^a}(t)L(1,\psi_{t^2-4P})\\&\hspace{10em}+ \frac{\ell-1}\ell c_{\ell^a}(P+1)P^{1-k/2}\Bigg).
\end{align*}
\end{proposition}
Here \[
    c_{\ell^a}(t)=\begin{cases}
        \ell^{a-1}(\ell-1)&\text{if $v_\ell(t)> a-1$},\\
        -\ell^{a-1}&\text{if $v_\ell(t)= a-1$},\\
        0&\text{if $v_\ell(t)< a-1$}.
    \end{cases}
\]
We obtain a similar expression for the quaternion algebra case. First, some notation. Let $D$ be a division quaternion algebra over a finite extension $F$ of $\Q_p$. Let $U_D^0 = \O_D^\times$ and $U_D^n = 1+\fP_D^n$ for $n>0$. Then if $\pi$ is an irreducible admissible representation of $D^\times$, its \textit{conductor exponent} $n(\pi)$ is defined to be the smallest $n\geq1$ such that $\pi$ is trivial on $U_D^{n-1}$. Note that in the language of \cite{BusHenLL}, $n(\pi)-2$ is the level of $\pi$. The conductor of $\pi$ is $\fp_F^n$. It agrees with the conductor of its Jacquet--Langlands transfer.

Now let $D$ be the unique quaternion algebra over $\Q$ that is ramified at $\{\infty, \ell\}$. If $\pi$ is an automorphic representation of $D^\times$, we define its conductor to be the product of the conductors of the local representations $\pi_p$. For $p\ne\ell$ we have $D_p^\times\cong\GL_2(\Q_p)$ and the conductor is defined to be the smallest $p^{n}$ such that $\pi_p$ has a non-zero $K_0(p^n)$-fixed vector. For $p=\ell$, the conductor is defined as above.
\begin{proposition}\label{proposition:quat_TF}
For $a\geq2$ and any prime $P\ne\ell$ we have:\begin{align*}
    &\sum_{\pi\in\FF_{D, k}(\ell^{2a})}\epsi(\pi)\sqrt P\lambda_P(\pi) \\
    &= \frac{(-1)^{k/2+1}}4(1+\ell^{-1})(1-\chi_\ell(-P))\sum_{t\in\Z, t^2<4P}U_{k-2}\left(\frac{t}{2\sqrt{P}}\right)c_{\ell^a}(t)H(t^2-4P)\\
    &=\frac{(-1)^{k/2+1}}{4\pi}(1+\ell^{-1})(1-\chi_\ell(-P))\sum_{t\in\Z, t^2<4P}\sqrt{4P-t^2}U_{k-2}\left(\frac{t}{2\sqrt{P}}\right)c_{\ell^a}(t)L(1,\psi_{t^2-4P}).
\end{align*}
\end{proposition}

We will also need to count the cusp forms. In the modular forms case this follows immediately from formulas in \cite{MarDims}:
\begin{proposition}\label{proposition:GL2_count}
    For $a\geq2$ we have: \begin{align*}
        \sum_{\substack{f\in H^\new_k(\ell^{2a})}}1 &= \frac{k-1}{12}\ell^{2a-3}(\ell-1)^2(\ell+1)-\frac12\ell^{a-2}(\ell-1)^2,\\
        \sum_{\substack{f\in H^\new_k(\ell^{2a+1})}}1 &= \frac{k-1}{12}\ell^{2a-2}(\ell-1)^2(\ell+1).
    \end{align*}
\end{proposition}
For quaternion algebras we will prove:
\begin{proposition}\label{proposition:quat_count}
    For $a\geq2$ we have:
    \begin{align*}
     \sum_{\pi\in\FF_{D, k}(\ell^{2a})}1 &=\frac{k-1}{24}\ell^{2a-3}(\ell-1)^2(\ell+1),\\
    \sum_{\pi\in\FF_{D, k}(\ell^{2a+1})}1 &=\frac{k-1}{12}\ell^{2a-2}(\ell-1)^2(\ell+1).
\end{align*}
\end{proposition}
As a consequence of the Jacquet--Langlands correspondence we get:
\begin{corollary}
    Asymptotically half of all newforms of level $\ell^{2a}$ have a ramified principal series component at $\ell$. More precisely, we have \begin{align*}
        \lim_{\ell^a\to\infty}\frac{1}{\#H_k^\new(\ell^{2a})}\sum_{\substack{f\in H_k^\new(\ell^{2a})\\\text{$\pi_\ell(f)$ is principal series}}} 1= \lim_{\ell^a\to\infty}\frac{1}{\#H_k^\new(\ell^{2a})}\sum_{\substack{f\in H_k^\new(\ell^{2a})\\\text{$\pi_\ell(f)$ is supercuspidal}}}1=\frac12.
    \end{align*}
    Here $\pi_\ell(f)$ denotes the $\ell$-adic component of the automorphic representation attached to $f$.
\end{corollary}
To justify this, it suffices to note that $\FF_{D, k}(\ell^{2a})$ does not contain any one-dimensional representations and that there are no Steinberg twists of $\GL_2(\Q_\ell)$ with trivial central character and conductor exponent $\geq3$.

\subsection{Maass forms}\label{section:maass_props}
We first discuss some properties of the archimedean test function. For a continuous $\SO(2)$-biinvariant function $f : \PGL_2(\R)^+\to\C$ let $V_f\in C([0,\infty))$ be the unique function such that \begin{align*}
    f(g) = V_f\left(\frac{\tr g^tg}{\det g}-2\right).
\end{align*}
Then let $Q_f\in C([0,\infty))$ be the Abel transform of $V_f$, so \begin{align*}
    Q_f(u) = \int_\R V_f(u+x^2)\mathrm dx.
\end{align*}
Let $g_f\in C(\R)$ be defined by $g_f(v) = Q_f(e^v+e^{-v}-2)$. Then the Selberg transform of $f$ is defined to be \begin{align*}
    h_f(t) = \int_\R g_f(v)e^{ivt}\mathrm dv.
\end{align*}
It is an even function on $\R$. We assume that $h_f$ extends to a holomorphic function on the strip $\{z\in \C:\abs{\Im z}< A\}$ with $A>1$, satisfying $\abs{h_f(z)}\ll (1+\abs z)^{-B}$ for some $B>2$.

For later we need:
\begin{proposition}\label{proposition:Qf_bound}
    We have \begin{align*}
        \abs{Q_f(u)}+u\fabs{Q_f'(u)}\ll_{f}(1+u)^{-A}.
    \end{align*}
\end{proposition}
This follows from \cite[Proposition 8.15]{KniLiKTF}.
\begin{proposition}\label{proposition:Qf_log_bound}
    For $u\geq1$ we have \begin{align*}
        \int_\R V_f(u+x^2)\log\left(1+\frac{x^2}{u}\right)\mathrm dx \ll_{f}(1+u)^{-A}.
    \end{align*}
\end{proposition}
\begin{proof}
   By \cite[Proposition 8.16]{KniLiKTF} we have \begin{align*}
    V_f(u)\ll (1+u)^{-A-\frac12},
   \end{align*}
   so \begin{align*}
    \abs{\int_\R V_f(u+x^2)\log\left(1+\frac{x^2}{u}\right)\mathrm dx}&\ll\int_\R (1+u+x^2)^{-A-\frac12}\log\left(1+\frac{x^2}{u}\right)\mathrm dx\\
    &=(1+u)^{-A}\int_\R(1+y^2)^{-A-\frac12}\log\left(1+\frac{1+u}{u}y^2\right)\mathrm dy.
   \end{align*}
   In the last step we substituted $x = (1+u)^{1/2}y$. Note that for $u\geq1$ we have $\frac{1+u}{u}\leq 2$ and the last integral converges.
\end{proof}

Now we can state the trace formulas.
\begin{proposition}\label{proposition:maass_TF}
    For $a\geq1$ we have \begin{align*}
        \sum_{\phi\in H_\maass^\new(\ell^{2a+1})}h_f(t_\phi)\epsi(\phi)\sqrt P \lambda_P(\phi)= \sqrt{\frac{P}{\ell}}\sum_{\substack{t\in\Z}}Q_f\left(\frac{t^2}{\ell P}\right)L(1,\psi_{t^2 + 4\ell P})c_{\ell^{a+1}}(t).
    \end{align*}
    For $a\geq2$ we have for the even exponents:  \begin{align*}
        \sum_{\phi\in H_\maass^\new(\ell^{2a})}h_f(t_\phi)\epsi(\phi)\sqrt P \lambda_P(\phi)= \sqrt P \frac{\ell-\chi_\ell(P)}{\ell}\sum_{\substack{t\in\Z\\t^2+4P\notin\Q^{\times2}}}Q_f\left(\frac{t^2}{P}\right)L(1,\psi_{t^2 + 4P})c_{\ell^a}(t) +\mathrm{E}_{P},
    \end{align*}
    where, under GRH for Dirichlet $L$-functions, the error satisfies \begin{align*}
        \sum_{\substack{P/\ell^{2a}\in E\\\text{$P$ prime}}}\mathrm{E}_{P}\log P\ll_{E, f} \ell^{3a}a^3\log^3\ell.
    \end{align*}
\end{proposition}

We also need the following count:
\begin{proposition}\label{proposition:maass_count}
    For $n\geq2$ we have \begin{align*}
        \sum_{\phi\in H_\maass^\new(\ell^{n})}h_f(t_\phi)= \frac\pi3 f(1)\ell^{n-3}(\ell+1)(\ell-1)^2 + O_f(n(\log\ell)\ell^{n/2}).
    \end{align*}
\end{proposition}

\section{Prime Averaging Results}
The main tool used in the proof is \cite[Lemma 4.5]{BBLLDMurms} and certain analogues. Define the averages \begin{align*}
    \wtilde\psi_t(m)&:=\frac1{\varphi(m^2)}\sum_{\substack{n\mod m^2\\ (n,m)=1}}\psi_{t^2-4n}(m),\\
    \wtilde\psi_{t, \chi_\ell}(m) &:= \frac1{\varphi(\ell m^2)}\sum_{\substack{n\mod \ell m^2\\(n, \ell m)=1}}\chi_\ell(-n)\psi_{t^2-4n}(m),\\
    \wtilde\psi_t^{(\ell)+}(m)&:=\frac{1}{\varphi(m^2)}\sum_{\substack{n\mod m^2\\(n,m)=1}}\psi_{t^2+4\ell n}(m),\\
    \wtilde\psi_{t}^{+}(m)&:=\frac{1}{\varphi(m^2)}\sum_{\substack{n\mod m^2\\(n,m)=1}}\psi_{t^2+4n}(m),\\
    \wtilde\psi_{t,\chi_\ell}^{+}(m)&:=\frac{1}{\varphi(\ell m^2)}\sum_{\substack{n\mod \ell m^2\\(n,\ell m)=1}}\chi_\ell(n)\psi_{t^2+4n}(m).
\end{align*}
We use the term ``holomorphic'' to describe both modular forms and automorphic forms on the quaternion algebra. In the ``holomorphic'' case we will consider prime powers $\ell^{2a}$ with $a\geq2$, while in the ``Maass odd'' (resp.\ ``Maass even'') case we will consider prime powers $\ell^{2a+1}$ (resp.\ $\ell^{2a}$).
\begin{lemma}\label{lemma:prime_Phi_average}
    Let $t\in\Z$, $A,B\in\R$ with $A<B$, and $\Phi\in C^1([A,B])$. Set $M=\norm\Phi_\infty$ and $V = \norm{\Phi'}_1$. In the ``holomorphic'' case we also assume $\frac{t^2}4< A$. Then, assuming GRH for Dirichlet $L$-functions:
    \begin{align*}
        \sideset{}{^{*}}\sum_{\substack{P\in [A,B]\\\text{$P$ prime}}}\omega(P)L(1, \psi_{D(P)})\Phi(P)\log P = L(1,\wtilde \psi_t^*)\int_A^B\Phi(u)\mathrm du + O_\epsi(RM^{4/5}(M+V)^{1/5}B^{9/10+\epsi}),
    \end{align*}
    where:
    \[
    \renewcommand{\arraystretch}{1.25}
    \begin{array}{l|l|l|l|l|l}
        \text{Case} & D(P) & \omega(P) & \wtilde\psi_t^* & R & \Sigma^*\\
        \hline
        \text{holomorphic} & t^2-4P & 1 & \wtilde\psi_t & 1 & -\\
        \text{holomorphic twisted} & t^2-4P & \chi_\ell(-P) & \wtilde\psi_{t,\chi_\ell} & \ell^{1/5} & -\\
        \text{Maass odd} & t^2+4\ell P & 1 & \wtilde\psi_t^{(\ell)+} & 1 & \text{$t^2+4\ell P\ne\square$}\\
        \text{Maass even} &t^2+4P & 1 & \wtilde\psi_t^{+} & 1 & \text{$t^2+4P\ne\square$}\\
        \text{Maass even twisted} & t^2+4P & \chi_\ell(P) & \wtilde\psi_{t,\chi_\ell}^{+} & \ell^{1/5} &\text{$t^2+4P\ne\square$}
    \end{array}
    \]
    Here $\Sigma^*$ indicates that we impose the indicated condition on $P$ in the sum.
\end{lemma}
\begin{proof}
    The first case (``holomorphic'') is exactly \cite[Lemma 4.5]{BBLLDMurms}. The other cases are proved similarly. The condition that $t^2+4P$ is not a square in the even Maass cases is harmless. Indeed, if $t^2+4P$ is a square, then necessarily $P=\abs t +1$, so at most one prime is excluded by this condition. A similar remark applies to the odd Maass case. In fact we will only need the result in the case $\ell\mid t$ in which case $t^2+4\ell P$ is automatically not a square. We only sketch the ``holomorphic twisted'' case to explain how to handle the character $\chi_\ell$. Let $I = \{P\in (A, B]: \text{$P$ prime}\}$. For any $m\geq1$ we write \begin{align*}
        \sum_{P\in I}\chi_\ell(-P)\Phi(P)\psi_{t^2-4P}(m)\log P &= \sum_{\substack{n\mod \ell m^2\\\gcd(n, \ell m)=1}}\sum_{\substack{P\in I\\ P\equiv n\mod \ell m^2}}\chi_\ell(-P)\Phi(P)\psi_{t^2-4P}(m)\log P\\
        &\hspace{3em} + \sum_{\substack{P\in I\\ P\mid \ell m^2}}\chi_\ell(-P)\Phi(P)\psi_{t^2-4P}(m)\log P.
    \end{align*}
    The second sum is $O(M\log(\ell m))$. The first can be written as \begin{align*}
        \sum_{\substack{n\mod \ell m^2\\\gcd(n, \ell m)=1}}\chi_\ell(-n)\psi_{t^2-4n}(m)\sum_{\substack{P\in I\\ P\equiv n\mod \ell m^2}}\Phi(P)\log P
    \end{align*}
    by \cite[Lemma 4.1]{BBLLDMurms}.
    The inner sum is estimated using the GRH error term for primes in arithmetic progressions as in \cite[Lemma 4.5]{BBLLDMurms}, but with modulus $\ell m^2$, and we obtain \begin{align*}
        \sum_{P\in I}\chi_\ell(-P)\Phi(P)\psi_{t^2-4P}(m)\log P = \wtilde \psi_{t, \chi_\ell}(m)\int_A^B\Phi(u)\mathrm du + O(\varphi(\ell m^2)(M+V)B^{1/2}\log^2B).
    \end{align*}
    Using \cite[Lemma 4.1]{BBLLDMurms} and \cite[Lemma 4.2]{BBLLDMurms} one gets \begin{align*}
        \wtilde\psi_{t,\chi_\ell}(\ell^e r) = \wtilde\psi_{t,\chi_\ell}(\ell^e)\wtilde \psi_t(r)
    \end{align*}
    for $\ell\nmid r$. Thus, $\fabs{\wtilde\psi_{t,\chi_\ell}(\ell^e r)}\leq \fabs{\wtilde \psi_t(r)}$ and we also see $\wtilde\psi_{t,\chi_\ell}(r)=0$ if $\ell\nmid r$. From this we get the tail estimate \begin{align*}
        \sum_{m>x}\frac{\fabs{\wtilde\psi_{t,\chi_\ell}(m)}}m\leq\sum_{e=0}^\infty\frac1{\ell^e}\sum_{\substack{r>x/\ell^e\\ (r,\ell)=1}}\frac{\fabs{\wtilde\psi_t(r)}}r\ll x^{-1/2}.
    \end{align*} Here the last step comes from the tail estimate in \cite[Lemma 4.5]{BBLLDMurms}. The rest of the argument then goes through as in the proof of \cite[Lemma 4.5]{BBLLDMurms}.
\end{proof}

We can compare the various $\wtilde\psi^*$:
\begin{lemma}\label{lemma:L1_values}
    Assume $\ell\mid t$. Then \begin{align*}
        L(1, \wtilde\psi_{t,\chi_\ell}) &= \ell^{-1} L(1,\wtilde\psi_{t}),\\
        L(1, \wtilde\psi_t^{(\ell)+})&=L(1,\wtilde\psi_t)(1-\ell^{-2}),\\
        L(1, \wtilde\psi_t^{+})&=L(1,\wtilde\psi_t),\\
        L(1, \wtilde\psi_{t,\chi_\ell}^{+})&=\ell^{-1}L(1,\wtilde\psi_t).
    \end{align*}
\end{lemma}
\begin{proof}
    Making the change of variables $n\mapsto -n$ in the definition of the $\wtilde\psi^+$ shows \begin{align*}
        \wtilde\psi_t^{(\ell)+} &= \wtilde\psi_t^{(\ell)},\\
        \wtilde\psi_t^{+} &= \wtilde\psi_t,\\
        \wtilde\psi_{t,\chi_\ell}^+ &= \wtilde\psi_{t,\chi_\ell}.
    \end{align*}
    Here $\wtilde\psi_t^{(\ell)}$ is from \cite{BKLMYMurms} and in particular \cite[Lemma 3.2]{BKLMYMurms} shows that $L(1,\wtilde\psi_t^{(\ell)}) = L(1,\wtilde\psi_t)(1-\ell^{-2})$. Thus, we only have to show the first identity $L(1, \wtilde\psi_{t,\chi_\ell}) = \ell^{-1} L(1,\wtilde\psi_{t})$. We have already noted in the proof of \cref{lemma:prime_Phi_average} that $\wtilde\psi_{t,\chi_\ell}(m)=0$ if $\ell\nmid m$. We now claim that $\wtilde\psi_{t,\chi_\ell}(\ell^er)=\wtilde\psi_{t}(\ell^{e-1}r)$ for $e\geq1$. Indeed, if $(n,\ell r)=1$, then $t^2-4n\equiv -4n\ne0\mod\ell$, so $\psi_{t^2-4n}(\ell^e) = \legendrebig{t^2-4n}{\ell^e} = \chi_\ell(-n)\legendrebig{t^2-4n}{\ell^{e-1}}=\chi_\ell(-n)\psi_{t^2-4n}(\ell^{e-1})$. Then by \cite[Lemma 4.1]{BBLLDMurms}, \begin{align*}
        \wtilde\psi_{t,\chi_\ell}(\ell^er)&=\frac{1}{\varphi(\ell^{2e+1}r^2)}\sum_{\substack{n\mod \ell^{2e+1}r^2\\(n,\ell r)=1}}\psi_{t^2-4n}(\ell^{e-1}r)\\
        &=\frac{1}{\varphi(\ell^{2e+1}r^2)}\frac{\varphi(\ell^{2e+1}r^2)}{\varphi(\ell^{2e-2}r^2)}\sum_{\substack{n\mod \ell^{2e-2}r^2\\(n,\ell^{e-1} r)=1}}\psi_{t^2-4n}(\ell^{e-1}r)\\
        &=\wtilde\psi_t(\ell^{e-1}r).
    \end{align*}
    Then we see that \begin{align*}
        L(1,\wtilde\psi_{t,\chi_\ell}) &= \sum_{e = 1}^\infty \frac{1}{\ell^e}\sum_{\substack{r=1\\ (r,\ell)=1}}^\infty \frac{\wtilde \psi_t(\ell^{e-1}r)}{r}\\
        &=\ell^{-1}L(1,\wtilde\psi_t).
    \end{align*}
\end{proof}

We apply \cref{lemma:prime_Phi_average} to our situation. For the Maass form error estimates we introduce \begin{align*}
     w_{\ell, a}(t) &= \left(1+\frac{t^2}{\ell^{2a}}\right)^{-A},
\end{align*}
where $A$ is as in \cref{section:maass_props}.

\begin{corollary}\label{corollary:P_sum_asymptotic}
   Let $E\subset \mathbb R_{>0}$ be compact and $\ell\mid t$. Then, assuming GRH for Dirichlet $L$-functions:
    \begin{align*}
        &\sideset{}{^{*}}\sum_{\substack{P\in \ell^{n}E\\\text{$P$ prime}}}\omega(P)\Phi_{\ell, t}(P)L(1,\psi_{D(P)})\log P\\
        &=L(1, \wtilde\psi_t)\ell^{n}C\int_{E}^*\Phi_{\ell, t}(\ell^nv)\mathrm dv + O_{\epsi, k, f, E}(R\ell^{14 a/5+a\epsi}),
    \end{align*}
    where $n=2a+1$ in the ``Maass odd'' case and $n=2a$ otherwise. The sum excludes primes for which $t^2+4P$ is a square in the ``Maass even'' cases, primes $P$ with $P\leq t^2/4$ in the ``holomorphic'' cases, and the integral is over $E\cap (t^2/(4\ell^{2a}),\infty)$ in the ``holomorphic'' cases. Moreover,
    \[
        \renewcommand{\arraystretch}{1.25}
        \begin{array}{l|l|l|l|l|l}
        \text{Case} &  \Phi_{\ell, t}(u) & D(P) & \omega(P) &C & R \\
        \hline
        \text{holomorphic} &  \sqrt{4u-t^2}U_{k-2}(t/(2\sqrt u)) & t^2-4P & 1& 1 & 1 \\
        \text{holomorphic twisted} &  \sqrt{4u-t^2}U_{k-2}(t/(2\sqrt{u})) &  t^2-4P &\chi_\ell(-P)&\ell^{-1} & \ell^{1/5} \\
        \text{Maass odd} &  Q_f(t^2/(\ell u))\sqrt{u/\ell} & t^2+4\ell P& 1 & 1-\ell^{-2} &\ell^{9/10}w_{\ell,a+1}(t) \\
        \text{Maass even} &  Q_f(t^2/u)\sqrt u & t^2+4P & 1 & 1 & w_{\ell,a}(t) \\
        \text{Maass even twisted} & Q_f(t^2/u)\sqrt u & t^2+4P & \chi_\ell(P) & \ell^{-1} &\ell^{1/5}w_{\ell,a}(t).
        \end{array}
    \]
\end{corollary}
\begin{proof}
    This follows from \cref{lemma:prime_Phi_average}. In all cases we use \cref{lemma:L1_values} to simplify the $L$ value to $L(1,\wtilde\psi_t)C$ and obtain the final integrals by substituting $v = \ell^{-n}u$.
    
    We sketch the arguments. In the ``holomorphic'' cases we take $\Phi(u) = \sqrt{4u-t^2}U_{k-2}\left(\frac{t}{2\sqrt{u}}\right)$. Then as in \cite[p.9]{BKLMYMurms} one sees that $M, V\ll_{E, k}\ell^a$ on $\ell^{2a}E$. We also remark that we have to apply \cref{lemma:prime_Phi_average} for $A = \frac {t^2}4$ as well in which case $\Phi$ is not $C^1$ at $A$. This causes no difficulty, since $\int_A^B\abs{\Phi'(u)}\mathrm du$ is still bounded in the same way, so a small approximation argument using $A_\delta = \frac{t^2}4+\delta$ and letting $\delta\to0$ suffices.

    In the ``Maass odd'' case we take $\Phi(u) = Q_f\left(\frac{t^2}{\ell u}\right)\sqrt{u/\ell}$ on $\ell^{2a+1}E$. Note that here the condition $t^2+4\ell P\ne\square$ is vacuous since $\ell\mid t$. It follows from \cref{proposition:Qf_bound} that \begin{align*}
        M&\ll_{f, E} \ell^{a}w_{\ell, a+1}(t),\\
        V&\ll_{f, E} \ell^{a}w_{\ell, a+1}(t).
    \end{align*} 
    In the ``Maass even'' cases let $\Phi(u) = Q_f\left(\frac{t^2}{u}\right)\sqrt {u}$ on $\ell^{2a}E$. It follows from \cref{proposition:Qf_bound} that \begin{align*}
        M&\ll_{f, E} \ell^{a}w_{\ell, {a}}(t),\\
        V&\ll_{f, E} \ell^{a}w_{\ell, a}(t).
    \end{align*} 
    The rest is algebraic simplification.
\end{proof}

Finally, we also record the trivial estimate:
\begin{lemma}\label{lemma:trivial_hyperbolic_estimate}
    Let $E$ be a compact interval in $\R_{>0}$. Then \[
        \sum_{\substack{P/\ell^{2a}\in E\\ \text{$P$ prime}}} (\log P)c_{\ell^a}(P+1)P^{1-k/2} = O_{E}(a\ell^{3a}\log \ell).
    \]
\end{lemma}

\section{Proofs of the Main Theorems}
\subsection{Proof of \cref{theorem:n_limit}}
The theorem now follows quickly from \cref{corollary:P_sum_asymptotic} and the trace formulas as in \cite{BKLMYMurms}. For example, writing $E = [\alpha,\beta]$, in the $\GL_2$-case we have \begin{align*}
    &(-1)^{k/2+1}\sum_{\substack{P/\ell^{2a}\in E\\ \text{$P$ prime}}}\log P\sum_{\substack{f\in H^\new_k(\ell^{2a})}}\epsi(f)\sqrt P\lambda_P(f) \\
    &=\sum_{\substack{P/\ell^{2a}\in E\\ \text{$P$ prime}}}\log P\Bigg(\frac{\ell-\chi_\ell(-P)}{2\pi\ell}\sum_{t\in\Z, t^2<4P}\sqrt{4P-t^2}U_{k-2}\left(\frac{t}{2\sqrt{P}}\right)c_{\ell^a}(t)L(1,\psi_{t^2-4P})\\ &\hspace{5em}+ \frac{\ell-1}\ell c_{\ell^a}(P+1)P^{1-k/2}\Bigg)\\
    &=\left(\sum_{\substack{t\in \Z\\ t^2 < 4\beta\ell^{2a}}}c_{\ell^a}(t)\sum_{\substack{P\in \ell^{2a}E\cap (t^2/4, \infty)\\ \text{$P$ prime}}}(\log P)\frac{\ell-\chi_\ell(-P)}{2\pi\ell}\sqrt{4P-t^2}U_{k-2}\left(\frac{t}{2\sqrt{P}}\right)L(1,\psi_{t^2-4P})\right)\\&\hspace{5em}+ O_{E}(a\ell^{3a}\log \ell)\\
    &=\frac1{2\pi}\sum_{\substack{t\in \Z\\ t^2 < 4\beta\ell^{2a}}}c_{\ell^a}(t)\Bigg(L(1, \wtilde\psi_t)\ell^{2a}(1-\ell^{-2})\int_{E\cap (t^2/(4\ell^{2a}),\infty)}\sqrt{4\ell^{2a}v-t^2}U_{k-2}\left(\frac t{2\ell^a\sqrt v}\right)\mathrm dv\\&\hspace{9em} + O_{\epsi, k, E}(\ell^{14 a/5+a\epsi})\Bigg) + O_{E}(a\ell^{3a}\log \ell).
\end{align*}
Here we used \cref{corollary:P_sum_asymptotic} and \cref{lemma:trivial_hyperbolic_estimate}. Note that in \cref{corollary:P_sum_asymptotic} we need $\ell\mid t$. Indeed, only terms with $\ell^{a-1}\mid t$ contribute, since otherwise $c_{\ell^a}(t)\ne0$. Hence, we can change variables $\ell^{a-1}r = t$ and obtain \begin{align*}
    &(-1)^{k/2+1}\sum_{\substack{P/\ell^{2a}\in E\\ \text{$P$ prime}}}\log P\sum_{\substack{f\in H^\new_k(\ell^{2a})}}\epsi(f)\sqrt P\lambda_P(f)\\
    &=\frac1{\pi}\sum_{\substack{r\in \Z\\ r^2 <4\beta\ell^2}}\left(c_{\ell^a}(\ell^{a-1}r)L(1, \wtilde\psi_{\ell^{a-1}r})\ell^{3a}(1-\ell^{-2})\int_{E\cap (r^2/(4\ell^{2}),\infty)}\sqrt{v-\frac{r^2}{4\ell^2}}U_{k-2}\left(\frac r{2\ell\sqrt v}\right)\mathrm dv\right)\\&\hspace{5em}+O_{\epsi, k, E}(\ell^{19 a/5+a\epsi}) + O_{E}(a\ell^{3a}\log \ell).
\end{align*}
We also used $\sum_{t^2<4\beta\ell^{2a}}\abs{c_{\ell^a}(t)}\ll_E\ell^a$. We can drop the second error term as it is dominated by the first. \cref{proposition:GL2_count} and the Prime Number Theorem with GRH error term give us \begin{align*}
    \sum_{\substack{P/\ell^{2a}\in E\\ \text{$P$ prime}}}\log P\sum_{\substack{f\in H^\new_k(\ell^{2a})}}1 = \frac{k-1}{12}\abs E\ell^{4a-3}(\ell-1)^2(\ell+1)(1+O_{E, k}(\ell^{-a}(a\log\ell)^2)),
\end{align*}
so \begin{align*}
    \frac1{\displaystyle\sum_{\substack{P/\ell^{2a}\in E\\ \text{$P$ prime}}}\log P\sum_{\substack{f\in H^\new_k(\ell^{2a})}}1} = \left(\frac{k-1}{12}\abs E\ell^{4a-3}(\ell-1)^2(\ell+1)\right)^{-1}(1+O_{E, k}(\ell^{-a}(a\log\ell)^2)),
\end{align*}
as $\ell^a\to\infty$. Using this together with \[L(1, \wtilde\psi_{\ell^{a-1}r}) = \frac{\pi^2}6\prod_{p\nmid \ell r}\left(\frac{p^2-p-1}{p(p-1)}\right)\] from \cite[Lemma 4.3]{BBLLDMurms} and simplifying gives the modular forms case of \cref{theorem:n_limit}. The quaternionic case is proved in exactly the same way.

Next we consider Maass forms. In the case of odd exponents $n=2a+1$ we use \[\sum_{t\in\Z}\fabs{c_{\ell^{a+1}}(t)}w_{\ell, {a+1}}(t)\ll\ell^{a+1}\] to obtain \begin{align*}
    &\sum_{\substack{P/\ell^{2a+1}\in E\\ \text{$P$ prime}}}\log P\sum_{\substack{\phi\in H^\new_\maass(\ell^{2a+1})}}h_f(t_\phi)\epsi(\phi)\sqrt P\lambda_P(\phi)\\
    &=\sum_{t\in\Z}c_{\ell^{a+1}}(t)\left(L(1,\wtilde\psi_{t})(1-\ell^{-2})\ell^{2a+1}\int_E Q_f\left(\frac{t^2}{\ell^{2a+2}v}\right)\sqrt{\ell^{2a}v}\mathrm dv  + O_{\epsi,f,E}(\ell^{14a/5+a\epsi+9/10}w_{\ell,a+1}(t))\right)\\
    &=\sum_{r\in\Z}(\one_{\ell\mid r}-\ell^{-1})L(1,\wtilde\psi_{\ell r})(1-\ell^{-2})\ell^{4a+2}\int_E Q_f\left(\frac{r^2}{\ell^2 v}\right)\sqrt v\mathrm dv + O_{\epsi, f, E}\left(\ell^{19a/5+a\epsi+19/10}\right).
\end{align*}
For even exponents $n = 2a$ we obtain \begin{align*}
    &\sum_{\substack{P/\ell^{2a}\in E\\ \text{$P$ prime}}}\log P\sum_{\substack{\phi\in H^\new_\maass(\ell^{2a})}}h_f(t_\phi)\epsi(\phi)\sqrt P\lambda_P(\phi)\\
    &=\sum_{r\in\Z}(\one_{\ell\mid r}-\ell^{-1})L(1,\wtilde\psi_{\ell r})(1-\ell^{-2})\ell^{4a}\int_E Q_f\left(\frac{r^2}{\ell^2 v}\right)\sqrt v\mathrm dv + O_{\epsi, f, E}\left(\ell^{19a/5+a\epsi}\right).
\end{align*}
Note the extra error term coming from \cref{proposition:maass_TF} in the even case is smaller than the error here, so we may drop it. \cref{proposition:maass_count} and the prime number theorem with GRH error term gives
\begin{align*}
    \sum_{\substack{P/\ell^{n}\in E\\ \text{$P$ prime}}}\log P \sum_{\phi\in H_\maass^\new(\ell^{n})}h_f(t_\phi) = \abs E\frac{\pi}{3}f(1)\ell^{2n-3}(\ell-1)^2(\ell+1)\left(1+O_{E, f}(\ell^{-n/2}(n\log\ell)^2)\right).
\end{align*}
Plugging this in and simplifying, we obtain the Maass case of \cref{theorem:n_limit}.

\subsection{Proof of \cref{theorem:ell_limit}}
From the proof of \cref{theorem:n_limit} we see the $\ell$-dependence of the error term:\begin{align*}\label{eqn:gl2_ell_explicit_error}
    \frac{\displaystyle\sum_{\substack{P/\ell^{2a}\in E\\ \text{$P$ prime}}}\log P\sum_{f\in H^\new_k(\ell^{2a})}\epsi(f)P^{1/2}\lambda_P(f)}{\displaystyle\sum_{\substack{P/\ell^{2a}\in E\\ \text{$P$ prime}}}\log P \sum_{f\in H^\new_k(\ell^{2a})}1}
        = M_{E,\ell,k} + O_{E, k, a, \epsi}(\ell^{-a/5+\epsi}) + O_{E, k}(M_{E, \ell, k}\ell^{-a}(a\log \ell)^2). \tag{$\dagger$}
\end{align*}
For the quaternion algebra case the only difference is that the first error term is $O_{E, k, a, \epsi}(\ell^{-a/5+\frac15+\epsi})$. We deduce \cref{theorem:ell_limit} from this by computing $\lim_{\ell\to\infty}\MM_{\ell, k}(v)$. Recall that \begin{align*}
    \MM_{\ell, k}(v) = (-1)^{\frac k2+1}\frac{2\pi}{(k-1)(1-\ell^{-1})}\sum_{\substack{r\in\Z\\ \abs r < 2\ell\sqrt v}}(\one_{\ell\mid r}-\ell^{-1})A_{\ell r}\sqrt{v-\frac {r^2}{4\ell^2}}U_{k-2}\left(\frac r{2\ell\sqrt v}\right),
\end{align*}
with $A_{t}=\prod_{\substack{p\nmid t}}\left(\frac{p^2-p-1}{p(p-1)}\right)$. Abbreviate $c_p = \frac{p^2-p-1}{p(p-1)}$. Let $f_v(x) = \sqrt{v-\frac{x^2}4}U_{k-2}\left(\frac x{2\sqrt v}\right)$. Note that $f_v(x) = O_k(\sqrt v)$. We may split up the sum as \begin{align*}
   \sum_{\substack{r\in\Z\\ \abs r < 2\ell\sqrt v}}(\one_{\ell\mid r}-\ell^{-1})A_{\ell r}f_v(r/\ell) = (1-\ell^{-1})\sum_{\substack{s\in\Z\\ \abs s < 2\sqrt v}}A_{\ell s}f_v(s)-\ell^{-1}\sum_{\substack{r\in\Z\\ \abs r < 2\ell\sqrt v, \ell\nmid r}}A_{\ell r}f_v(r/\ell).
\end{align*}
For $s\in\Z-\{0\}$ we have $A_{\ell s} = A_s(1+(c_\ell^{-1}-1)\one_{\ell\nmid s})$, hence the first sum is \begin{align*}
    \sum_{\substack{s\in\Z\\ \abs s < 2\sqrt v}}A_{\ell s}f_v(s) = \sum_{\substack{s\in\Z\\ \abs s < 2\sqrt v}}A_{s}f_v(s) + O_k(\ell^{-2}v).
\end{align*}
Next the second sum. For $\ell\nmid r$ we have $A_{\ell r} = c_\ell^{-1}A_r$, hence we can write \begin{align*}
    \ell^{-1}\sum_{\substack{r\in\Z\\ \abs r < 2\ell\sqrt v, \ell\nmid r}}A_{\ell r}f_v(r/\ell) &= \ell^{-1}c_\ell^{-1}\Bigg(\sum_{\substack{r\in\Z\\ \abs r < 2\ell\sqrt v}}A_{r}f_v(r/\ell)-\sum_{\substack{r\in\Z\\ \abs r < 2\ell\sqrt v, \ell\mid r}}A_{r}f_v(r/\ell)\Bigg)\\
    &=\ell^{-1}\sum_{\substack{r\in\Z\\ \abs r < 2\ell\sqrt v}}A_{r}f_v(r/\ell) + O_k(\ell^{-1}v).
\end{align*}
Let $q(d) = \mu^2(d)\prod_{p\mid d}(c_p^{-1}-1) = \mu^2(d)\prod_{p\mid d}\frac1{p^2-p-1}$. Then $A_r = A_1\sum_{d\mid r}q(d)$. If we write $g_v(x) = f_v(x)\one_{[-2\sqrt v, 2\sqrt v]}(x)$, then \begin{align*}
    \frac1\ell\sum_{\substack{r\in\Z\\ \abs r < 2\ell\sqrt v}}A_{r}f_v(r/\ell) &= A_1\frac1\ell\sum_{r\in\Z}\Big(\sum_{d\mid r}q(d)\Big)g_v(r/\ell)\\
    &=A_1\sum_{d=1}^\infty q(d)\frac1\ell\sum_{m\in\Z}g_v(dm/\ell).
\end{align*}
Note that the inner sum is a Riemann sum and we have \begin{align*}
    \frac1\ell\sum_{m\in\Z}g_v(dm/\ell) = \frac1d\int_\R g_v(x)\mathrm dx+O\Big(\frac1\ell\int_\R\abs{g_v'(u)}\mathrm du\Big).
\end{align*}
One easily computes that $\int_\R\abs{g_v'(u)}\mathrm du = O_k(\sqrt v)$ and \begin{align*}
    \int_\R g_v(x) \mathrm dx= \int_{-2\sqrt v}^{2\sqrt v}\sqrt{v-\frac{x^2}4}U_{k-2}\left(\frac x{2\sqrt v}\right)\mathrm dx=2v\int_{-1}^1\sqrt{1-t^2}U_{k-2}(t)\mathrm dt = \pi v\one_{k=2}.
\end{align*}
The last equality follows from the orthogonality of the Chebyshev polynomials with respect to the measure $\sqrt{1-t^2}\mathrm dt$. Hence \begin{align*}
     \frac1\ell\sum_{\substack{r\in\Z\\ \abs r < 2\ell\sqrt v}}A_{r}f_v(r/\ell)&=A_1\sum_{d=1}^\infty q(d)\Big(\pi d^{-1}v\one_{k=2} + O_k(\ell^{-1}\sqrt v)\Big)\\
     &=\one_{k=2}\pi vA_1\sum_{d=1}^\infty \frac{q(d)}{d} + O_k(\ell^{-1}\sqrt v).
\end{align*}
Note that \[
    A_1\sum_{d=1}^\infty \frac{q(d)}d = \prod_p\frac{p^2-p-1}{p(p-1)}\prod_p\left(1+\frac{1}{p(p^2-p-1)}\right)=\frac1{\zeta(2)}=\frac6{\pi^2}.
\]
So putting everything together we get \begin{align*}
    \MM_{\ell, k}(v)
    &=(-1)^{\frac k2+1}\frac{2\pi}{(k-1)}\Bigg(\sum_{\substack{s\in\Z\\ \abs s < 2\sqrt v}}A_{s}\sqrt{v-\frac{s^2}4}U_{k-2}\left(\frac s{2\sqrt v}\right) - \one_{k=2}\frac6\pi v + O_k(\ell^{-1}(\sqrt v + v))\Bigg).
\end{align*}
Integrating gives \begin{align*}
    M_{E, \ell, k} = M_{E, \infty, k} + O_{E, k}(\ell^{-1}).
\end{align*}
Certainly $M_{E, \ell, k}\ell^{-a}(a\log \ell)^2\ll_{E,k, a} \ell^{-1}$. Therefore, by \ref{eqn:gl2_ell_explicit_error}\begin{align*}
    \frac{\displaystyle\sum_{\substack{P/\ell^{2a}\in E\\ \text{$P$ prime}}}\log P\sum_{f\in H^\new_k(\ell^{2a})}\epsi(f)P^{1/2}\lambda_P(f)}{\displaystyle\sum_{\substack{P/\ell^{2a}\in E\\ \text{$P$ prime}}}\log P \sum_{f\in H^\new_k(\ell^{2a})}1}
        = M_{E, \infty, k} + O_{E, k, a, \epsi}(\ell^{-a/5+\epsi}+\ell^{-1}).
\end{align*}
The quaternion algebra case follows immediately. The only difference is that in \ref{eqn:gl2_ell_explicit_error} the first error is now $O_{E, k, a,\epsi}(\ell^{-a/5+\frac15+\epsi})$, resulting in the slightly different error term in \cref{theorem:ell_limit} (ii). The result for Maass forms is proved in the same way. We only note that the analogue of \ref{eqn:gl2_ell_explicit_error} is now \begin{align*}
    &\frac{\displaystyle\sum_{\substack{P/\ell^{n}\in E\\ \text{$P$ prime}}}\log P\sum_{\phi\in H_\maass^\new(\ell^{n})}h_f(t_\phi)\epsi(\phi)P^{1/2}\lambda_P(\phi)}{\displaystyle\sum_{\substack{P/\ell^{n}\in E\\ \text{$P$ prime}}}\log P \sum_{\phi\in H_\maass^\new(\ell^{n})}h_f(t_\phi)} \\
    &= M_{E,\ell, f} + O_{E, f, n,\epsi}(\ell^{-n/10 + \epsi}) + O_{E, f}(M_{E, \ell, f}\ell^{-n/2}(n\log\ell)^2).
\end{align*}
In the proof the function $f_v(x)$ is now $\sqrt v Q_f\left(\frac{x^2}{v}\right)$ and $g_v = f_v$ as there is no compact support cutoff in the Maass form case.

\section{Proofs of the Explicit Trace Formulas}\label{section:tf_proofs}
First note that for both \cref{proposition:GL2_TF} and \cref{proposition:quat_TF} the second equality follows from \[
    H(D) = \frac{\sqrt{\abs D}}{\pi}L(1,\psi_D),
\]
for any negative discriminant $D$ by \cite[(2.4)]{BKLMYMurms}.
\subsection{Proof of \cref{proposition:GL2_TF}}
We proceed similarly to \cite{BKLMYMurms} and use an inclusion-exclusion argument to derive the expression for the trace on the new space using the explicit Eichler-Selberg trace formula.

First note that \begin{align*}
    \sum_{\substack{f\in H^\new_k(\ell^{2a})}}\epsi(f)\sqrt P\lambda_P(f) = (-1)^{k/2}P^{(2-k)/2}\tr\big(T_P\circ W_{\ell^{2a}}\mid S_k^\new(\Gamma_0(\ell^{2a}))\big).
\end{align*}
Indeed, the Atkin-Lehner involution $W_{\ell^{2a}}$ gives us the local root number at the $\ell$-adic place and $(-1)^{k/2}$ is the local root number at infinity. 
By inverting the formula \cite[§2 (5)]{SkoZagJFMF}\begin{align*}
    \tr(T_P\circ W_{N}\mid S_k(N)) = \sum_{\substack{M| N\\ N/M=\square}}\tr(T_P\circ W_M\mid S_k^\new(M)),
\end{align*}
we obtain
\begin{align*}
    \tr(T_P\circ W_{\ell^{2a}}\mid S_k^\new(\ell^{2a})) = \tr(T_P\circ W_{\ell^{2a}}\mid S_k(\ell^{2a}))-\tr(T_P\circ W_{\ell^{2a-2}}\mid S_k(\ell^{2a-2})).
\end{align*}
For the trace on the old spaces we have the formula \cite[§2 (7)]{SkoZagJFMF}\begin{align*}
    \tr(T_P\circ W_{\ell^{2a}}\mid S_k(\ell^{2a})) = &-\frac12\sum_{r \in \{2a, 2a-2\}}\mu(\ell^{a-r/2})\sum_{\substack{s^2<4\ell^rP\\\ell^{a+r/2}\mid s}}P^{(k-2)/2}U_{k-2}\left(\frac s{2\sqrt{\ell^rP}}\right)H(s^2-4P\ell^r)\\
    &-\varphi(\ell^a)\one_{\ell^a\mid (P+1)} + \one_{k=2}(P+1).
\end{align*}
Therefore \begin{align*}
    \tr&(T_P\circ W_{\ell^{2a}}\mid S_k^\new(\ell^{2a})) \\
    = &-\frac12\sum_{r \in \{2a, 2a-2\}}\mu(\ell^{a-r/2})\sum_{\substack{s^2<4\ell^rP\\\ell^{a+r/2}\mid s}}P^{(k-2)/2}U_{k-2}\left(\frac s{2\sqrt{\ell^rP}}\right)H(s^2-4P\ell^r)\\
    &+\frac12\sum_{r \in \{2a-2, 2a-4\}}\mu(\ell^{a-1-r/2})\sum_{\substack{s^2<4\ell^rP\\\ell^{a-1+r/2}\mid s}}P^{(k-2)/2}U_{k-2}\left(\frac s{2\sqrt{\ell^rP}}\right)H(s^2-4P\ell^r)\\
    &-\varphi(\ell^a)\one_{\ell^a\mid (P+1)} + \one_{k=2}(P+1)\\
    &+\varphi(\ell^{a-1})\one_{\ell^{a-1}\mid (P+1)} - \one_{k=2}(P+1).
\end{align*}
Note that the $k=2$ terms cancel and the hyperbolic terms add up to the desired $-\frac{\ell-1}{\ell}c_{\ell^a}(P+1)$. For the elliptic terms we change variables in the inner sum to $s = \ell^{r/2}t$ and abbreviate $D_t = t^2-4P$, so we get \begin{align*}
    \tr(T_P\circ W_{\ell^{2a}}\mid S_k^\new(\ell^{2a})) = \frac12P^{(k-2)/2}\sum_{t^2<4P}U_{k-2}\left(\frac t{2\sqrt{P}}\right)H_{t} - \frac{\ell-1}{\ell}c_{\ell^a}(P+1).
\end{align*}
where \[
    H_t := \one_{\ell^{a-1}\mid t}(H(\ell^{2a-2}D_t)-H(\ell^{2a-4}D_t))-\one_{\ell^a\mid t}( H(\ell^{2a}D_t)-H(\ell^{2a-2}D_t)).
\]
If $D$ is a discriminant with $\ell\nmid D$, it then follows from \cite[Theorem 7.24]{CoxPF} that for $j\geq1$\begin{align*}
    H(\ell^{2j}D)-H(\ell^{2j-2}D) = \ell^j\left(1-\ell^{-1}\legendre{D}{\ell}\right)H(D).
\end{align*}
Now note that when $\ell^{a-1}\mid t$ we have $\ell\nmid D_t$ and moreover $D_t\equiv -4P\mod \ell$, hence\begin{align*}
     H_t &= H(D_t)\left(\one_{\ell^{a-1}\mid t}\ell^{a-1}\left(1-\ell^{-1}\legendre{-P}{\ell}\right)-\one_{\ell^a\mid t}\ell^{a}\left(1-\ell^{-1}\legendre{-P}{\ell}\right)\right)\\
     &=-\frac{\ell-\chi_\ell(-P)}{\ell}c_{\ell^a}(t)H(D_t).
\end{align*}
Plugging this back into the formula for $\sum_{\substack{f\in H^\new_k(\ell^{2a})}}\epsi(f)\sqrt P\lambda_P(f) = (-1)^{k/2}P^{(2-k)/2}\tr(T_P\circ W_{\ell^{2a}}\mid S_k^\new(\Gamma_0(\ell^{2a})))$ we obtain \cref{proposition:GL2_TF}.

\begin{remark*}
    This explicit trace formula can be derived along the same lines as in the next section, using the adelic Arthur-Selberg trace formula for $\GL_2$ in \cite{KniLiTHO}. Then we would use the test function $r_n = e_n-e_{n-2}$ with $e_n = \vol(\overline{K_0(\ell^n)})^{-1}\one_{\left(\begin{smallmatrix}
    0&-1\\\ell^n&0
\end{smallmatrix}\right)\overline{K_0(\ell^n)}}\in C_c^\infty(\PGL_2(\Q_\ell))$ at the $\ell$-adic place, which has the property \begin{align*}
    \tr\pi(r_n) = \begin{cases}
        \epsi(\pi) &\text{if $n(\pi)= n$},\\
        0 &\text{if $n(\pi)\ne n$},
    \end{cases}
\end{align*}
for any infinite-dimensional irreducible admissible representation $\pi$ of $\PGL_2(\Q_\ell)$. We will use this approach for the Maass forms case.
\end{remark*}

\subsection{Normalization of Measures}
For the proofs of the remaining formulas we use the adelic trace formulas, which require integration over a number of different groups. Thus, we normalize the Haar measures here. The presentation is modeled on the normalization table in \cite{FinLapASTF}. In the table, $p$ denotes a prime number and $v$ any place of $\Q$.
\begin{center}
    \renewcommand{\arraystretch}{1.25}
    \begin{longtable}{l | p{20em}}
        Group & Measure normalization\\
        \midrule
        discrete & counting measure\\
        $\Q_p$, $\Q_p^\times$ & $\vol(\Z_p)=1$, $\vol(\Z_p^\times)=1$\\
        $\R$ &\text{Lebesgue measure}\\
        $\R^\times,\R_{>0}$ & $\frac{\mathrm d^{\mathrm{Leb}}x}{\abs x}$\\
        $\C^\times$ & $\frac{2\mathrm dx\mathrm dy}{x^2+y^2}$\\
        $L^\times$ with $L/\Q_p$ quadratic & $\vol(\O_L^\times)=1$\\
        $D^\times$ with $D/\Q_p$ quaternion algebra & $\vol(\O_D^\times)=1$\\
        $\GL_2(\Q_p)$ & $\vol(\GL_2(\Z_p))=1$\\
        $\PSO_2(\R)$ & $\vol(\PSO_2(\R))=1$\\
        $\PSU_2(\R)$ & $\vol(\PSU_2(\R))=1$\\
        $Z(\Q_v) = \left\{\begin{pmatrix}
            z&0\\0&z
        \end{pmatrix}\right\}\subset\GL_2(\Q_v)$ & \text{through $\Q_v^\times\ni x\mapsto \begin{pmatrix}
            x&0\\0&x
        \end{pmatrix}$}\\
        $N(\Q_v)=\left\{\begin{pmatrix}
            1&x\\0&1
        \end{pmatrix}\right\}\subset\GL_2(\Q_v)$ &through $\Q_v\ni x\mapsto \begin{pmatrix}
            1&x\\0&1
        \end{pmatrix}$\\
        $A(\Q_v) = \left\{\begin{pmatrix}
            *&0\\0&*
        \end{pmatrix}\right\}\subset\GL_2(\Q_v)$ & through $\Q_v^\times\times\Q_v^\times\ni(x, y)\mapsto \begin{pmatrix}
            x&0\\0&y
        \end{pmatrix}$\\
        $\PGL_2(\R)$ & compatible with Iwasawa decomposition\\
        adelic group & restricted direct product measure
    \end{longtable}
\end{center}
For any group not that is not listed here the measure is chosen to be compatible with quotients, e.g.\ for $\PGL_2(\Q_p)=\GL_2(\Q_p)/Z(\Q_p)$ the measure is chosen to be the quotient measure. Also $\A_\Q^{\times, 1}$ has measure so that the quotient measure on $\A^\times_\Q / \A_\Q^{\times, 1}\cong \R_{>0}$ (under $\abs\cdot$) is the one defined above.

\subsection{Proof of \cref{proposition:quat_TF}}
\subsubsection{Construction of the Local Test Function}\label{section:quat_test_function}
Let $p$ be an odd prime and let $F$ be a finite extension of $\Q_p$. Let $D$ be a division quaternion algebra over $F$ and $G = D^\times$, $\overline G = D^\times / F^\times$. Let $K_n$ be the image of $U_D^n$ in $\overline G$. Recall that for an irreducible admissible representation $\pi$ of $\overline G$ we let $n(\pi)\geq1$ denote its conductor exponent, which is the smallest $n\geq1$ for which $K_{n-1}$ acts trivially on $\pi$. For $n\geq3$, we need a test function $\varphi_n\in C_c^\infty(\overline G)$ such that \begin{align*}
    \tr\pi(\varphi_n) = \begin{cases}
        \epsi(\pi)&\text{if $n(\pi)=n$},\\
        0&\text{if $n(\pi)\ne n$}.
    \end{cases}
\end{align*}
Let $\psi$ be a nontrivial additive character on $F$ of level $1$, i.e.\ trivial on $\fp_F$, but not on $\O_F$. Let $\psi_D = \psi\circ\trd$ be the induced character on $D$.

For an integer $l\geq1$ let $D(l) = \{x\in D^\times : v_D(x) = -l\}$. Then by \cite[Theorem 55.3]{BusHenLL}\footnote{In \cite{BusHenLL} there is a missing $\psi_D(x)$ in the integral and on the right-hand side the $q^{2n+1}$ should be in the numerator rather than in the denominator.} for an irreducible representation $\pi$ of $\overline G$ of conductor $n = l+2$ we have \[
    \int_{D(n-2)}\pi(x)\psi_D(x)\mathrm dx = \vol(U_D^{n-1})q^{n-1}\epsi(\pi)\id_\pi.
\]
Note that here the root number $\epsi(\pi)$ is normalized so that it is the negative of the root number of the Jacquet--Langlands transfer of $\pi$. We use this to construct our test function. We need to make the function invariant under $F^\times$. Hence, we define $T =T_n: \overline G\to\C$ by \[
    T(gF^\times) = \begin{cases}
        0 & \text{if $n \not\equiv v_D(g)\mod 2$},\\
        \int_{\O_F^\times}\psi_D(a\varpi_F^{(-v_D(g)-n+2)/2}g)\mathrm da &\text{if $n\equiv v_D(g)\mod2$}.
    \end{cases}
\]
Here $\varpi_F$ denotes a uniformizer of $F$. Note that $v_D(g)$ is well-defined mod $2$ and $\varpi_F^{(-v_D(g)-n+2)/2}g$ is well-defined up to a unit which makes the integral well-defined. Then it is easy to see that for admissible irreducible representations $\pi$ of $\overline G$ we have $\pi(T) = \int_{D(l)}\pi(x)\psi_D(x)\mathrm dx$. We claim that if $n(\pi)\ne n$, then $\pi(T) = 0$. This is essentially \cite[2.3.5]{BusFroGS}. First note that $T$ is $K_{n-1}$-invariant, so $\pi(T)\subset\pi^{K_{n-1}}$. This shows that $\pi(T)=0$ if $n(\pi)>n$. If $n(\pi)<n$, it is also easy to see that $\pi(T)=0$ using $\int_{U_D^{n-2}/U_D^{n-1}}\psi_D(\varpi_D^{-n+2}g)\mathrm dg = 0$. Hence, we have \begin{align*}
    \tr\pi(T) = (\dim\pi)\vol(U_D^{n-1})q^{n-1}\begin{cases}
        \epsi(\pi)&\text{if $n(\pi)=n$},\\
        0&\text{if $n(\pi)\ne n$}.
    \end{cases}
\end{align*}
In fact, if $n=n(\pi)\ne2$, $n$ alone determines the dimension of $\pi$. By \cite[Remark 8]{AnaPatTPPQA} we have \[
    \dim \pi = d_n:=\begin{cases}
            2q^{(n-2)/2}&\text{$n$ even},\\
            (q+1)q^{(n-3)/2}&\text{$n$ odd}.
        \end{cases}
\]
This essentially follows from the explicit description of all the possible representations $\pi$ as inductions from tori. Therefore, $\varphi:=\varphi_n := d_n^{-1}\vol(U_D^{n-1})^{-1}q^{-n+1}T_n$ is the desired test function.

The volume of $U_D^{n-1}$ is easily computed and we obtain $\varphi_n = d_n^{-1}(q^2-1)q^{n-3}T_n$ where $q$ is the cardinality of the residue field of $F$. To use $\varphi$ in the global trace formula we need its orbital integrals. First note that $\varphi(1) = 0$. For $\gamma\in G-\{1\}$, its centralizer $G_\gamma$ is a torus associated to a quadratic extension $L$ of $F$. Note that we have normalized the Haar measure by $\vol(\O_L^\times)=1$. It is easy to see that $T$ is conjugation invariant, hence we have \begin{align*}
    O_\gamma(T) = \int_{G_\gamma\backslash G}T(g^{-1}\gamma g)\mathrm dg = T(\gamma)\vol(G_\gamma\lquot G).
\end{align*}
With our measure normalization we have \[\vol(G_\gamma\lquot G)=\begin{cases}
    1&\text{if $L/F$ ramified},\\
    2&\text{if $L/F$ unramified}.
\end{cases}\] Next we compute $T(\gamma)$. If $v_D(\gamma) \not\equiv n\mod 2$, then $T(\gamma)=0$, so assume $v_D(\gamma) \equiv n\mod 2$. It is easily seen that for $t\in F$ \begin{align*}
    \int_{\O_F^\times}\psi(at)\mathrm da = \begin{cases}
        1&\text{if $v(t)>0$},\\
        -\frac1{q-1}&\text{if $v(t)=0$},\\
        0&\text{if $v(t)<0$}.
    \end{cases}
\end{align*}
In our case we let $t=\varpi_F^{(-v_D(\gamma)-n+2)/2}\trd(\gamma)$. 
Putting things together we see that if $v(\Nrd(\gamma))\equiv n\mod 2$, then\begin{align*}
    O_\gamma(\varphi) =  d_n^{-1}(q^2-1)q^{n-3}\vol(G_\gamma\lquot G) \begin{cases}
        1&\text{if $2v(\trd\gamma)>v(\Nrd(\gamma))+n-2$},\\
        -\frac1{q-1}&\text{if $2v(\trd\gamma)=v(\Nrd(\gamma))+n-2$},\\
        0&\text{if $2v(\trd\gamma)<v(\Nrd(\gamma))+n-2$}.
    \end{cases}
\end{align*}

\subsubsection{The Rest of the Trace Formula}\label{section:quat_trace_formula}
Let $k$ be an even positive integer, $n\geq3$ and $\ell, P$ distinct primes with $\ell$ odd. Now let $D$ be the quaternion algebra over $\Q$ ramified at $\{\infty, \ell\}$. View $D^\times$ as an algebraic group $G$ and let $Z$ denote its center. Then let $\overline G = G / Z$. We define a test function $\varphi\in C_c^\infty(\overline G(\A_\Q))$ by $\varphi = \bigotimes_{v}\varphi_v$ with $\varphi_v\in C_c^\infty(\overline G(\Q_v))$ as follows.
\begin{itemize}
    \item For $v=\ell$ we let $\varphi_\ell$ be the function constructed in the previous section.
    \item For $v = P$ let $\varphi_v = \one_{Z(\Q_P)\backslash \GL_2(\Z_P)\left(\begin{smallmatrix}
        P&0\\0&1
    \end{smallmatrix}\right)\GL_2(\Z_P)Z(\Q_P)}$. This acts as the Hecke operator on the $P$-adic representations.
    \item For $v=p\ne\ell, P$ let $\varphi_v = \one_{Z(\Q_p)\backslash \GL_2(\Z_p)Z(\Q_p)}$.
    \item For $v = \infty$ we have $\overline G(\R)=D_\infty^\times/\R^\times\cong \PSU(2)=\SU(2)/\pm1$. Its irreducible representations are exactly the $\Sym^n(\C^2)$ with $n$ even, where $\C^2$ is the standard representation of $\SU(2)$. We say the representation $\Sym^{k-2}(\C^2)$ is of weight $k$. Fix an invariant inner product on $(V,\pi):=\Sym^{k-2}(\C^2)$. Let $v$ be a unit vector in $V$. Then let $\varphi_\infty(g):=(\dim\pi)\overline{\langle \pi(g)v,v\rangle}$ be the complex conjugate of a certain matrix coefficient of $\Sym^{k-2}(\C^2)$. With this normalization we have \begin{align*}
        \tr \pi_\infty(\varphi_\infty)=\begin{cases}
            1&\text{if $\pi_\infty\cong\Sym^{k-2}(\C^2)$},\\
            0&\text{otherwise.}
        \end{cases}
    \end{align*}
\end{itemize}
Since $\overline G(\Q)\lquot \overline G(\A_\Q)$ is compact, we have \begin{align*}
    L^2(\overline G(\Q)\lquot \overline G(\A_\Q)) = \bigoplus_{\substack{\text{automorphic $\pi$}}}\pi^{\oplus m(\pi)} = \bigoplus_{\substack{\text{automorphic $\pi$}\\\text{with $\dim\pi=\infty$}}}\pi^{\oplus m(\pi)} \oplus\bigoplus_{\substack{\text{Hecke character $\chi$}\\\text{with $\chi^2=1$}}}\chi.
\end{align*}
Here the quadratic Hecke characters are viewed as characters of $\overline G(\A_\Q)$ via the reduced norm. By the multiplicity one theorem \cite[p.\ 416]{KnaRogATF}, $m(\pi)=1$ for all automorphic $\pi$. Note that $\varphi_\ell$ acts as $0$ on the one-dimensional representations. Therefore, the spectral side of the simple trace formula reads \begin{align*}
    \tr \big(R(\varphi) \mid L^2(\overline G(\Q)\lquot \overline G(\A_\Q))\big)=(-1)^{k/2+1}\sum_{\pi\in\FF_{D, k}(\ell^n)}\epsi(\pi)P^{1/2}\lambda_P(\pi).
\end{align*}
Note that the local root number at $\infty$ is $(-1)^{k/2+1}$ (it is the negative of the root number of the Jacquet--Langlands transfer, the weight $k$ discrete series representation of $\PGL_2(\R)$, so that the global root numbers are the same). Also note that the one-dimensional characters do not contribute since $\chi_\ell(\varphi_\ell)=0$ for any quadratic Hecke character $\chi$. The geometric side is $J_{\id}(\varphi)+J_{\mathrm{ell}}(\varphi)$ where \begin{align*}
    J_\id(\varphi)&=\vol(\overline G(\Q)\lquot \overline G(\A_\Q))\varphi(1),\\
    J_{\mathrm{ell}}(\varphi)&=\int_{\overline G(\Q)\backslash \overline G(\A_\Q)}\sum_{\substack{\gamma\in \overline G(\Q)\\\text{$\gamma$ elliptic}}}\varphi(g^{-1}\gamma g)\mathrm dg.
\end{align*}
Note that $J_\id(\varphi)=0$ since e.g.\ $\varphi_P(1)=0$. We approach the elliptic terms as in \cite{KniLiTHO} and \cite{GJFGL2APV} by lifting elements in $\overline G(\Q)$ to elements in $G(\Q)$. First note that if $\wtilde\gamma\in \overline G(\Q)$ is such that $g\mapsto \varphi(g^{-1}\gamma g)$ is not identically $0$, then $\wtilde\gamma$ has exactly two lifts to $\gamma\in G(\Q)$ with $\det\gamma=P$. We may then write \begin{align*}
    J_{\mathrm{ell}}(\varphi)&=\frac12\sum_{\substack{\text{elliptic $[\gamma]\subset G(\Q)$}\\\det\gamma=P}}\int_{\overline {G_\gamma(\Q)}\backslash \overline G(\A_\Q)}\varphi(g^{-1}\gamma g)\mathrm dg.
\end{align*}
Here $\overline {G_\gamma(\Q)}\lquot \overline G(\Q) = Z(\Q)\lquot G_\gamma(\Q)$ (which does not equal $\overline G_\gamma(\Q)$ if $\tr\gamma=0$). The sum runs over elliptic conjugacy classes in $G(\Q)$ of determinant $P$. If $\gamma\in G(\Q)\subset D$ is elliptic, then $L = L_\gamma=\Q[\gamma]$ is a quadratic extension of $\Q$ and $G_\gamma(\Q) = L^\times, G_\gamma(\A_\Q)=\A_L^\times$ by \cite[Proposition 26.1]{KniLiTHO}. We then further write as in \cite[p.\ 215]{GJFGL2APV},\begin{align*}
     J_{\mathrm{ell}}(\varphi)&=\frac12\sum_{\substack{\text{elliptic $[\gamma]\subset G(\Q)$}\\\det\gamma=P}}\int_{{L_\gamma}^\times Z(\A_\Q)\backslash G(\A_\Q)}\varphi(g^{-1}\gamma g)\mathrm dg\\
     &=\frac12\sum_{\substack{\text{elliptic $[\gamma]\subset G(\Q)$}\\\det\gamma=P}}\vol({L_\gamma}^\times\A_\Q^\times\lquot \A_{L_\gamma}^\times)\int_{G_\gamma(\A_\Q)\backslash G(\A_\Q)}\varphi(g^{-1}\gamma g)\mathrm dg\\
     &=\frac12\sum_{\substack{\text{elliptic $[\gamma]\subset G(\Q)$}\\\det\gamma=P}}\vol({L_\gamma}^\times\A_\Q^\times\lquot \A_{L_\gamma}^\times)\int_{G_\gamma(\R)\backslash G(\R)}\varphi_\infty(g^{-1}\gamma g)\mathrm dg\int_{G_\gamma(\A_{\Q}^\infty)\backslash G(\A_{\Q}^\infty)}\varphi^\infty(g^{-1}\gamma g)\mathrm dg.
\end{align*}
To compute the archimedean orbital integral, let $\chi$ denote the character of the representation $(V,\pi)=\Sym^{k-2}(\C^2)$ of $\PSU(2) = Z(\R)\lquot G(\R)$. Also note that $L_\infty = L\otimes\R\cong \C$ since $L\hookrightarrow D$. By \cite[Exercise 7.6]{DeiEchHA} we have \begin{align*}
    \int_{Z(\R)\backslash G(\R)} \varphi_\infty(g^{-1}\gamma g)\mathrm dg= \overline{\chi(\gamma)}.
\end{align*}
Therefore,
\begin{align*}
    \int_{G_\gamma(\R)\backslash G(\R)}\varphi_\infty(g^{-1}\gamma g)\mathrm dg = \vol(Z(\R)\lquot L_\infty^\times)^{-1}\overline {\chi(\gamma)}.
\end{align*}
Note by our choice of measure on $L_\infty^\times\cong\C^\times$, the volume is $2\pi$. We have $\overline{\chi(\gamma)} = U_{k-2}\left(\frac{\trd\gamma}{2\sqrt{\nrd\gamma}}\right)$ by the well-known character formula for $\SU(2)$ \cite[Example 12.23]{HallLG}. 

For the volume of the adelic quotient we obtain \begin{align*}
    \vol(L^\times\A_\Q^\times\lquot \A_L^\times) = \frac{4\pi h_L}{w_L},
\end{align*}
where $h_L$ (resp.\ $w_L$) is the class number (resp.\ number of roots of unity) of $L$. This follows essentially from \cite[Theorem 4.3.2]{TatT}. Note that our measure normalization on $\A_L^\times$ is slightly different from \cite{TatT}. We also obtain an extra factor of $2$ due to the comparison of the induced measures on the norm $1$ ideles as in \cite[(4.31)]{KniCLSN}.

Finally, for the nonarchimedean orbital integral, we split off the $\ell$-adic place:\begin{align*}
    \int_{G_\gamma(\A_{\Q}^{\infty})\backslash G(\A_{\Q}^\infty)}\varphi^\infty(g^{-1}\gamma g)\mathrm dg
    &=O_{\gamma_\ell}(\varphi_\ell)\int_{G_\gamma(\A_{\Q}^{\{\infty,\ell\}})\backslash G(\A_{\Q}^{\{\infty,\ell\}})}\varphi^{\{\infty,\ell\}}(g^{-1}\gamma g)\mathrm dg.
\end{align*}
We computed $O_{\gamma_\ell}(\varphi_\ell)$ before. The remaining orbital integral is computed exactly as in \cite{KniLiTHO} and we obtain the following: If $n=2a$ is even, then \begin{align*}
    \tr &\big(R(\varphi) \mid L^2(\overline G(\Q)\lquot \overline G(\A_\Q)) \big) =J_{\mathrm{ell}}(\varphi)\\
    &= \sum_{\substack{t\in\Z, t^2<4P\\ L:=\Q(\sqrt{t^2-4P})\hookrightarrow D}}\frac{h_L}{w_L}U_{k-2}\left(\frac{t}{2\sqrt{P}}\right)O_{\gamma_{t,P},\ell}(\varphi_\ell)\sum_{\substack{\Z[\gamma_{t,P}]\subset\O\subset\O_L\\\text{$\O$ is maximal at $\ell$}}}[\O_L^\times : \O^\times]\frac{h(\O)}{h_L}.
\end{align*}
The sum runs over integers $t$ with $t^2<4P$ and such that $\Q(\sqrt{t^2-4P})$ embeds into $D$. Here $\gamma_{t,P}$ denotes an element $\gamma$ in $D^\times$ with $\trd\gamma=t,\nrd\gamma=P$. Note that in \cite{KniLiTHO} the group is $\GL_2$, and there is no ramified place $\ell$ as in our case. Thus, we only sum over orders away from $\ell$. The last sum ranges over orders $\O$ satisfying the indicated conditions. Using the explicit formula for $O_{\gamma_{t,P},\ell}(\varphi_\ell)$ this easily simplifies to the formula in \cref{proposition:quat_TF}. For example, $\frac{h_L}{w_L} [\O_L^\times : \O^\times]\frac{h(\O)}{h_L}=\frac{h(\O)}{\abs{\O^\times}}$ and in order for $O_{\gamma_{t,P},\ell}(\varphi_\ell)$ to be non-zero, we need $\ell\mid t$, which we assume in the following. Then $\Z[\gamma_{t,P}]$ will already be maximal at $\ell$ since $\ell\nmid t^2-4P$, hence the inner sum will just be $\frac12 H(t^2-4P)$. We can write the local volume term $\vol(G(\Q_\ell)_\gamma\lquot G(\Q_\ell))$ as $1-\legendre{d_{\Q(\sqrt{t^2-4P})}}{\ell}$. Note that this is $0$ if $\Q(\sqrt{t^2-4P})$ does not embed into $D$, hence we may remove the assumption $\Q(\sqrt{t^2-4P})\hookrightarrow D$ by inserting the term $1-\legendre{d_{\Q(\sqrt{t^2-4P})}}{\ell}=1-\legendre{-P}{\ell}$ which accounts for the local centralizer covolume term at the same time.

We note that the same computations in the odd case $n=2a+1$ lead to the formula \begin{align*}
    (-1)^{k/2+1}\sum_{\pi\in\FF_{D, k}(\ell^{2a+1})}&\epsi(\pi)P^{1/2}\lambda_P(\pi) \\
    &=\frac1{2\ell}\sum_{\substack{t\in\Z, t^2<4\ell P}}U_{k-2}\left(\frac{t}{2\sqrt{\ell P}}\right)c_{\ell^{a+1}}(t)H(t^2-4\ell P).
\end{align*}
This is exactly the formula in \cite[Proposition 2.1]{BKLMYMurms}.


\subsection{Proof of \cref{proposition:quat_count}}
This result is presumably well known. Since we are not aware of a reference, we include a proof sketch. The proof will be very similar to that of \cref{proposition:quat_TF}. We first construct a local test function that picks out representations of the given conductor.
\subsubsection{Construction of the Local Test Function}
Let $p$ be an odd prime and let $F$ be a finite extension of $\Q_p$ with residue field cardinality $q$. Let $D$ be a division quaternion algebra over $F$ and $G = D^\times$, $\overline G = D^\times /F^\times$. Let $n\geq3$. Let $e_n = \frac{1}{\vol(K_n)}\one_{K_n}$. Then we have $\pi(e_n)=0$ if and only if $\pi^{K_n}=0$ if and only if $n(\pi)\geq n+2$. So let \[
    \varphi_n = d_n^{-1}(e_{n-1} - e_{n-2}),
\]
where $d_n$ is the dimension of any conductor exponent $n$ representation.
Then \begin{align*}
    \tr\pi(\varphi_n) = \begin{cases}
        1&\text{if $n(\pi)=n$,}\\
        0&\text{if $n(\pi)\ne n$}.
    \end{cases}
\end{align*}

Next we compute the orbital integrals. Let the measures be normalized as before. For $\gamma\in G-\{1\}$, we have \begin{align*}
    O_\gamma(e_n) = \int_{G_\gamma\backslash G}e_n(g^{-1}\gamma g)\mathrm dg = e_n(\gamma)\vol(G_\gamma\lquot G),
\end{align*}
so\[
    O_\gamma(\varphi) = d_n^{-1}(e_{n-1}(\gamma)-e_{n-2}(\gamma))\vol(G_\gamma\lquot G).
\]
In particular, this is non-zero if and only if $\gamma\in K_{n-2}$. Note that \[
    \vol(K_n) = \frac{\vol(U^n_D)}{\vol(U_F^{\lceil n/2\rceil})}=(q+1)^{-1}q^{\lceil n/2\rceil-2n+1}.
\]
Hence we get \[
    O_\gamma(\varphi) = d_n^{-1}\vol(G_\gamma\lquot G)\begin{cases}
        (q+1)(q^{-\lceil (n-1)/2\rceil+2n-3}-q^{-\lceil (n-2)/2\rceil+2n-5})&\text{if $\gamma \in K_{n-1}$},\\
        -(q+1)q^{-\lceil (n-2)/2\rceil+2n-5}&\text{if $\gamma \in K_{n-2}\setminus K_{n-1}$},\\
        0&\text{if $\gamma \notin K_{n-2}$},
    \end{cases}
\] 
and \begin{align*}
    \varphi(1) = d_n^{-1}(q+1)(q^{-\lceil (n-1)/2\rceil+2n-3}-q^{-\lceil (n-2)/2\rceil+2n-5}).
\end{align*}

\subsubsection{The Rest of the Trace Formula}
The other local test functions are defined as before, except that now there is no prime $P$. Then we proceed exactly as before. The main term is now \[
    J_{\id}(\varphi) = \vol(\overline G(\Q)\lquot \overline G(\A_\Q))\varphi(1) = \vol(\overline G(\Q)\lquot \overline G(\A_\Q))(k-1)\varphi_\ell(1).
\]
To compute the volume choose any maximal order $\O$ of $D$ such that $\O_p^\times = \GL_2(\Z_p)$ under the identification $D_p^\times=\GL_2(\Q_p)$ for $p\ne\ell$. Then $\what\O^\times = \prod_{p}K_p=K^\infty$. Choose representatives $I$ of the right ideal classes in $\Cls(\O)$ and choose $\alpha_I\in G(\A_\Q^\infty)$ such that $I = \alpha_I\what\O\cap D$. Then \begin{align*}
    \overline G(\Q)\lquot\overline G(\A_\Q^\infty) / \overline{K^\infty}\simeq G(\Q)\lquot G(\A_\Q^\infty)/K^\infty=\bigsqcup_I G(\Q)\alpha_IK^\infty
\end{align*}
by \cite[(30.4.14)]{VoiQA}. Here $\overline{K^\infty}$ denotes the image of $K^\infty$ in $\overline G(\A_\Q^\infty)$. Therefore,\begin{align*}
    \overline G(\Q)\lquot\overline G(\A_\Q) = \bigsqcup_I\Gamma_I\lquot \overline G(\R)\alpha_I\overline{K^\infty},
\end{align*}
with $\Gamma_I = \overline{G}(\Q)\cap \alpha_I\overline{K^\infty}\alpha_I^{-1}\cong\O_{L}(I)^\times/\{\pm1\}$. Hence, we obtain $\vol(\overline G(\Q)\lquot \overline G(\A_\Q)) = \frac{\ell-1}{12}$ from the Eichler mass formula \cite[Theorem 25.1.1]{VoiQA}. 

The one-dimensional terms vanish again. Indeed, a one-dimensional representation with trivial central character has conductor exponent at most $2$ at $\ell$. Since here $n\geq3$, both $e_{n-1}$ and $e_{n-2}$ act as the identity and hence $\varphi_\ell$ has trace $0$. Therefore, the spectral side is $\sum_{\pi\in\FF_{D, k}(\ell^n)}1$.

The elliptic terms are handled as before, and we get:\begin{align*}
    \sum_{\pi\in\FF_{D, k}(\ell^n)}&1 \\
    &=\frac{\ell-1}{12}(k-1)\varphi_\ell(1)+\frac12\sum_{\substack{t\in\Z, t^2<4}}U_{k-2}\left(\frac{t}{2}\right)O_{\gamma_{t, 1}}(\varphi_\ell)H(t^2-4)\\
    &=\frac{\ell-1}{12}(k-1)d_n^{-1}(\ell+1)(\ell^{-\lceil (n-1)/2\rceil+2n-3}-\ell^{-\lceil (n-2)/2\rceil+2n-5})\\
    &=\frac1{1+\delta_{n\equiv0\mod2}}\frac{k-1}{12}(\ell-1)^2(\ell+1)\ell^{n-3}.
\end{align*}
Here the elliptic terms vanish; indeed, it is not difficult to see that $O_{\gamma}(\varphi_\ell)=0$ if $\nrd\gamma=1,\trd\gamma\in\{0,\pm1\}$, assuming $n\geq4$ or $n=3$ and $\ell>3$.

\subsection{Proof of \cref{proposition:maass_TF}}
The proof is similar in principle to the previous cases, but here we have more terms in the trace formula. The main difficulty is the case of even conductor exponent, where we will have to bound certain Eisenstein contributions.

We first identify $\phi\in H_\maass^\new(N)$ with cuspidal automorphic representations $\pi$ of $\PGL_2(\A_\Q)$ of conductor $N$ and with $\pi_\infty$ principal series. We describe a suitable test function on $\PGL_2(\A_\Q)$. Let $G = \GL_2, \overline G = \PGL_2$. $B$ denotes the algebraic subgroup of $G$ consisting of upper triangular matrices. For a finite prime $p$, let $K_p = \GL_2(\Z_p)$, and let $K_0(\ell^m)$ denote the usual subgroup \begin{align*}
    K_0(\ell^m) = \left\{\begin{pmatrix}
        a&b\\c&d
    \end{pmatrix}\in\GL_2(\Z_\ell)\,\Big\vert \,c\equiv0\mod\ell^m \right\}.
\end{align*}
Define $\varphi = \bigotimes_v\varphi_v:\overline G(\A_\Q) \to\C$, where we define the local factors as follows:
\begin{itemize}
    \item $v = \infty$. Then $\overline G(\R) = \PGL_2(\R)$. Let $\overline{K}_\infty^+ = \PSO(2)$.
    Let $j = \begin{pmatrix}
        1&0\\0&-1
    \end{pmatrix}$. Define $\varphi_\infty(g) = f(gj)$ for $g\in \PGL_2(\R) - \PGL_2^+(\R)$ and $\varphi_\infty = 0$ on $\PGL_2^+(\R)$. Then for a unitary irreducible representation $\pi_\infty$ of $\PGL_2(\R)$, we have \begin{align*}
        \tr\pi_\infty(\varphi_\infty) = \begin{cases}
        \epsi(\pi_\infty)h_f(t_{\pi_\infty})&\text{if $\pi_\infty^{\overline K_\infty^+}\ne0$},\\
        0&\text{if $\pi_\infty^{\overline K_\infty^+}=0$}.
        \end{cases}
    \end{align*}
    Here $t_{\pi_\infty}$ denotes the spectral parameter defined as follows. If $\pi_\infty^{\overline K_\infty}\ne0$, then $\pi_\infty = \Ind_{B(\R)}^{G(\R)}\chi$ for a character of the form $\chi\begin{pmatrix}a&*\\0&d\end{pmatrix} = (\sign(ad))^\delta\abs{\frac ad}^s$ with $\delta\in\{0,1\}$. Then $t_{\pi_\infty} = -is$, $\pi_\infty(f)$ maps into the $\overline{K}_\infty^+$-invariant line and acts on it via $h_f(t_{\pi_\infty})$ by \cite[Proposition 3.9]{KniLiKTF}. The matrix $j$ normalizes $\overline{K}_\infty^+$ and it acts on the $\overline{K}_\infty^+$-invariant line by multiplication by $(-1)^\delta$ which is $\epsi(\pi_\infty)$. This proves the above assertion.
    
    \item $v=\ell$. Let $\varphi_\ell=e_n = \vol(\overline{K_0(\ell^n)})^{-1}\one_{w_{\ell^n}\overline{K_0(\ell^n)}}\in C_c^\infty(\PGL_2(\Q_\ell))$, where $w_{\ell^n} = \begin{pmatrix}
        0&-1\\\ell^n&0
    \end{pmatrix}$. Then we have
    \begin{align*}
        \tr\pi_\ell(e_n) = \begin{cases}
            \epsi(\pi_\ell) &\text{if $n(\pi_\ell)\leq n$ and $n(\pi_\ell)\equiv n\mod 2$},\\
            0 &\text{otherwise},
        \end{cases}
    \end{align*}
    for any infinite-dimensional irreducible admissible representation $\pi_\ell$ of $\PGL_2(\Q_\ell)$. This follows from \cref{lemma:local_intertwiner_en}. Thus, to obtain exactly the conductor $\ell^n$ forms we will eventually consider $e_{n}-e_{n-2}$ where $e_{-1}=0$. 
    \item $v=p =  P$. In this case define $\varphi_P = \one_{\overline{K_P}\left(\begin{smallmatrix}
        P&0\\0&1
    \end{smallmatrix}\right)\overline{K_P}}$. Then for any infinite-dimensional irreducible representation $\pi_P$ of $\overline G(\Q_P)$:
        \begin{align*}
    \tr\pi_P(\varphi_P) = \begin{cases}
        \sqrt P \lambda_P(\pi_P)&\text{if $\pi_P$ is unramified},\\
        0&\text{otherwise.}
    \end{cases}
    \end{align*}
    \item $v=p\ne\ell, P$. In this case define $\varphi_p = \one_{\overline{K_p}}$. Then for any infinite-dimensional irreducible representation $\pi_p$ of $\overline G(\Q_p)$:
        \begin{align*}
    \tr\pi_p(\varphi_p) = \begin{cases}
        1&\text{if $\pi_p$ is unramified},\\
        0&\text{otherwise.}
    \end{cases}
    \end{align*}
\end{itemize}
We note that in any case if $\pi_p = \chi_p\circ\det$ is a one-dimensional representation of $G(\Q_p)$ and $\tr\pi_p(\varphi_p)\ne0$, then $\chi_p$ is unramified.

To emphasize the $n$-dependence, we will also write $\varphi^{(n)} = \varphi$. Let $\psi = \varphi^{(n)}-\varphi^{(n-2)}$. Note that $\psi=(e_{n}-e_{n-2})\otimes \bigotimes_{v\ne\ell}\varphi_v$. Then, putting the above together, we obtain \begin{align*}
    \tr\big( R(\psi)\mid L_\cusp^2(\overline G(\A_\Q))\big)= \sum_{\phi\in H_\maass^\new(\ell^n)}h_f(t_\phi)\epsi(\phi)\sqrt P \lambda_P(\phi).
\end{align*}

We apply the Arthur-Selberg trace formula from \cite{GJFGL2APV} and obtain \begin{align*}
    J_\cusp(\psi):=\tr\big( R(\psi)\mid L_\cusp^2(\overline G(\A_\Q))\big)=J_\id(\psi)-J_{\mathrm{one}}(\psi)+J_{\mathrm{ell}}(\psi)+J_{\mathrm{par}}(\psi)+J_{\mathrm{hyp}}(\psi)+J_{\mathrm{Eis}}(\psi).
\end{align*}
We will now define and compute all these terms. In particular we will show that $J_{\mathrm{ell}}(\psi)$ gives the main term in \cref{proposition:maass_TF}, while the remaining terms are $0$ if $n$ is odd and contribute to the error if $n$ is even. In the following, $m$ will denote either $n$ or $n-2$.

\subsubsection{The Identity Contribution}
Note that $\psi_\infty(1)=0$, so $J_\id(\psi) = \vol(\overline G(\Q)\lquot\overline G(\A_\Q))\psi(1)=0$.

\subsubsection{The One-Dimensional Contribution}
Here \begin{align*}
    J_{\mathrm{one}}(\varphi^{(m)}) = \sum_{\chi^2=1}\tr\chi(\varphi^{(m)}),
\end{align*}
where the sum runs over Hecke characters $\chi:\Q^\times\lquot \A_\Q^\times\to\C^\times$ of order $2$ which are viewed as characters of $\PGL_2(\A_\Q)$ via $\det$. As we noted, if $\pi_\chi(\varphi)\ne0$ for some Hecke character $\chi=\bigotimes_v\chi_v$ of order $2$, then $\chi_p$ is unramified for all finite $p$. Since $\chi$ is of order $2$, $\chi_\infty$ is trivial on $\R_{>0}$. Therefore, $\chi$ is trivial, in which case it is easily seen that $\chi(\varphi^{(m)}) = A(P+1)$ where $A =\int_{\overline G(\R)}f(g)\mathrm dg$.\footnote{In fact, $A = h_f(i/2)$ by an easy computation, but we will not need this.} Hence, \begin{align*}
    J_{\mathrm{one}}(\varphi^{(m)})= A(P+1),
\end{align*}
and therefore for $n\geq2$,\begin{align*}
     J_{\mathrm{one}}(\psi)= 0.
\end{align*}

\subsubsection{Elliptic Contribution}\label{section:maass_elliptic_1}
We have: \begin{align*}
    J_{\mathrm{ell}}(\varphi^{(m)}) &= \int_{\overline G(\Q)\backslash \overline G(\A_\Q)}\sum_{\gamma\in\overline G(\Q)\text{ elliptic}}\varphi^{(m)}(g^{-1}\gamma g)\mathrm dg.
\end{align*}

We treat this contribution as before in \cref{section:quat_trace_formula}. First note that any elliptic conjugacy class for which $g\mapsto \varphi^{(m)}(g^{-1}\gamma g)$ is not identically zero has a representative $\gamma$ in $\GL_2(\Q)$, unique up to sign, of determinant $-\ell^mP$. In particular such elliptic classes become hyperbolic over $\R$. We may then write the elliptic contribution as \begin{align*}\frac12\sum_{\substack{\text{elliptic $[\gamma]\subset G(\Q)$}\\\det\gamma=-\ell^mP}}\vol({L_\gamma}^\times\A_\Q^\times\lquot \A_{L_\gamma}^\times)\int_{G_\gamma(\R)\backslash G(\R)}\varphi^{(m)}_\infty(g^{-1}\gamma g)\mathrm dg\int_{G_\gamma(\A_{\Q}^\infty)\backslash G(\A_{\Q}^\infty)}\varphi^{(m)\infty}(g^{-1}\gamma g)\mathrm dg.
\end{align*}

\begin{lemma}\label{lemma:arch_hyperbolic_orbital_integral}
    Let $\gamma\in G(\R)$ be hyperbolic with $\det\gamma<0$ and let $u_\gamma = -\frac{(\tr\gamma)^2}{\det\gamma}$. Then \begin{align*}
        \int_{G_\gamma(\R)\backslash G(\R)} \varphi_\infty(g^{-1}\gamma g)\mathrm dg = \frac{1}{\sqrt{u_\gamma+4}}Q_f(u_\gamma).
    \end{align*}
\end{lemma}
\begin{proof}
    Without loss of generality we may assume $\gamma=\begin{pmatrix}
        y^{1/2}&0\\0&-y^{-1/2}
    \end{pmatrix}$. Let $t = y^{1/2}-y^{-1/2}$, so that $u_\gamma = t^2$. Then $G_\gamma(\R)$ is the diagonal torus, and using the $\SO(2)$-biinvariance of $f$, the Iwasawa decomposition gives \begin{align*}
        \int_{G_\gamma(\R)\backslash G(\R)} \varphi_\infty(g^{-1}\gamma g)\mathrm dg &=\int_{N(\R)} \varphi_\infty(n^{-1}\gamma n)\mathrm dn\\
        &=\int_{\R} f\begin{pmatrix}
            y^{1/2}&-x(y^{1/2}+y^{-1/2})\\
            0&y^{-1/2}
        \end{pmatrix}\mathrm dx\\
        &=\int_{\R}V_f(t^2+x^2(t^2+4))\mathrm dx\\
        &=\frac{1}{\sqrt{t^2+4}}\int_{\R}V_f(t^2+x^2)\mathrm dx\\
        &=\frac{1}{\sqrt{t^2+4}}Q_f(t^2).
    \end{align*}
\end{proof}
So in our case the archimedean orbital integral is\begin{align*}
    \frac1{\sqrt{\frac{t^2}{\ell^mP}+4}}Q_f\left(\frac{t^2}{\ell^mP}\right),
\end{align*}
where $t = \tr\gamma$. Since now $t^2+4\ell^mP>0$, $L=L_\gamma$ is real quadratic, so the volume term computes to \begin{align*}
    \vol({L_\gamma}^\times\A_\Q^\times\lquot \A_{L_\gamma}^\times)=\frac{8h_LR_L}{w_L},
\end{align*}
where $R_L$ denotes the regulator. 

For the $\ell$-adic orbital integral we have \begin{align*}
    \int_{G_\gamma(\Q_\ell)\backslash G(\Q_\ell)} e_m(g^{-1}\gamma g)\mathrm dg &= \begin{cases}
        \vol(\Z_\ell[\gamma]^\times)^{-1} & \text{if $\ell^m\mid t$,}\\
        0&\text{otherwise}.
    \end{cases}
\end{align*}
Indeed, as in \cite{KniLiTHO} we may identify $G_\gamma(\Q_\ell)\backslash G(\Q_\ell) / K_0(\ell^m)$ with pairs $(\Lambda,\Lambda')$ of lattices in $L_\ell$ with $\Lambda / \Lambda'\cong \Z/\ell^m\Z$ up to homothety. The support condition of $e_m$ then translates to $\gamma \Lambda = \Lambda'$ and $\gamma\Lambda'=\ell^m\Lambda$ and this forces $\ell^m\mid t$. This is the main difference from the computation in \cite{KniLiKTF} because of the presence of the Atkin-Lehner involution in our case.  Then the same arguments as in \cite{KniLiTHO} give \begin{align*}
    \int_{G_\gamma(\Q_\ell)\backslash G(\Q_\ell)} e_m(g^{-1}\gamma g)\mathrm dg &= \sum_{\substack{\Z_\ell[\gamma]\subset\O\subset\O_{L_\gamma}\\ \O/\gamma\O\cong\Z/\ell^m\Z}}\vol(\O^\times)^{-1},
\end{align*}
where the sum runs over orders $\O$. Finally, it is easy to see that the only order $\O$ containing $\Z_\ell[\gamma]$ for which $\O/\gamma\O\cong\Z/\ell^m\Z$ is $\O=\Z_\ell[\gamma]$.

The nonarchimedean orbital integral away from $\ell$ is then exactly as in \cite{KniLiTHO} and putting it together with the integral at $\ell$, we obtain \begin{align*}
    \int_{G_\gamma(\A_{\Q}^\infty)\backslash G(\A_{\Q}^\infty)}\varphi^{(m)\infty}(g^{-1}\gamma g)\mathrm dg = \sum_{\substack{\Z[\gamma]\subset\O\subset\O_{L}\\ \Z_\ell[\gamma]=\O_\ell}}[\O_L^\times : \O^\times]\frac{h(\O)}{h(\O_L)}.
\end{align*}

Putting everything together gives:\begin{align*}
    J_{\mathrm{ell}}(\varphi^{(m)}) &= \int_{\overline G(\Q)\lquot \overline G(\A_\Q)}\sum_{\gamma\in\overline G(\Q)\text{ elliptic}}\varphi^{(m)}(x^{-1}\gamma x)\mathrm dx\\
    &=\sum_{\substack{t\in\Z\\ \ell^m\mid t,\,t^2 + 4\ell^mP\notin\Q^{\times2}\\ L:= \Q(\sqrt{t^2 + 4\ell^mP})}}\frac{4 h_LR_L}{w_L}\frac1{\sqrt{\frac{t^2}{\ell^mP}+4}}Q_f\left(\frac{t^2}{\ell^mP}\right)\sum_{\substack{\Z[\gamma_{t, -\ell^mP}]\subset\O\subset\O_{L}\\ \Z_\ell[\gamma_{t, -\ell^mP}]=\O_\ell}}[\O_L^\times : \O^\times]\frac{h(\O)}{h(\O_L)}.
\end{align*}
Here $\gamma_{t, -\ell^mP}$ denotes any solution of $x^2-tx-\ell^mP=0$. Note that the condition $t^2+4\ell^mP\notin\Q^{\times2}$ is automatic if $m$ is odd. If we write $t^2+4\ell^mP=d_Lf^2$, then \begin{align*}
    &\frac{4 h_LR_L}{w_L}\sum_{\substack{\Z[\gamma_{t, -\ell^mP}]\subset\O\subset\O_{L}}}[\O_L^\times : \O^\times]\frac{h(\O)}{h(\O_L)} \\
    &= L(1,\psi_{d_L})\sqrt{d_L}\sum_{c\mid f}c\prod_{p\mid c}\left(1-\legendre{d_L}{p}\frac1p\right)\\
    &=L(1,\psi_{d_L})\sqrt{d_L}\prod_{p\mid f}\left(1 + (p - \legendre{d_L}{p})\frac{p^{v_p(f)}-1}{p-1}\right)\\
    &=L(1,\psi_{t^2 + 4\ell^mP})\sqrt{d_Lf^2}.
\end{align*}
We may write the sum over orders $\O$ with $\O_\ell =\Z_\ell[\gamma_{t, -\ell^mP}]$ as the difference of the sum over orders with $\O_\ell \supset\Z_\ell[\gamma_{t, -\ell^mP}]$ and the sum over orders with $\O_\ell \supset\Z_\ell[\gamma_{\ell^{-1}t, -\ell^{m-2}P}]=\O_\ell$. We first assume that $m = 2b+1$ is odd. Then this gives \begin{align*}
    &\frac{4 h_LR_L}{w_L}\sum_{\substack{\Z[\gamma_{t, -\ell^mP}]\subset\O\subset\O_{L}\\ \Z_\ell[\gamma_{t, -\ell^mP}]=\O_\ell}}[\O_L^\times : \O^\times]\frac{h(\O)}{h(\O_L)}\\
    &= L(1,\psi_{t^2 + 4\ell^{m}P})\sqrt{t^2+4\ell^{m}P}-L(1,\psi_{\ell^{-2}t^2 + 4\ell^{m-2}P})\sqrt{\ell^{-2}t^2+4\ell^{m-2}P}.
\end{align*}
After some algebra, this simplifies to \begin{align*}
    L(1,\psi_{t^2 + 4\ell^{m}P})\sqrt{t^2+4\ell^{m}P}\left(1-\frac{\ell^{b}-1}{\ell^{b+1}-1}\right).
\end{align*}
Now for $\psi = \varphi^{(n)}-\varphi^{(n-2)}$ with $n=2a+1$ we obtain \begin{align*}
    J_{\mathrm{ell}}(\psi)&=\sum_{\phi\in H_\maass^\new(\ell^n)}h_f(t_\phi)\epsi(\phi)\sqrt P \lambda_P(\phi)\\
    &=\sum_{\substack{t\in\Z\\ \ell^n\mid t}}\frac1{\sqrt{\frac{t^2}{\ell^nP}+4}}Q_f\left(\frac{t^2}{\ell^nP}\right)L(1,\psi_{t^2 + 4\ell^{n}P})\sqrt{t^2+4\ell^{n}P}\left(1-\frac{\ell^{a}-1}{\ell^{a+1}-1}\right)\\
    &\hspace{1em}-\sum_{\substack{t\in\Z\\ \ell^{n-2}\mid t}}\frac1{\sqrt{\frac{t^2}{\ell^{n-2}P}+4}}Q_f\left(\frac{t^2}{\ell^{n-2}P}\right)L(1,\psi_{t^2 + 4\ell^{n-2}P})\sqrt{t^2+4\ell^{n-2}P}\left(1-\frac{\ell^{a-1}-1}{\ell^{a}-1}\right)\\
    &=\sqrt{\ell^nP}\sum_{\substack{t\in\Z\\ \ell^n\mid t}}Q_f\left(\frac{t^2}{\ell^nP}\right)L(1,\psi_{t^2 + 4\ell^{n}P})\left(1-\frac{\ell^{a}-1}{\ell^{a+1}-1}\right)\\
    &\hspace{1em}-\sqrt{\ell^{n-2}P}\sum_{\substack{t\in\Z\\ \ell^{n-2}\mid t}}Q_f\left(\frac{t^2}{\ell^{n-2}P}\right)L(1,\psi_{t^2 + 4\ell^{n-2}P})\left(1-\frac{\ell^{a-1}-1}{\ell^{a}-1}\right).
\end{align*}
In the second sum we change variables $t\to t/\ell$ and obtain \begin{align*}
    &\sum_{\phi\in H_\maass^\new(\ell^n)}h_f(t_\phi)\epsi(\phi)\sqrt P \lambda_P(\phi)\\
    &=\sqrt{\ell^{n-2}P}\sum_{\substack{t\in\Z\\ \ell^{n-1}\mid t}}Q_f\left(\frac{t^2}{\ell^nP}\right)\Bigg(\ell\one_{\ell^n\mid t} L(1,\psi_{t^2 + 4\ell^{n}P})\left(1-\frac{\ell^{a}-1}{\ell^{a+1}-1}\right)\\
    &\hspace{6em}-L(1,\psi_{\ell^{-2}(t^2 + 4\ell^{n}P)})\left(1-\frac{\ell^{a-1}-1}{\ell^{a}-1}\right)\Bigg)\\
    &=\sqrt{\ell^{n}P}\sum_{\substack{t\in\Z\\ \ell^{n-1}\mid t}}Q_f\left(\frac{t^2}{\ell^nP}\right)\Bigg(\one_{\ell^n\mid t} L(1,\psi_{t^2 + 4\ell^{n}P})\left(1-\frac{\ell^{a}-1}{\ell^{a+1}-1}\right)\\
    &\hspace{6em}-L(1,\psi_{t^2 + 4\ell^{n}P})\frac{\ell^{a}-1}{\ell^{a+1}-1}\left(1-\frac{\ell^{a-1}-1}{\ell^{a}-1}\right)\Bigg)\\
    &=\sqrt{\ell^{n}P}\ell^{-a-1}\frac{\ell-1}{\ell^{a+1}-1}\sum_{\substack{t\in\Z}}Q_f\left(\frac{t^2}{\ell^nP}\right)L(1,\psi_{t^2 + 4\ell^{n}P})c_{\ell^n}(t).
\end{align*}
Now change variables to $\ell^as = t$ and we get \begin{align*}
    &\sum_{\phi\in H_\maass^\new(\ell^n)}h_f(t_\phi)\epsi(\phi)\sqrt P \lambda_P(\phi)\\
    &=\sqrt{\ell^{n}P}\ell^{-1}\frac{\ell-1}{\ell^{a+1}-1}\sum_{\substack{s\in\Z}}Q_f\left(\frac{s^2}{\ell P}\right)L(1,\psi_{\ell^{2a}(s^2 + 4\ell P)})c_{\ell^{a+1}}(s)\\
    &=\sqrt{\frac{P}{\ell}}\sum_{\substack{s\in\Z}}Q_f\left(\frac{s^2}{\ell P}\right)L(1,\psi_{s^2 + 4\ell P})c_{\ell^{a+1}}(s).
\end{align*}
Now we consider the case of even $m = 2b$. Similarly as before we obtain \begin{align*}
    &\frac{4 h_LR_L}{w_L}\sum_{\substack{\Z[\gamma_{t, -\ell^mP}]\subset\O\subset\O_{L}\\ \Z_\ell[\gamma_{t, -\ell^mP}]=\O_\ell}}[\O_L^\times : \O^\times]\frac{h(\O)}{h(\O_L)}\\
    &= L(1,\psi_{t^2 + 4\ell^{m}P})\sqrt{t^2+4\ell^{m}P}-L(1,\psi_{\ell^{-2}t^2 + 4\ell^{m-2}P})\sqrt{\ell^{-2}t^2+4\ell^{m-2}P}\\
    &=L(1,\psi_{t^2 + 4\ell^{m}P})\sqrt{t^2+4\ell^{m}P}\frac{\left(\ell-\legendre{d_L}{\ell}\right)(\ell^b-\ell^{b-1})}{\ell-1+\left(\ell-\legendre{d_L}{\ell}\right)(\ell^{b}-1)}.
\end{align*}
Note the character appears since in the even case we have $\ell\nmid d_L$. Therefore \begin{align*}
    J_{\mathrm{ell}}(\varphi^{(m)}) = \sqrt{\ell^mP}\sum_{\substack{t\in\Z\\ \ell^m\mid t,(t^2+4\ell^mP)\notin\Q^{\times2}\\ L:= \Q(\sqrt{t^2 + 4\ell^mP})}}Q_f\left(\frac{t^2}{\ell^mP}\right)L(1,\psi_{t^2 + 4\ell^{m}P})\frac{\left(\ell-\legendre{d_L}{\ell}\right)(\ell^b-\ell^{b-1})}{\ell-1+\left(\ell-\legendre{d_L}{\ell}\right)(\ell^{b}-1)}.
\end{align*}
Now change variables to $t = \ell^bs$. Then $t^2+4\ell^mP = \ell^m(s^2+4P)$. We still have $\ell^b\mid s$, hence $s^2+4P\equiv 4P\mod \ell$ and we obtain $\legendre{d_L}{\ell} = \chi_\ell(P):=\legendre{P}{\ell}$. Then \begin{align*}
     J_{\mathrm{ell}}(\varphi^{(m)}) &= \sqrt{\ell^mP}\sum_{\substack{s\in\Z\\ \ell^b\mid s,\,(s^2+4P)\notin\Q^{\times2}}}Q_f\left(\frac{s^2}{P}\right)L(1,\psi_{\ell^m(s^2 + 4P)})\frac{\left(\ell-\chi_\ell(P)\right)(\ell^b-\ell^{b-1})}{\ell-1+\left(\ell-\chi_\ell(P)\right)(\ell^{b}-1)}\\
     &=\sqrt{P}\ell^{b-1}(\ell-\chi_\ell(P))\sum_{\substack{s\in\Z\\ \ell^b\mid s,\,(s^2+4P)\notin\Q^{\times2}}}Q_f\left(\frac{s^2}{P}\right)L(1,\psi_{s^2 + 4P}).
\end{align*}
Then for $n = 2a$, we get \begin{align*}
    J_{\mathrm{ell}}(\psi) &=\sqrt{P}\ell^{a-1}(\ell-\chi_\ell(P))\sum_{\substack{s\in\Z\\ \ell^a\mid s,\,(s^2+4P)\notin\Q^{\times2}}}Q_f\left(\frac{s^2}{P}\right)L(1,\psi_{s^2 + 4P})\\
    &\hspace{3em}-\sqrt{P}\ell^{a-2}(\ell-\chi_\ell(P))\sum_{\substack{s\in\Z\\ \ell^{a-1}\mid s,\,(s^2+4P)\notin\Q^{\times2}}}Q_f\left(\frac{s^2}{P}\right)L(1,\psi_{s^2 + 4P})\\
    &=\sqrt P \frac{\ell-\chi_\ell(P)}{\ell}\sum_{\substack{s\in\Z\\(s^2+4P)\notin\Q^{\times2}}}Q_f\left(\frac{s^2}{P}\right)L(1,\psi_{s^2 + 4P})c_{\ell^a}(s).
\end{align*}

\subsubsection{Parabolic Contribution}
Here \begin{align*}
    J_{\mathrm{par}}(\varphi^{(m)}) = f.p.\int_{\A^\times}\int_K\varphi^{(m)}\left(k^{-1}\begin{pmatrix}
        1&a\\0&1
    \end{pmatrix}k\right)\mathrm dk\abs a^s\mathrm da.
\end{align*}
where $f.p.$ denotes the constant term in the Laurent series expansion around $s=1$. One easily sees by looking either at the infinite or the $P$-adic place that the integrand is identically $0$, hence the parabolic contribution is $0$. 

\subsubsection{Hyperbolic Contribution}\label{section:hyperbolic_maass1}
Here \begin{align*}
    J_{\mathrm{hyp}}(\varphi) = -\frac12\vol(\Q^\times\lquot \A_\Q^1)\int_K\int_{N(\A)}\sum_{\alpha\in\Q, \,\alpha\ne 1}\varphi\left(k^{-1}n^{-1}\begin{pmatrix}
        \alpha&0\\0&1
    \end{pmatrix}nk\right)\log H(wnk)\mathrm dn\mathrm dk.
\end{align*}
Here $w$ denotes $\begin{pmatrix}
    0&-1\\1&0
\end{pmatrix}$, $K$ is the standard maximal compact subgroup $O(2)\times \GL_2(\what \Z)$ and $H = \bigotimes_v H_v$ where $H_v\left(\begin{pmatrix}
    a_v&*\\0&b_v
\end{pmatrix}k_v\right) =\fabs{\frac{a_v}{b_v}}$. We distinguish the cases $m$ even and $m$ odd.

First assume $m$ is odd. We claim that the integrand is identically $0$. Indeed, suppose $\overline\gamma\in \PGL_2(\Q_\ell)$ is hyperbolic and $e_m(\overline\gamma)\ne0$. We may lift $\overline\gamma$ to $\gamma\in \GL_2(\Q_\ell)$ with $v_\ell(\det\gamma)= m$. Let $t = \tr\gamma$. Then $v_\ell(t)\geq m$ by considering the support of $e_m$. Since $\gamma$ is hyperbolic, $t^2-4\det(\gamma)$ is a square in $\Q_\ell$, but $v_\ell(t^2-4\det(\gamma))=m$, which is a contradiction if $m$ is odd. Hence, the integrand vanishes identically.

Now assume $m$ is even. Note that with our measure normalization, $\vol(\Q^\times\lquot \A^1)=1$. We may write \begin{align*}
    J_{\mathrm{hyp}}(\varphi)=-\frac12\sum_{u}\sum_{\alpha\ne1}J_{\alpha,u}(\varphi_u)\prod_{v\ne u}O_{\alpha,v}(\varphi_v),
\end{align*}
where the sum (resp.\ product) runs over all places $u$ (resp.\ all places $v\ne u$), and
\begin{align*}
    O_{\alpha,v}(\varphi_v)&=\int_{K_v}\int_{N(\Q_v)}\varphi_v\left(k_v^{-1}n_v^{-1}\begin{pmatrix}
    \alpha&0\\0&1
\end{pmatrix}n_vk_v\right)\mathrm dn_v\mathrm dk_v,\\
J_{\alpha,u}(\varphi_u)&=\int_{K_u}\int_{N(\Q_u)}\varphi_u\left(k_u^{-1}n_u^{-1}\begin{pmatrix}
    \alpha&0\\0&1
\end{pmatrix}n_uk_u\right)\log H(wn_uk_u)\mathrm dn_u\mathrm dk_u.
\end{align*}
First note that if $\alpha\ne1$ is such that the integrand is not identically $0$, then there is $z\in \Q_\ell$ such that $g=z\begin{pmatrix}
    \alpha&0\\0&1
\end{pmatrix}$ has $v_\ell(\det g)=m$ and $v_\ell(\tr g)\geq m$, which forces $v_\ell(\alpha)=0$. Similarly, one considers the other places and finds $\alpha=-P$ or $\alpha=-P^{-1}$, the minus sign being forced by the support at $\infty$. Both give the same values since the resulting elements are conjugate, so we assume $\alpha=-P$. We now compute the orbital integrals.
\begin{itemize}
    \item $\infty$. Let $u_\alpha = \frac{(1+\alpha)^2}{-\alpha}$. Then, by \cref{lemma:arch_hyperbolic_orbital_integral} (note $\alpha<0$ here)\begin{align*}
        O_{\alpha,\infty}(\varphi_\infty)=
            \frac{1}{\sqrt{u_\alpha+4}}Q_f\left(u_\alpha\right).
        \end{align*}
        For the weighted orbital integral the same computation as in \cref{lemma:arch_hyperbolic_orbital_integral} gives:\begin{align*}
            J_{\alpha,\infty}(\varphi_\infty)=
            -\frac{1}{\sqrt{u_\alpha+4}}\int_\R V_f(u_\alpha+x^2)\log\left(1+\frac{x^2}{u_\alpha+4}\right)\mathrm dx.
        \end{align*}
\item $p$-adic places. First note that if $n = \begin{pmatrix}
    1&x\\0&1
\end{pmatrix}$, then $H(wn) = \begin{cases}
    \abs x^{-2}&\text{if $v_p(x)<0$},\\
    1&\text{if $v_p(x)\geq0$}.
\end{cases}$\begin{itemize}
    \item $p\ne \ell$. Here we also allow $p=P$. By \cite[p.\ 284]{KniLiTHO} the unweighted orbital integral is $\abs{\alpha-1}^{-1}$. Let $b = v_p(\alpha-1)$. For the weighted orbital integral we obtain \begin{align*}
        J_{\alpha,p}(\varphi_p)&=-2\int_{\Q_p-\Z_p}\varphi_p\left(\begin{pmatrix}\alpha&(\alpha-1)x\\
            0&1
        \end{pmatrix}\right)\log \abs x \mathrm dx\\
        &=-2\int_{\{x\in\Q_p\colon -b\leq v_p(x)< 0\}}\log \abs x\mathrm dx\\
        &=-2\sum_{k=1}^{b}\int_{\{p^{-k}\Z_p^\times\}}\log \abs x\mathrm dx\\
        &=-2\sum_{k=1}^{b}\log(p^k)p^k\frac{p-1}{p}\\
        &=-2(\log p)\left(bp^b-\frac{p^b-1}{p-1}\right).
    \end{align*}
    Note that if $p=P$, then the term is $0$.
    \item $p=\ell$. If $A$ runs over a set of coset representatives for $K_\ell/K_0(\ell^m)$, then \begin{align*}
        O_{\alpha,\ell}(e_m)&= \vol(\overline{K_0(\ell^m)})\sum_A\int_{N(\Q_\ell)}\varphi_\ell\left(A^{-1}n_\ell^{-1}\begin{pmatrix}
    \alpha&0\\0&1
\end{pmatrix}n_\ell A\right)\mathrm dn_\ell\\
&= \sum_A\int_{N(\Q_\ell)}\one_{w_{\ell^m}\overline{K_0(\ell^m)}}\left(A^{-1}n_\ell^{-1}\begin{pmatrix}
    \alpha&0\\0&1
\end{pmatrix}n_\ell A\right)\mathrm dn_\ell.
    \end{align*}
    And similarly for the weighted orbital integral. From \cite[Lemma 13.1]{KniLiTHO} we see that we may take the representatives $\begin{pmatrix}
        \delta&1\\1&0
    \end{pmatrix}$ with $\delta\in\Z/\ell^m\Z$ and $\begin{pmatrix}
        1&0\\\tau&1
    \end{pmatrix}$ with $\tau\in\ell\Z/\ell^m\Z$. It is easily seen that for the former class of representatives the integrand is identically $0$. Thus we consider $A=\begin{pmatrix}
        1&0\\\tau&1
    \end{pmatrix}$. Then the contribution is \begin{align*}
        \int_{\Q_\ell}\one_{w_{\ell^m}\overline{K_0(\ell^m)}}\left(\begin{pmatrix}
            \alpha+(\alpha-1)x\tau&(\alpha-1)x\\
            -(\alpha-1)\tau(1+x\tau)&1-(\alpha-1)x\tau
        \end{pmatrix}\right)\mathrm dx.
    \end{align*}
    Write $m=2j$ and $\beta = \alpha-1$. It is easily seen that the integrand is non-zero if and only if \begin{align*}
        v_\ell(\alpha+\beta x\tau)\geq j, \quad v_\ell(1-\beta x\tau)\geq j, \quad v_\ell(\beta \tau(1+x\tau))= j,\quad v_\ell(\beta x)= -j.
    \end{align*}
    By combining the first two we see that we need $v_\ell(\alpha+1)\geq j$, so $v_\ell(\beta)=0$ and then $v_\ell(x)=-j$. Then it is only possible for $v_\ell(\alpha+\beta x\tau)$ to be $\geq j$ if $v_\ell(\tau)=j$. In this case the integrand is $1$ for $x\in (\beta\tau)^{-1}+\Z_\ell$ and $0$ otherwise, so the integral is $1$, assuming $v_\ell(\alpha+1)\geq j$ and $v_\ell(\tau)=j$. There are $\varphi(\ell^j)=\ell^j-\ell^{j-1}$ many such $\tau$ modulo $\ell^m$, hence \begin{align*}
        O_{-P,\ell}(e_m) = \varphi(\ell^j)\one_{\ell^j\mid P-1},
    \end{align*}
    and similarly \begin{align*}
        J_{-P,\ell}(e_m) = -2j(\log\ell)\varphi(\ell^j)\one_{\ell^j\mid P-1}.
    \end{align*}
    Therefore, if $n=2a$,\begin{align*}
        O_{-P,\ell}(\psi_\ell) &= O_{-P,\ell}(e_n)-O_{-P,\ell}(e_{n-2}) = \frac{\ell-1}{\ell}c_{\ell^{a}}(P-1),\\
        J_{-P,\ell}(\psi_\ell) &= J_{-P,\ell}(e_n)-J_{-P,\ell}(e_{n-2}) = -2(\log\ell)\left(a\varphi(\ell^a)\one_{\ell^a\mid P-1}-(a-1)\varphi(\ell^{a-1})\one_{\ell^{a-1}\mid P-1}\right).
    \end{align*}
\end{itemize}
\end{itemize}
We now put all the orbital integrals together and obtain \begin{align*}
    J_{\mathrm{hyp}}(\psi)=&J_{\alpha,\infty}(\varphi_\infty)\frac{\ell-1}{\ell}c_{\ell^{a}}(P-1)\prod_{p\ne\ell}\abs{P+1}_p^{-1}\\
    &+O_{\alpha,\infty}(\varphi_\infty)J_{-P,\ell}(\psi_\ell)\prod_{p\ne\ell}\abs{P+1}_p^{-1}\\
    &+O_{\alpha,\infty}(\varphi_\infty)\frac{\ell-1}{\ell}c_{\ell^{a}}(P-1)\sum_{q\ne \ell}J_{-P, q}(\varphi_q)\prod_{p\ne q,\ell}\abs{P+1}_p^{-1}.
\end{align*}
Note that $c_{\ell^a}(P-1)=J_{-P,\ell}(\psi_\ell)=0$ if $\ell\mid P+1$, so by the product formula, \begin{align*}
    J_{\mathrm{hyp}}(\psi)=&J_{\alpha,\infty}(\varphi_\infty)\frac{\ell-1}{\ell}c_{\ell^{a}}(P-1)(P+1)\\
    &+O_{\alpha,\infty}(\varphi_\infty)J_{-P,\ell}(\psi_\ell)(P+1)\\
    &+O_{\alpha,\infty}(\varphi_\infty)\frac{\ell-1}{\ell}c_{\ell^{a}}(P-1)\sum_{q\ne \ell}J_{-P, q}(\varphi_q)(P+1)\abs{P+1}_q.
\end{align*}
If we write $m_q = v_q(P+1)$, then \begin{align*}
    \sum_{q\ne \ell}J_{-P, q}(\varphi_q)(P+1)\abs{P+1}_q &= -2(P+1)\sum_{q\ne\ell}(\log q)\left(m_q+\frac{q^{-m_q}-1}{q-1}\right)\\
    &=-2(P+1)\Big(\log(P+1)+\sum_{q\ne\ell}(\log q)\frac{q^{-m_q}-1}{q-1}\Big)\\
    &=O(P\log P).
\end{align*}
By \cref{proposition:Qf_bound} and \cref{proposition:Qf_log_bound} we have \begin{align*}
    O_{\alpha,\infty}(\varphi_\infty),J_{\alpha,\infty}(\varphi_\infty)\ll_{f, b}P^{-b-\frac12}.
\end{align*}
for some $b>\frac12$.
Thus, putting everything together we obtain \begin{align*}
    J_{\mathrm{hyp}}(\psi) \ll_{f,b} (\one_{\ell^{a-1}\mid P-1} + \ell\one_{\ell^a\mid P-1})(a(\log \ell)\ell^{a-1}P^{-b+\frac12}\log P).
\end{align*}
Note in particular, \begin{align*}
    \sum_{\substack{P/\ell^{2a}\in E\\\text{$P$ prime}}}J_{\mathrm{hyp}}(\psi)\log P\ll_{f, b, E} a^2(\log \ell)^2\ell^{2a}.
\end{align*}

\subsubsection{Eisenstein Contribution}\label{section:eisenstein_maass1}
We have \begin{align*}
    J_{\mathrm{Eis}}(\psi) = J_{\mathrm{Eis, 1}}(\psi)+J_{\mathrm{Eis, 2}}(\psi),
\end{align*}
with \begin{align*}
    J_{\mathrm{Eis, 1}}(\psi)&=\frac{1}{4\pi}\sum_{\chi}\int_{-i\infty}^{i\infty} \tr\big(M(\eta_{\chi, s})^{-1}M'(\eta_{\chi, s})\pi_{\eta_{\chi, s}}(\psi)\big)\mathrm ds,\\
    J_{\mathrm{Eis, 2}}(\psi)&=-\frac1{4}\sum_{\chi^2=1}\tr(M(\chi,\chi)\pi_{(\chi,\chi)}(\psi)).
\end{align*}
Here $\eta_{\chi, s} = (\chi\abs{\cdot}^{s},\chi^{-1}\abs{\cdot}^{-s})$ and $\pi_{\eta_{\chi, s}}$ is the induced representation $\Ind_{B(\A)}^{\GL_2(\A)}\eta_{\chi, s}$ and $M(\eta_{\chi, s})$ is an intertwining operator $\pi_{\eta_{\chi, s}}\to\pi_{\eta_{\chi^{-1}, -s}}$, defined in more detail in \cite{GJFGL2APV}. The derivative $M'(\eta_{\chi,s})$ is defined by identifying the underlying space of $\pi_{\eta_{\chi, s}}$ with that of $\pi_{\eta_{\chi, 0}}$, i.e.\ via flat sections in the usual way, see again \cite{GJFGL2APV}. The sum in $J_{\mathrm{Eis, 1}}$ (resp.\ $J_{\mathrm{Eis, 2}}$) runs through finite order (resp.\ order one or two) Hecke characters. The intertwining operator can be normalized by $M(\eta_{\chi, s}) = m(\eta_{\chi, s})R(\eta_{\chi, s}) = \bigotimes_v m_v(\eta_{\chi_v, s})R_v(\eta_{\chi, s})$, where \begin{align*}
    m(\eta_{\chi, s}) = \frac{\Lambda(2s, \chi^2)}{\Lambda(1+2s,\chi^2)},\\
    m_v(\eta_{\chi_v, s}) = \frac{L_v(2s, \chi_v^2)}{L_v(1+2s,\chi_v^2)}.
\end{align*}
Here $L_v$ is the local Euler factor and $\Lambda$ is the completed $L$-function. Again we refer to \cite{GJFGL2APV} and \cite{GelASTF} for more details. We may then write $J_{\mathrm{Eis, 1}}$ as \begin{align*}
    J_{\mathrm{Eis, 1}}(\psi) = &\frac{1}{4\pi}\sum_{\chi}\int_{-i\infty}^{i\infty} \frac{m'(\eta_{\chi, s})}{m(\eta_{\chi, s})}\prod_v\tr(\pi_{\eta_{\chi_v, s}}(\psi_v))\mathrm ds \\
    &+ \sum_{u}\frac{1}{4\pi}\sum_{\chi}\int_{-i\infty}^{i\infty} \tr\big(R_u(\eta_{\chi_u, s})^{-1}R_u'(\eta_{\chi_u, s})\pi_{\eta_{\chi_u, s}}(\psi_u)\big)\prod_{v\ne u}\tr(\pi_{\eta_{\chi_v, s}}(\psi_v))\mathrm ds.
\end{align*}
We look at the first term. We have \begin{align*}
    \prod_v\tr(\pi_{\eta_{\chi_v, s}}(\psi_v)) &= \begin{cases}
        h_f(-is)\sqrt P(\chi_P(P)P^{-s}+\chi_P(P)^{-1}P^{s})&\text{if $c(\chi)^2=\ell^n$},\\
        0&\text{otherwise.}
    \end{cases}
\end{align*}
In particular, it is $0$ if $n$ is odd. If $n=2a$ is even, we will bound the prime average: \begin{align*}
    &\sum_{\substack{P\in \ell^{2a}E\\\text{$P$ prime}}}\log P\sum_{\chi}\int_{-i\infty}^{i\infty} \frac{m'(\eta_{\chi, s})}{m(\eta_{\chi, s})}\prod_v\tr(\pi_{\eta_{\chi_v, s}}(\psi_v))\mathrm ds\\
    &=\sum_{\chi,\,c(\chi)=\ell^a}\sum_{\substack{P\in \ell^{2a}E\\\text{$P$ prime}}}\log P\int_{-i\infty}^{i\infty} \frac{m'(\eta_{\chi, s})}{m(\eta_{\chi, s})}h_f(-is)\sqrt P(\chi_P(P)P^{-s}+\chi_P(P)^{-1}P^{s})\mathrm ds.
\end{align*}
\begin{lemma}\label{lemma:Eis1_error_chi_sum}
    Under GRH, for $s\in i\R$ we have the bound \begin{align*}
    \abs{\sum_{\substack{P\in \ell^{2a}E\\\text{$P$ prime}}}\log P \sqrt P(\chi_P(P)P^{-s}+\chi_P(P)^{-1}P^{s})}\ll_{E}\ell^{2a}(1+\abs s)a^2\log^2(\ell),
\end{align*}
\end{lemma}
\begin{proof}
    GRH gives the bound \begin{align*}
        \sum_{P\leq x}\chi(P)\log P \ll x^{1/2}\log^2(\ell^a x).
    \end{align*}
    The lemma then follows by partial summation.
\end{proof}
By the functional equation we have \begin{align*}
        \frac{m'(\eta_{\chi, s})}{m(\eta_{\chi, s})} &= 2\frac{\Lambda'}{\Lambda}(2s,\chi^2)-2\frac{\Lambda'}{\Lambda}(1+2s,\chi^2)\\
        &=-2\frac{\Lambda'}{\Lambda}(1-2s,\chi^{-2})-2\frac{\Lambda'}{\Lambda}(1+2s,\chi^2).
    \end{align*}
    Let $q'$ be the conductor of $\chi^2$. Then we have\begin{align*}
        \frac{\Lambda'}{\Lambda}(z,\chi^2) = \frac12\log q'+\frac{\gamma'}{\gamma}(z,\chi^2)+\frac{L'}{L}(z,\chi^2).
    \end{align*}
    Here $\frac{\gamma'}{\gamma}(z,\chi^2)$ is a certain expression in the Gamma function, which by Stirling is bounded by $\log(q'(\abs s+2))$ \cite[p. 103]{IwaKowAnaNT}. By \cite[Proposition 5.7]{IwaKowAnaNT} we have \begin{align*}
        \abs{\frac{L'}{L}(1-2iy,\chi^{-2}) + \frac{L'}{L}(1+2iy, \chi^2)} \ll \sum_{\substack{\rho=\beta+i\gamma\\\abs{1+2iy-\rho}<1}}\frac{2(1-\beta)}{(1-\beta)^2+(2y-\gamma)^2} + O(\log(q'(\abs {y}+2))),
    \end{align*}
    where the sum is over non-trivial zeros $\rho$ of $L(s,\chi^2)$.
    
Let $W(y) = \abs{h_f(y)}(1+\abs y)\ll_f (1+\abs y)^{-1-\epsi}$. We have \begin{align*}
    \int_{\{\abs{1+2iy-\rho}<1\}}W(y)\frac{2(1-\beta)}{(1-\beta)^2+(2y-\gamma)^2}\mathrm dy&\ll\sup_{\abs{2y-\gamma}\leq1}W(y).
\end{align*} 
Therefore,
\begin{align*}
    &\int_{-i\infty}^{i\infty}\abs{\frac{m'(\eta_{\chi, s})}{m(\eta_{\chi, s})}h_f(-is)}(1+\abs s)\mathrm ds\\
    &\ll_f 1+\log(q')+\sum_{T\in\Z}\#\{\rho=\beta+i\gamma:\gamma\in [T, T+1]\}\sup_{\abs{2y-T}\leq2}W(y)\\
    &\ll_f 1+\log(q')+\sum_{T\in\Z}\log(q'(\abs T+2))(1+\abs T)^{-1-\epsi}\\
    &\ll_f 1+\log q'.
\end{align*}
Here we used \cite[Proposition 5.7 (a)]{IwaKowAnaNT} to bound the number of zeros. Putting things together we obtain, noting $q'\leq q = \ell^a$:
\begin{align*}
     &\abs{\sum_{\substack{P\in \ell^{2a}E\\\text{$P$ prime}}}\log P\sum_{\chi}\int_{-i\infty}^{i\infty} \frac{m'(\eta_{\chi, s})}{m(\eta_{\chi, s})}\prod_v\tr(\pi_{\eta_{\chi_v, s}}(\psi_v))\mathrm ds}\\
     &\ll_{E}\ell^{2a}a^2\log^2\ell\sum_{c(\chi)=\ell^a}\int_{-i\infty}^{i\infty}\abs{\frac{m'(\eta_{\chi, s})}{m(\eta_{\chi, s})}h_f(-is)}(1+\abs s)\mathrm ds\\
     &\ll_{E, f}\ell^{3a}a^3\log^3\ell.
\end{align*}
The last part could be simplified using GRH, since we are assuming it in other parts anyway.

Now to the remaining terms in $J_{\mathrm{Eis, 1}}$. Let \[S_u = \sum_{\chi}\int_{-i\infty}^{i\infty} \tr\big(R_u(\eta_{\chi_u, s})^{-1}R_u'(\eta_{\chi_u, s})\pi_{\eta_{\chi_u, s}}(\psi_u)\big)\prod_{v\ne u}\tr(\pi_{\eta_{\chi_v, s}}(\psi_v))\mathrm ds.\] 
If $u\ne\ell$, then $\pi_{\eta_{\chi_u, s}}$ has its range in the one-dimensional space of $K_u$ fixed vectors and $R_u$ is constant on this, so $S_u=0$.

For the remaining term $u=\ell$ we need the following lemma on how $R_u$ acts:
\begin{lemma}\label{lemma:local_intertwiner_en}
    Let $p$ be a prime and let $\pi$ be an infinite-dimensional irreducible admissible representation of $\PGL_2(\Q_p)$ of conductor exponent $N\leq n$. Let $v_0$ be any non-zero vector in $\pi^{K_0(p^N)}$. For $k\geq0$ define $v_k = \pi(a_k)v_0$ where $a_k=\begin{pmatrix}
        1&0\\
        0& p^{k}
    \end{pmatrix}$. \begin{enumerate}
        \item $v_0,\dots,v_{n-N}$ is a basis for $\pi^{K_0(p^n)}$.
        \item $e_n = \vol(\overline{K_0(p^n)})^{-1}\one_{w_{p^n}K_0(p^n)}$ acts as \begin{align*}
            \pi(e_n)v_{i} = \epsi(\pi)v_{n-N-i},
        \end{align*}
        for $i=0,\dots,n-N$ where $\epsi(\pi)$ is the root number of $\pi$.
        \item Suppose $\chi$ is a unitary ramified character of $\Q_p^\times$ of conductor exponent $c$, and $s\in i\R$. Suppose $\pi=\pi_{\eta_{\chi, s}}$, where $\eta_{\chi,s}=(\chi\abs\cdot^s,\chi^{-1}\abs\cdot^{-s})$ as before, and $v_{0, \chi, s} = f_{\chi, s}$, characterized by $f_{\chi, s}\left(\begin{pmatrix}
            a&b\\0&d
        \end{pmatrix}\gamma_ck\right)=\chi(a)\chi^{-1}(d)$ with $a,d\in\Z_p^\times$, $b\in\Z_p$, $k\in K_0(p^N)$, $\gamma_c=\begin{pmatrix}
            1&0\\p^{c}&1
        \end{pmatrix}$, and $f_{\chi, s}=0$ on $K-B(\Z_p)\gamma_c K_0(p^N)$. Let $v_{k,\chi,s} = \pi(a_k)v_{0,\chi,s}$. Then \begin{align*}
            R_p(\eta_{\chi, s}) v_{k, \chi, s} = C_{\chi}(s)v_{k, \chi^{-1}, -s}
        \end{align*}
        for $0\leq k\leq n-N$, where \begin{align*}
            C_\chi(s) = \chi(p)^{2c}p^{-2cs}\frac{\epsi(1-2s,\chi^{-2},\psi)\epsi(1+s, \chi,\psi)}{\epsi(1-s, \chi^{-1},\psi)}\frac{L_p(1+2s, \chi^2)}{L_p(1-2s, \chi^{-2})}.
        \end{align*} Here $\psi$ is an additive character on $\Q_p$ of level $0$. Moreover, $v_{k,\chi,s} = p^{ks}\wtilde v_{k,\chi,s}$ where $\wtilde v_{k,\chi, s}\vert_K$ is independent of $s$, i.e.\ $\wtilde v_{k,\chi,s}$ is a flat section of the principal series.
    \end{enumerate}
\end{lemma}
Taking this for granted at the moment, we may now bound $S_\ell$. For $0\leq k\leq m-N$, we compute $R_\ell'(\eta_{\chi_\ell, s})v_{k,\chi_\ell,s}$ via the flat section $\wtilde v_{k,\chi,s}$:
\begin{align*}
    R_\ell(\eta_{\chi_\ell, s})^{-1}R_\ell'(\eta_{\chi_\ell, s})v_{k,\chi_\ell,s} &= \frac{(C_{\chi_\ell}(s)\ell^{-2ks})'}{C_{\chi_\ell}(s)\ell^{-2ks}} v_{k,\chi_\ell, s}.
\end{align*}
Therefore, if $\pi_{\eta_{\chi_\ell, s}}$ has conductor exponent $N$, then \begin{align*}
     R_\ell(\eta_{\chi_\ell, s})^{-1}R_\ell'(\eta_{\chi_\ell, s})\pi_{\eta_{\chi_\ell, s}}(e_m) v_{k,\chi_\ell,s} = \begin{cases}
        \epsi(\pi_{\eta_{\chi_\ell, s}})\frac{(C_{\chi_\ell}(s)\ell^{-2(m-N-k)s})'}{C_{\chi_\ell}(s)\ell^{-2(m-N-k)s}} v_{m-N-k,\chi_\ell, s}&\text{if $N\leq m$,}\\
        0&\text{if $N>m$.}
     \end{cases}
\end{align*}
Then for the trace we get \begin{align*}
    &\tr\Big(R_\ell(\eta_{\chi_\ell, s})^{-1}R_\ell'(\eta_{\chi_\ell, s})\pi_{\eta_{\chi_\ell, s}}(e_m)\Big) \\
    &= \begin{cases}
        \epsi(\pi_{\eta_{\chi_\ell, s}})\left(\frac{C_{\chi_\ell}'(s)}{C_{\chi_\ell}(s)}-(m-N)\log \ell\right)&\text{if $N\leq m$ and $N\equiv m\mod 2$,}\\
        0&\text{otherwise,}
    \end{cases}
\end{align*}
and \begin{align*}
    \tr\Big(R_\ell(\eta_{\chi_\ell, s})^{-1}R_\ell'(\eta_{\chi_\ell, s})\pi_{\eta_{\chi_\ell, s}}(e_n-e_{n-2})\Big) = \begin{cases}
        \epsi(\pi_{\eta_{\chi_\ell, s}})\frac{C_{\chi_\ell}'(s)}{C_{\chi_\ell}(s)}&\text{if $N = n$,}\\
        -2\epsi(\pi_{\eta_{\chi_\ell, s}})\log \ell&\text{if $N \leq n-2$ and $N\equiv n\mod 2$,}\\
        0&\text{otherwise.}
    \end{cases}
\end{align*}
In particular, we see that also in this case the trace vanishes for $n$ odd as $\pi_{\eta_{\chi_\ell, s}}$ has even conductor exponent. So now assume that $n=2a$ with $a\geq2$. In the case $N=n$, this implies that $\chi_\ell^2$ is still ramified, with the same conductor exponent $a$. Using the formula $\epsi(s, \chi, \psi) = \epsi(\frac12,\chi,\psi)\ell^{-n(\chi)(s-\frac12)}$, we get in this case\[C_{\chi_\ell}(s) = \chi_\ell(\ell)^{2a}\ell^{-2as}\frac{\epsi(\frac12, \chi_\ell^{-2},\psi)\ell^{-a(\frac12-2s)}\epsi(\frac12, \chi_\ell,\psi)\ell^{-a(\frac12+s)}}{\epsi(\frac12, \chi_\ell^{-1},\psi)\ell^{-a(\frac12-s)} } = k_{\chi_\ell}\ell^{-2as},
\]
with a constant $k_{\chi_\ell}$, so \begin{align*}
    \frac{C_{\chi_\ell}'(s)}{C_{\chi_\ell}(s)} = -2a\log\ell.
\end{align*}
Let $A(\chi) = \begin{cases}
    -2a\log\ell &\text{if $c(\chi) = \ell^a$},\\
    -2\log\ell&\text{if $c(\chi)\mid\ell^{a-1}$}.
\end{cases}$ Then we get \begin{align*}
    S_\ell&=\sum_{\chi, c(\chi)\mid\ell^a}\int_{-i\infty}^{i\infty}\epsi(\pi_{\eta_{\chi, s}})A(\chi)h_f(-is)\sqrt P(\chi_P(P)P^{-s} + \chi_P(P)^{-1}P^s)\mathrm ds\\
    &=\sum_{\chi, c(\chi)\mid\ell^a}\int_{-i\infty}^{i\infty}A(\chi)h_f(-is)\sqrt P(\chi_P(P)P^{-s} + \chi_P(P)^{-1}P^s)\mathrm ds.
\end{align*}
The global root number of the principal series representation is $1$, so have we dropped it. We have \begin{align*}
    \sum_{\chi, c(\chi)\mid\ell^a}A(\chi)\chi_P(P) &= -2a\log\ell\sum_{\chi\mod \ell^a}\chi(P) -2(1-a)\log\ell\sum_{\chi\mod\ell^{a-1}}\chi(P).
\end{align*}
Here on the right side we sum over Dirichlet characters. By character orthogonality, we obtain \begin{align*}
    \sum_{\chi, c(\chi)\mid\ell^a}A(\chi)\chi_P(P) = -2a(\log\ell) \varphi(\ell^a)\one_{P\equiv 1\mod\ell^a}-2(1-a)(\log\ell)\varphi(\ell^{a-1})\one_{P\equiv 1\mod\ell^{a-1}}.
\end{align*}
Therefore,
\begin{align*}
    \sum_{\substack{P\in\ell^{2a}E\\\text{$P$ prime}}}\sum_{\chi, c(\chi)\mid\ell^a}\sqrt PA(\chi)\chi_P(P)\log P \ll_E a\ell^{3a}\log\ell,
\end{align*}
which gives \begin{align*}
    \sum_{\substack{P\in\ell^{2a}E\\\text{$P$ prime}}} S_\ell\log P\ll_{E, f}a\ell^{3a}\log\ell,
\end{align*}

Combining this with the first part gives (assuming GRH):
\begin{align*}
    \sum_{\substack{P\in\ell^{2a}E\\\text{$P$ prime}}}J_{\mathrm{Eis, 1}}\ll_{E, f}\ell^{3a}a^3\log^3\ell.
\end{align*}

\begin{proof}[Proof of \cref{lemma:local_intertwiner_en}]~
    \begin{enumerate}
        \item This follows from general newform theory, see \cite[p. 4]{RobSchLNFGSp4}.
        \item By \cite[Theorem 3.2.2]{SchLNF}, $\pi(w_{p^N})v_0 = \epsi(\pi)v_0$. The rest follows from the simple computation \begin{align*}
            \pi(e_n)v_k = \pi\left(\begin{pmatrix}
                0&-1\\p^n&0
            \end{pmatrix}\begin{pmatrix}
        1&0\\
        0& p^{k}
    \end{pmatrix}\right)v_0 = \pi\left(p^k\begin{pmatrix}
        1&0\\0&p^{n-N-k}
    \end{pmatrix}w_{p^N}\right)v_0=\epsi(\pi)v_{n-N-k}.
        \end{align*}
        \item First note that by \cite{SchLNF}, $f_{\chi,s}$ is indeed in $\pi^{K_0(p^N)}$. In \cite[(20)]{SchLNF} the newform is normalized to be $\chi(p)^{-c}p^{cs}f_{\chi,s}$. Thus, \cite[(39)]{SchLNF} gives \begin{align*}
            M_p(\eta_{\chi_p, s})f_{\chi,s} = \chi(p)^{2c}p^{-2cs}\frac{\epsi(1-2s,\chi^{-2},\psi)\epsi(1+s, \chi,\psi)}{\epsi(1-s, \chi^{-1},\psi)}\frac{L_p(2s, \chi^2)}{L_p(1-2s, \chi^{-2})}f_{\chi^{-1},-s},
        \end{align*}
        where $\psi$ is of level $0$. By definition, $M_p(\eta_{\chi, s}) = m_p(\eta_{\chi, s})R_p(\eta_{\chi_p, s})$, with $m_p(\eta_{\chi, s}) = \frac{L_p(2s, \chi^2)}{L_p(1+2s,\chi^2)}$. This gives the scalar $C_{\chi}(s)$. The result for the remaining $v_{k, \chi, s}$ follows since $R_p$ is an intertwining operator. For the last claim, let $0\leq j\leq n-N$ and note that $\gamma_{c+j}a_j=a_j\gamma_c$ and \[\supp(v_{j,\chi,s}\vert_K) = (B(\Q_p)\gamma_c K_0(p^N)a_j^{-1}) \cap K=B(\Z_p)\gamma_{c+j} K_0(p^{n}).\] For $k\in K_0(p^n)$ and $a,d\in\Z_p^\times$, we then have \begin{align*}
            v_{j,\chi, s}\left(\begin{pmatrix}
                a&*\\0&d
            \end{pmatrix}\gamma_{c+j}k\right) &= f_{\chi, s}\left(\begin{pmatrix}
                a&*\\0&d
            \end{pmatrix}\gamma_{c+j}ka_j\right) \\
            &= f_{\chi, s}\left(\begin{pmatrix}
                a&*\\0&dp^j
            \end{pmatrix}\gamma_{c}a_j^{-1}ka_j\right)\\
            &=\chi(p)^{-j}p^{j(s+\frac12)}\chi(a)\chi^{-1}(d),
        \end{align*}
        and the last claim follows.
    \end{enumerate}
\end{proof}

It remains to consider \[J_{\mathrm{Eis, 2}}(\psi)=-\frac1{4}\sum_{\chi^2=1}\tr(M(\chi,\chi)\pi_{(\chi,\chi)}(\psi)).
\]
Note that $M(\chi,\chi):\pi_{(\chi,\chi)}\to \pi_{(\chi,\chi)}$ acts as a scalar since the representation is irreducible. So we only have to consider $\tr(\pi_{(\chi,\chi)}(\psi))$. Note that at the $\ell$-adic place $\pi_{(\chi_\ell,\chi_\ell)}$ has conductor dividing $\ell^2$. It then follows from \cref{lemma:local_intertwiner_en} (2) that this trace is $0$.

\subsection{Proof of \cref{proposition:maass_count}}
We proceed in a similar way as in the preceding section. We again define a test function to count the representations. Let $\varphi = \bigotimes_v\varphi_v:\overline G(\A_\Q) \to\C$, where we define the local factors as follows:
\begin{itemize}
    \item $v = \infty$. Let $\varphi_\infty = f$. Then for a unitary irreducible representation $\pi_\infty$ of $\PGL_2(\R)$, we have \begin{align*}
        \tr\pi_\infty(\varphi_\infty) = \begin{cases}
        h_f(t_{\pi_\infty})&\text{if $\pi_\infty^{\overline K_\infty^+}\ne0$},\\
        0&\text{if $\pi_\infty^{\overline K_\infty^+}=0$}.
        \end{cases}
    \end{align*}
    \item $v=p\ne\infty$. \begin{itemize}
        \item $p=\ell$. Let $e_n = \vol(\overline{K_0(\ell^n)})^{-1}\one_{\overline{K_0(\ell^n)}}\in C_c^\infty(\PGL_2(\Q_\ell))$. Then we have
        \begin{align*}
            \tr\pi_\ell(e_n) = \dim \pi_\ell^{\overline{K_0(\ell^n)}}=\begin{cases}
                n-n(\pi)+1&\text{if $n(\pi)\leq n$},\\
                0 &\text{otherwise},
            \end{cases}
        \end{align*}
        for any infinite-dimensional irreducible admissible representation $\pi_\ell$ of $\PGL_2(\Q_\ell)$. This follows from \cite{CasAL}. Thus, to obtain exactly the conductor $\ell^n$ forms we will eventually consider $e_{n}-2e_{n-1}+e_{n-2}$.
        \item $p\ne\ell$. In this case define $\varphi_p = \one_{\overline{K_p}}$. Then for any infinite-dimensional irreducible representation $\pi_p$ of $\overline G(\Q_p)$:
            \begin{align*}
        \tr\pi_p(\varphi_p) = \begin{cases}
            1&\text{if $\pi_p$ is unramified},\\
            0&\text{otherwise.}
        \end{cases}
        \end{align*}
    \end{itemize}
\end{itemize}

To emphasize the $n$-dependence, we will again also write $\varphi^{(n)} = \varphi$. Let $\psi = \varphi^{(n)}-2\varphi^{(n-1)}+\varphi^{(n-2)}$. Then, putting the above together, we obtain \begin{align*}
    \tr\big( R(\psi)\mid L_\cusp^2(\overline G(\A_\Q))\big) = \sum_{\phi\in H_\maass^\new(\ell^n)}h_f(t_\phi).
\end{align*}

We again apply the Arthur-Selberg trace formula from \cite{GJFGL2APV} and obtain
\begin{align*}
    J_\cusp(\psi):=\tr\big( R(\psi)\mid L_\cusp^2(\overline G(\A_\Q))\big)=J_\id(\psi)-J_{\mathrm{one}}(\psi)+J_{\mathrm{ell}}(\psi)+J_{\mathrm{par}}(\psi)+J_{\mathrm{hyp}}(\psi)+J_{\mathrm{Eis}}(\psi).
\end{align*}
We show that the identity term gives the main asymptotic and the remaining terms are $\ll_f n(\log\ell)\ell^{n/2}$.

\subsubsection{The Identity Contribution}
In this case, the identity contribution is the main term. To compute $\psi(1)$, note that $\vol(\overline{K_0(\ell^n)})=\frac{1}{\ell^{n-1}(\ell+1)}$. With our measure normalization we have $\vol(\overline G(\Q)\lquot\overline G(\A_\Q))=\vol(\overline G(\Q)\lquot\overline G(\A_\Q) / \overline{K}_\infty^+ \overline{K}^\infty)=\vol(\PSL_2(\Z)\lquot\H)=\frac\pi3$, so 
\begin{align*}
    J_\id(\psi) &= \vol(\overline G(\Q)\lquot\overline G(\A_\Q))\psi(1)\\
    &=\frac\pi3f(1)\left({\ell^{n-1}(\ell+1)}-2{\ell^{n-2}(\ell+1)}+{\ell^{n-3}(\ell+1)}\right)\\
    &=\frac\pi3 f(1)\ell^{n-3}(\ell+1)(\ell-1)^2.
\end{align*}

\subsubsection{The One-Dimensional Contribution}
Here \begin{align*}
    J_{\mathrm{one}}(\psi) = \sum_{\chi^2=1}\tr\chi(\psi).
\end{align*}
This is $0$ for the same reason as before.

\subsubsection{Elliptic Contribution}
We compute $J_{\mathrm{ell}}$ exactly as in \cref{section:maass_elliptic_1}. The nonarchimedean orbital integral is again as in \cite{KniLiTHO}, except that elliptic classes of positive discriminant now also contribute. We obtain \begin{align*}
    J_{\mathrm{ell}}(\varphi^{(m)}) = \sum_{\substack{t\in\Z\\ t^2-4\notin\Q^{\times2}, t\ne\pm2,\\ L:= \Q(\sqrt{t^2 -4})}}\sqrt{\abs{D_L}}L(1,\psi_{D_L})O_{\gamma_{t,1},\infty}(f)\sum_{\substack{\Z[\gamma_{t, 1}]\subset\O\subset\O_{L}}}[\O_L^\times : \O^\times]\frac{h(\O)}{h(\O_L)}A_\ell(\O)B_\ell(\O).
\end{align*}
Here $\gamma_{t,1}$ denotes a solution to $x^2-tx+1=0$ and $D_L$ denotes the discriminant of $L$. Let $M_\O=[\O:\Z[\gamma_{t,1}]]$. Then $A_\ell(\O)$ is the number of $c\in (\Z/\ell^m\Z)^\times$ that lift to some $\wtilde c\in \Z/\ell^m\gcd(\ell^m, M_\O)\Z$ satisfying \[\wtilde c^2-t\wtilde c+1\equiv0\mod \ell^m\gcd(\ell^m,  M_\O),\]
and $B_\ell(\O) = \frac{\psi_\ell(\ell^m)}{\psi_\ell(\ell^m / (\ell^m, M_\O))}$, where $\psi_\ell(\ell^k) = \ell^k(1+\ell^{-1})$ for $k\geq1$ and $\psi_\ell(1)=1$. The term $A_\ell(\O)B_\ell(\O)$ is from \cite[Proposition 26.38]{KniLiTHO}.

\begin{lemma}\label{lemma:ell_ell_orbital_integral_bound}
    In the above setting, let $s = v_\ell(M_\O)$. Then \begin{align*}
        A_\ell(\O)B_\ell(\O)\ll\ell^{m/2+s/2}.
    \end{align*}
\end{lemma}
We assume this for the moment and bound the elliptic contribution. For fixed $t$, let $t^2-4 = D_Lf^2$. Then note that \begin{align*}
    &\abs{\sqrt{\abs{D_L}}L(1,\psi_{D_L})\sum_{\substack{\Z[\gamma_{t, 1}]\subset\O\subset\O_{L}}}[\O_L^\times : \O^\times]\frac{h(\O)}{h(\O_L)}A_\ell(\O)B_\ell(\O)}\\
    &\leq\ell^{m/2}\abs{\sqrt{\abs{D_L}}L(1,\psi_{D_L})\sum_{\substack{\Z[\gamma_{t, 1}]\subset\O\subset\O_{L}}}[\O_L^\times : \O^\times]\frac{h(\O)}{h(\O_L)}\ell^{v_\ell([\O:\Z[\gamma_{t,1}]])/2}}\\
    &=\ell^{m/2}\abs{\sqrt{\abs{D_L}}L(1,\psi_{D_L})\sum_{k\mid f}k\ell^{v_\ell(f/k)/2}\prod_{p\mid k}\left(1-\legendre{d_L}{p}\frac1p\right)}\\
    &=\ell^{m/2}\Bigg\vert\sqrt{\abs{D_L}}L(1,\psi_{D_L})\prod_{p\mid f, p\nmid \ell}\left(1 + (p - \legendre{d_L}{p})\frac{p^{v_p(f)}-1}{p-1}\right)\\
    &\hspace{10.33em}\cdot\prod_{p\mid (f,\ell)}p^{v_\ell(f)/2}\left(1+\frac{p^{(v_\ell(f)+1)/2}-p^{1/2}}{p^{1/2}-1}\left(1-\legendre{d_L}{p}\frac1p\right)\right)\Bigg\vert\\
    &\ll\ell^{m/2}\abs{\sqrt{\abs{D_L}}L(1,\psi_{D_L})\prod_{p\mid f}4p^{v_p(f)}}\\
    &\ll\ell^{m/2}\abs{\sqrt{\abs{D_L}}L(1,\psi_{D_L})4^{\omega(f)}f}\\
    &\ll_\epsi\ell^{m/2}\abs{D_L}^{\frac12+\epsi}f^{1+\epsi}\\
    &\ll_\epsi \ell^{m/2}(1+\abs t)^{1+\epsi}.
\end{align*}
Here $\omega(f)$ denotes the prime divisor counting function and we use the well-known bounds $2^{\omega(f)}\leq \sigma_0(f)\ll_\epsi f^\epsi$ and $L(1,\psi_D)\ll\log\abs D$.

Next the orbital integral at $\infty$. First note that there are only finitely many $t$ such that $t^2-4<0$ in which case we bound trivially $O_{\gamma_{t,1},\infty}(f)\ll_f1$. For $\abs{t}>2$ we compute as before in \cref{lemma:arch_hyperbolic_orbital_integral}, \begin{align*}
    O_{\gamma_{t,1},\infty}(f) = \frac{Q_f(t^2-4)}{\sqrt{t^2-4}}.
\end{align*}
Hence, by \cref{proposition:Qf_bound} if we choose $\epsi$ small enough:\begin{align*}
    J_{\mathrm{ell}}(\psi)\ll_\epsi \ell^{n/2}\left(\sum_{\substack{t\in\Z\\\abs t\geq3}}\frac{\abs{Q_f(t^2-4)}}{\sqrt{t^2-4}}(1+\abs t)^{1+\epsi} + O_f(1)\right)\ll_{f,\epsi}\ell^{n/2}.
\end{align*}

\begin{proof}[Proof of \cref{lemma:ell_ell_orbital_integral_bound}]
    Since $\ell$ is odd, $A_\ell(\O)$ is the number of $x\in\Z/\ell^m\Z$ that lift to some $\wtilde x\in\Z/\ell^m\gcd(\ell^m,M_\O)\Z$ satisfying \begin{align*}
        \wtilde x^2\equiv t^2-4\mod \ell^m\gcd(\ell^m,  M_\O).
    \end{align*}
    Also note that $\gcd(\ell^m,  M_\O) = \ell^{r}$ with $r = \min\{m, s\}\leq m$. Let $D_t = t^2-4$. We now distinguish cases.\begin{itemize}
        \item Case $D_t\equiv0\mod\ell^{m+r}$. Then we must have $\wtilde x\equiv0\mod \ell^{\lceil(m+r)/2\rceil}$ and this gives $\ell^{m-\lceil(m+r)/2\rceil}$ many choices for $x$. Therefore, $A_\ell(\O)\leq\ell^{(m-r)/2}$.
        \item Case $D_t\not\equiv0\mod\ell^{m+r}$. Then either there is no solution to the congruence or $D_t = \ell^{2b}u$ with $u$ a square in $(\Z/\ell\Z)^\times$. Then solutions $\wtilde x$ are of the form $\ell^bv$ with $v^2\equiv u\mod \ell^{m+r-2b}$. Note that modulo $\ell^{m+r-2b}$ there are only two such $v$. Fixing two such $v_1,v_2$, all the solutions modulo $\ell^{m+r}$ are $\ell^{b}(v_i+\ell^{m+r-2b}w)=\ell^bv_i+\ell^{m+r-b}w$ with $w\in \Z/\ell^{b}\Z$. There are $2\ell^{\max\{b-r,0\}}$ many images of these in $\Z/\ell^m\Z$. Hence, $A_\ell(\O)= 2\ell^{\max\{b-r,0\}}\leq2\ell^{(m-r)/2}$.
    \end{itemize}
    Therefore, $B_\ell(\O)A_\ell(\O) \ll \frac{\psi_\ell(\ell^m)}{\psi_\ell(\ell^{m-r})}\ell^{(m-r)/2}\ll\ell^{(m+r)/2}\leq\ell^{(m+s)/2}.$
\end{proof}

\subsubsection{Parabolic Contribution}
For $x\in\A_\Q$ let \begin{align*}
    F(x) &= \int_{K}\psi\left(k^{-1}\begin{pmatrix}
        1&x\\0&1
    \end{pmatrix}k\right)\mathrm dk \\
    &= \prod_{v} \int_{K_v}\psi_v\left(k_v^{-1}\begin{pmatrix}
        1&x_v\\0&1
    \end{pmatrix}k_v\right)\mathrm dk_v\\
    &=:\prod_{v}F_v(x_v).
\end{align*}
So if $Z_F(s) = \int_{\A^\times}F(x)\abs x^s\mathrm dx$, then $Z_F(s)=\prod_v Z_{F_v}(s)$ where $Z_{F_v}(s) = \int_{\Q_v^\times}F_v(t_v)\abs{t_v}^s\mathrm d{t_v}$. Let $\Lambda(s)$ denote the completed Riemann zeta function and $\theta(s) = \Lambda(s)^{-1}Z_F(s)$ as in \cite[p. 242]{GJFGL2APV}. Then $\theta$ is holomorphic at $1$ and \begin{align*}
    J_{par}(\psi) = \lambda_{-1}\theta'(1)+\lambda_0\theta(1),
\end{align*}
where $\Lambda(s) = \frac{\lambda_{-1}}{s-1}+\lambda_0+\dots$. We may write \begin{align*}
    \theta(1) = \prod_{v}L_v(1)^{-1}Z_{F_v}(1) = L_\infty(1)^{-1}Z_{F_\infty}(1)L_\ell^{-1}Z_{F_\ell}(1),
\end{align*}
where $L_v(s)$ denotes the local Euler factor of $\Lambda$. The derivative is \begin{align*}
    \theta'(1) = \sum_{u}\prod_{v\ne u}L_v(1)^{-1}Z_{F_v}(1)\frac{\mathrm d}{\mathrm ds}\Big\vert_{s=1}\Big(L_u(s)^{-1}Z_{F_u}(s)\Big).
\end{align*}
First note that all terms with $u\ne\infty, \ell$ vanish since $Z_{F_u}=L_u$ at those unramified places. So it remains to bound $Z_{F_\ell}(1)$ and $Z_{F_\ell}'(1)$. Recall that $\psi_\ell$ is a linear combination of the projectors $e_m$. If $A$ runs over a set of coset representatives for $K_\ell/K_0(\ell^m)$, then we have \begin{align*}
    &\int_{\Q_\ell^\times}\int_{K_\ell}e_m\left(k^{-1}\begin{pmatrix}
        1&t\\0&1
    \end{pmatrix}k\right)\mathrm dk (\log\abs t)\abs t\mathrm dt\\
    &=\sum_A\int_{\Q_\ell^\times}\one_{Z(\Q_\ell)K_0(\ell^m)}\left(A^{-1}\begin{pmatrix}
        1&t\\0&1
    \end{pmatrix}A\right)(\log\abs t)\abs t\mathrm dt.
\end{align*}
Note that here $\mathrm dt$ always denotes Haar measure on the respective group, in this case multiplicative. From \cite[Lemma 13.1]{KniLiTHO} we see that we may take the representatives $\begin{pmatrix}
    \delta&1\\1&0
\end{pmatrix}$ with $\delta\in\Z/\ell^m\Z$ and $\begin{pmatrix}
    1&0\\\tau&1
\end{pmatrix}$ with $\tau\in\ell\Z/\ell^m\Z$. In the former case we obtain \begin{align*}
    \int_{\Q_\ell^\times}\one_{Z(\Q_\ell)K_0(\ell^m)}\begin{pmatrix}
    1&0\\t&1
\end{pmatrix}(\log\abs t)\abs t\mathrm dt&=\int_{\Q_\ell^\times}\one_{\ell^m\Z_\ell}(t)(\log\abs t)\abs t\mathrm dt\\
&=\sum_{k=m}^\infty\log(\ell^{-k})\ell^{-k}\\
&=-(\log\ell)\frac{\ell^{1-m}(m\ell-m+1)}{(\ell-1)^2}.
\end{align*}
Thus, the contribution of these cosets is $O(m\log\ell)$.

Next for $A = \begin{pmatrix}
    1&0\\\tau&1
\end{pmatrix}$ with $\tau\in\ell\Z/\ell^m\Z$ we obtain \begin{align*}
    \int_{\Q_\ell^\times}\one_{Z(\Q_\ell)K_0(\ell^m)}\begin{pmatrix}
    1+t\tau&t\\-\tau^2 t&1-\tau t
\end{pmatrix}(\log\abs t)\abs t\mathrm dt.
\end{align*}
We have $\begin{pmatrix}
    1+t\tau&t\\-\tau^2 t&1-\tau t
\end{pmatrix}\in Z(\Q_\ell)K_0(\ell^m)$ if and only if $t\in\Z_\ell$ and $t\tau^2\in\ell^m\Z_\ell$. Let $r = v_\ell(\tau)$. We have $r=\infty$ for $\tau=0$ and $1\leq r\leq m-1$ otherwise. Then the integrand is non-zero exactly when $v_\ell(t)\geq M_r:=\max(0, m-2r)$, hence \begin{align*}
    \int_{\Q_\ell^\times}\one_{Z(\Q_\ell)K_0(\ell^m)}\begin{pmatrix}
    1+t\tau&t\\-\tau^2 t&1-\tau t
\end{pmatrix}(\log\abs t)\abs t\mathrm dt&=\int_{v_\ell(t)\geq M_r}(\log\abs t)\abs t\mathrm dt\\
&=\sum_{k=M_r}^\infty\log(\ell^{-k})\ell^{-k}\\
&=-(\log\ell)\frac{\ell^{1-M_r}((\ell-1)M_r+1)}{(\ell-1)^2}\\
&\ll(\log\ell)(M_r+1)\ell^{-M_r}.
\end{align*}

The number of $\tau\in\ell\Z/\ell^m\Z$ with $1\leq v_\ell(\tau)=r\leq m-1$ is $(\ell-1)\ell^{m-r-1}$, hence the contribution of these cosets is \begin{align*}
    &\ll\sum_{r=1}^{m-1}(\ell-1)\ell^{m-r-1}(\log\ell)(M_r+1)\ell^{-M_r} + O(\log \ell).
\end{align*}
The sum over $r<m/2$ is $\ll m(\log\ell)\ell^{m/2}$ and the sum over $r\geq m/2$ is also $\ll m(\log\ell)\ell^{m/2}$. To summarize,
\begin{align*}
    \abs{Z_{F_\ell}'(1)}\ll n(\log\ell)\ell^{n/2}.
\end{align*}
Similarly, one bounds \begin{align*}
    \abs{Z_{F_\ell}(1)}\ll\ell^{n/2}.
\end{align*}
This can also be deduced from \cite[Proposition 25.7]{KniLiTHO}.

Combining these bounds shows that \begin{align*}
    J_{\mathrm{par}}(\psi)\ll_{f}n(\log\ell)\ell^{n/2}.
\end{align*}

\subsubsection{Hyperbolic Contribution}
Recall that \begin{align*}
    J_{\mathrm{hyp}}(\psi) = -\frac12\vol(\Q^\times\lquot \A_\Q^1)\int_K\int_{N(\A_\Q)}\sum_{\alpha\in\Q^\times, \,\alpha\ne 1}\psi\left(k^{-1}n^{-1}\begin{pmatrix}
        \alpha&0\\0&1
    \end{pmatrix}nk\right)\log H(wnk)\mathrm dn\mathrm dk.
\end{align*}
We claim that the integrand is identically $0$. Indeed, the support condition of the nonarchimedean components of $\psi$ force $\alpha=\pm1$, the archimedean component gives $\alpha>0$, hence $\alpha=1$ which is excluded in the sum.

\subsubsection{Eisenstein Contribution}
Again we write \begin{align*}
    J_{\mathrm{Eis}}(\psi) = J_{\mathrm{Eis, 1}}(\psi)+J_{\mathrm{Eis, 2}}(\psi),
\end{align*}
with \begin{align*}
    J_{\mathrm{Eis, 1}}(\psi)&=\frac{1}{4\pi}\sum_{\chi}\int_{-i\infty}^{i\infty} \frac{m'(\eta_{\chi, s})}{m(\eta_{\chi, s})}\prod_v\tr(\pi_{\eta_{\chi_v, s}}(\psi_v))\mathrm ds \\
    &\hspace{2em}+ \sum_{u}\frac{1}{4\pi}\sum_{\chi}\int_{-i\infty}^{i\infty} \tr\big(R_u(\eta_{\chi_u, s})^{-1}R_u'(\eta_{\chi_u, s})\pi_{\eta_{\chi_u, s}}(\psi_u)\big)\prod_{v\ne u}\tr(\pi_{\eta_{\chi_v, s}}(\psi_v))\mathrm ds\\
    J_{\mathrm{Eis, 2}}(\psi)&=-\frac1{4}\sum_{\chi^2=1}\tr(M(\chi,\chi)\pi_{(\chi,\chi)}(\psi)).
\end{align*}
The term $J_{\mathrm{Eis, 2}}(\psi)$ vanishes for the same reason as before in \cref{section:eisenstein_maass1}, so let us focus on the first. For the first term in $J_{\mathrm{Eis, 1}}(\psi)$, we have \begin{align*}
    \sum_{\chi}\int_{-i\infty}^{i\infty} \frac{m'(\eta_{\chi, s})}{m(\eta_{\chi, s})}\prod_v\tr(\pi_{\eta_{\chi_v, s}}(\psi_v))\mathrm ds &= \sum_{\chi, c(\chi)^2=\ell^n}\int_{-i\infty}^{i\infty} \frac{m'(\eta_{\chi, s})}{m(\eta_{\chi, s})}h_f(-is)\mathrm ds.
\end{align*}
Note that this is $0$ if $n$ is odd. If $n$ is even, we use the same bounds following \cref{lemma:Eis1_error_chi_sum} to see that this is 
\begin{align*}
    \ll_f\sum_{\chi, c(\chi)^2=\ell^n}\log\ell^{n/2} = O(n(\log\ell)\ell^{n/2}).
\end{align*}

For the second term in $J_{\mathrm{Eis, 1}}(\psi)$ we again have that all the terms with $u\ne\ell$ vanish. Denote by $N=N_{\chi}$ the conductor exponent of $\pi_{\eta_{\chi_\ell, s}}$. Since $\pi_{\eta_{\chi_\ell, s}}(e_m)$ is just the projection onto $\pi_{\eta_{\chi_\ell, s}}^{K_0(\ell^m)}$, we see using \cref{lemma:local_intertwiner_en} that if $N\leq m$, then: \begin{align*}
    \tr\big(R_\ell(\eta_{\chi_\ell, s})^{-1}R_\ell'(\eta_{\chi_\ell, s})\pi_{\eta_{\chi_\ell, s}}(e_m)\big)&=\sum_{k=0}^{m-N}\frac{(C_{\chi_\ell}(s)\ell^{-2ks})'}{C_{\chi_\ell}(s)\ell^{-2ks}} \\
    &=(m-N+1)\frac{C_{\chi_\ell}'(s)}{C_{\chi_\ell}(s)} -2(\log\ell)\sum_{k=0}^{m-N}k\\
    &=(m-N+1)\frac{C_{\chi_\ell}'(s)}{C_{\chi_\ell}(s)} -(\log\ell)(m-N)(m-N+1).
\end{align*}
The trace is $0$ if $N>m$.
Therefore, for $\psi_\ell = e_n-2e_{n-1}+e_{n-2}$, we get\begin{align*}
    \tr\big(R_\ell(\eta_{\chi_\ell, s})^{-1}R_\ell'(\eta_{\chi_\ell, s})\pi_{\eta_{\chi_\ell, s}}(\psi_\ell)\big)=\begin{cases}
        0&\text{if $N> n$},\\
        \frac{C_{\chi_\ell}'(s)}{C_{\chi_\ell}(s)}&\text{if $N = n$,}\\
        -2\log\ell&\text{if $N \leq n-1$,}
    \end{cases}
\end{align*}
As before we have \begin{align*}
    \abs{\frac{C_{\chi_\ell}'(s)}{C_{\chi_\ell}(s)}} \ll n\log\ell.
\end{align*}
Therefore, \begin{align*}
    &\abs{\sum_{\chi}\int_{-i\infty}^{i\infty} \tr\big(R_\ell(\eta_{\chi_\ell, s})^{-1}R_\ell'(\eta_{\chi_\ell, s})\pi_{\eta_{\chi_\ell, s}}(\psi_\ell)\big)\prod_{v\ne \ell}\tr(\pi_{\eta_{\chi_v, s}}(\psi_v))\mathrm ds}\\
    &\ll n\log\ell\sum_{\chi,\,c(\chi)\mid \ell^{\lfloor n/2\rfloor}}\int_{-i\infty}^{i\infty} \abs{h_f(-is)}\mathrm ds\\
    &\ll_f n(\log\ell)\ell^{n/2}.
\end{align*}
\printbibliography

\end{document}